\documentclass[11pt]{article}
\usepackage{a4wide,xcolor,eucal,enumerate,mathrsfs}
\usepackage[normalem]{ulem}
\usepackage{pdfsync}
\usepackage[hidelinks]{hyperref}
\usepackage{amsmath,amssymb,epsfig,amsthm}%bbm,
\usepackage[latin1]{inputenc}
\usepackage{cite}

\newtheorem{theorem}{Theorem}[section]

\newtheorem{corollary}[theorem]{Corollary}
\newtheorem{lemma}[theorem]{Lemma}
\newtheorem{proposition}[theorem]{Proposition}
\newtheorem{definition}[theorem]{Definition}

\newtheorem{remark}[theorem]{Remark}
\newcommand{\E}{\mathbb{E}}

\newcommand{\R}{\mathbb{R}}
\newcommand{\cA}{{\ensuremath{\mathcal A}}}

\newcommand{\cF}{{\ensuremath{\mathcal F}}}

\newcommand{\cL}{{\ensuremath{\mathcal L}}}

\newcommand{\cR}{{\ensuremath{\mathcal R}}}

\newcommand{\mm}{{\mbox{\boldmath$m$}}}

\newcommand{\eeta}{{\mbox{\boldmath$\eta$}}}

\newcommand{\sfP}{{\sf P}}

\newcommand{\sfT}{{\sf T}}

\newcommand{\restr}[1]{\lower3pt\hbox{$|_{#1}$}}

\newcommand{\nchi}{{\raise.3ex\hbox{$\chi$}}}
\def\qed{\ifmmode % if math mode, assume display: omit penalty etc.
  \else \leavevmode\unskip\penalty9999 \hbox{}\nobreak\hfill
  \fi               
    \qquad           \hbox{\hskip.5em $\square$
                \hskip.1em}}

\renewcommand{\mm}{\mathfrak m} % Misura di riferimento
\newcommand{\bra}[1]{\left( #1 \right)}
\newcommand{\sqa}[1]{\left[ #1 \right]}
\newcommand{\cur}[1]{\left\{ #1 \right\}}

\newcommand{\abs}[1]{\left| #1 \right|}
\newcommand{\nor}[1]{\left\| #1 \right\|}

\newcommand{\vphi}{\varphi}

\renewcommand{\div}{\mathop{\rm div}\nolimits}        %divergence

\newcommand{\veps}{\varepsilon}

\newcommand{\Algebra}{{\mathscr A}}
\newcommand{\scrM}{\mathscr M}
\newcommand{\scrP}{\mathscr P}
\newcommand{\scrL}{\mathscr L}

\newcommand{\pinfty}{+ \infty}

\renewcommand{\P}{\mathbb{P}}

 \newcommand{\sfDelta}{\mathsf{\Delta}}

\title{Well-posedness of Multidimensional Diffusion Processes \\ with Weakly Differentiable Coefficients}
\author{Dario Trevisan\thanks{Universit\`a degli Studi di Pisa, \textsf{dario.trevisan@unipi.it}}  
}

\begin{document}
\maketitle
%\footnote{confronto con alessio, controesempio ellittico dim =1, confronto lebris-lions, notazione}

\begin{abstract}
We investigate well-posedness for martingale solutions of stochastic differential equations, under low regularity assumptions on their coefficients, widely extending the results first obtained by A.\ Figalli in \cite{figalli-sdes}. Our main results are a very general equivalence between different descriptions for multidimensional diffusion processes,  such as  Fokker-Planck equations and martingale problems, under minimal regularity and integrability assumptions, and new existence and uniqueness results for diffusions having weakly differentiable coefficients, by means of energy estimates and commutator inequalities. Our approach relies upon techniques recently developed jointly with L.\ Ambrosio in \cite{ambrosio-trevisan}, to address well-posedness for ordinary differential equations in metric measure spaces: in particular, we employ in a systematic way new representations and inequalities for commutators between smoothing operators and diffusion generators.
\end{abstract}

\section{Introduction}

Aim of this article is to study well-posedness (i.e., existence, uniqueness and stability) for martingale solutions of stochastic differential equations 
\begin{equation}\label{eq:intro-sde} d X_t = b(t,X_t) dt + \sigma(t,X_t) dW_t, \quad t \in (0,T), \end{equation}
providing in particular new results, under low regularity assumptions on the coefficients $b: (0,T)\times \R^d \to \R^d$, $\sigma: (0,T)\times \R^d \to \R^{d\times d}$.

The classical subject of martingale problems dates back at least to \cite{stroock-varadhan}, where it was first shown that continuous and uniformly elliptic covariances $a = \sigma \sigma^*$'s allow for uniqueness results which have no counterpart in the usual (It\^o-)Cauchy-Lipschitz theory, provided that the solution to \eqref{eq:intro-sde} is understood in a sufficiently weak sense. Since then, the theory has been growing, due to its robustness and strong connections with the theory of semigroups and parabolic PDE's, also in abstract (metric) frameworks, see e.g.\ \cite{ethier-kurtz}. 

Our primary goal here is to show that the techniques originally developed in \cite{ambrosio-trevisan} can be extended to the stochastic theory as well as specialized to the Euclidean setting, to extend in a systematic way the results established in the seminal paper \cite{figalli-sdes}. Actually, most of such techniques, tailored to study well-posedness problems for ordinary differential equations in metric measure spaces (possibly infinite-dimensional) are also well-suited also to the study of diffusions in metric measure spaces, as developed in the author's PhD dissertation \cite{trevisan-phd}. However, in this paper, we deal uniquely  with Euclidean spaces: among various motivations,  besides that a wider audience could be mainly interested in this setting, this allows us to compare new results and techniques with alternative approaches. Finally, Euclidean spaces are a useful ``intermediate'' step for the infinite dimensional theory, e.g.\ by cylindrical approximations; the theory developed here is also instrumental to the  developments in \cite{trevisan-diff-mms}.

Therefore, in this article, we adopt the same point of view as in \cite{figalli-sdes}, where precise connections between well-posedness of PDE's and martingale problems are settled, in particular for a wide class of diffusion having not necessarily continuous nor elliptic coefficients, provided that some Sobolev regularity holds. Of course, well-posedness has to be understood ``in average'' with respect to $\scrL^d$-a.e.\ initial condition (here and below, $\scrL^d$ is Lebesgue measure on $\R^d$). More precisely, a formalization akin to that of DiPerna-Lions (see e.g.\ \cite{ambrosio-crippa-edinburgh} for an account of the deterministic theory) is introduced, the main objects  being Stochastic Lagrangian Flows, i.e.,\ Borel families $(\eeta(x))_{x \in \R^d}$ of probability measures on $C([0,T];\R^d)$, such that
\begin{enumerate}[(i)]
\item $\eeta(x)$ solves \eqref{eq:intro-sde}, starting from $x$ at $t =0$, for $\scrL^d$-a.e.\ $x \in \R^d$;
\item the push-forward measures $(e_t)_{\sharp} \int \eeta(x)\,  d\scrL^d(x)$, where $e_t$ is the evaluation map at $t \in [0,T]$, are absolutely continuous with respect to $\scrL^d$, with uniformly bounded densities.
\end{enumerate}
Let us stress the fact that, as in the deterministic theory, uniqueness is understood for flows, thus in a selection sense: we are not claiming well-posedness for $\scrL^d$-a.e.\ initial datum. Moreover, we remark that, although the conditions above might read as perfect analogues of the notion of Regular Lagrangian flows \cite[Definition 13]{ambrosio-crippa-edinburgh}, Stochastic Lagrangian Flows are not necessarily (neither expected to be) deterministic maps of the initial point only; this is evident when $\sigma = 0$ above and any probability concentrated on possibly non-unique solutions to the ODE give rise to a solution to the martingale problem. Despite this discrepancy, such a theory provides rather efficient tools to study stochastic differential equations under low regularity assumptions, in Euclidean spaces, and, together with \cite{lebris-lions}, which deals with analogous issues from a PDE point of view, has become the starting point for further developments, among which we quote \cite{rockner-zhang-uniqueness-fpe, luo-fpe-sobolev-bv, fang-luo-thalmaier-sde-sobolev, zhang-degenerate}.

Before we proceed with a more detailed description of our results and techniques, let us stress the fact that we are concerned uniquely with martingale problems, so we do not address nor compare our results with those obtained for strong solutions of equations under low regularity assumptions on the coefficients (see the seminal paper \cite{veretennikov} and \cite{krylov-rockner, daprato-flandoli-priola-rockner} for more recent results). Rigorous correspondences between martingale (or weak) and strong solutions may be provided by the classical Yamada-Watanabe theorem \cite{yamada-watanabe} (and extensions, see e.g.\ \cite{kurtz-yw}). Moreover, the literature on Fokker-Planck equations for general measures is so vast that we must limit ourselves to a comparison of our results only with those which are strongly related and look similar in techniques and mathematical contents: this is done in Section \ref{sec:statements}.

We proceed with a brief description of our contributions developed below, which can be split into two parts, roughly corresponding to Section \ref{chap:equivalent-descriptions-rd} (toghether with Appendix \ref{chap:superposition-rd}) and Section \ref{sec:well-posedness-diperna-lions-figalli}.

In the first part, we investigate the problem of abstract equivalence between ``Eulerian'' and ``Lagrangian'' descriptions for multidimensional diffusion processes, where by the former we mean by Fokker-Planck equations and the latter consists of solutions to martingale problems. Although such a correspondence can not be considered novel and many ideas can be traced back at least to \cite{ambrosio-bv} in the theory of ODE's and DiPerna-Lions flows, as well as \cite{kurtz-stockbridge} for c\`adl\`ag martingale problems, to our knowledge, here we provide for the first time general results, under somewhat minimal integrability assumptions on coefficients as well as on solutions. Moreover, we choose to state and prove our results in such a way that they can be translated with a minimal effort to the case of general metric measure spaces, that we address in \cite{trevisan-diff-mms}.

In this part, the crucial result is Theorem \ref{thm:sp}, which provides a so-called ``superposition principle'', i.e., a (non-canonical) way to lift any probability-valued solution of a Fokker-Planck equation to some solution of the corresponding martingale problem. Here, ``to lift'' means that the $1$-marginals of the process which solve the martingale problem coincide with the given solution of the Fokker-Planck equation. Results in a similar spirit appear quite often in the literature (see also the comments just below the statement of Theorem \ref{thm:sp}) and could be traced back to L.C.\ Young's theory of generalized curves. Technically, one could start from already known results such as \cite[Theorem 2.6]{figalli-sdes} or \cite[Theorem 4.9.17]{kurtz-stockbridge} to provide a slightly shorter proof, but we preferred to postpone an almost self-contained derivation in Appendix \ref{chap:superposition-rd}: indeed, even if we rely on the results quoted above, it turns out that one has to settle non-trivial technical problems. In particular, an underlying result is Theorem \ref{theorem:basic-burkholder}, where we establish an estimate for the modulus of continuity of solutions to martingale problems under somewhat minimal integrability assumptions (based on a refined L\'evy-type estimate); an alternative but less effective approach, based on fractional Sobolev spaces, was developed in \cite{trevisan-phd}. Finally, we point out that we exploit a technique originally developed in \cite[Theorem 7.1]{ambrosio-trevisan}, in case of cylindrical approximations, to move from bounded coefficients to possibly unbounded ones.

In the second part, we address the problem of well-posedness for Fokker-Planck equations, providing sufficient conditions assuming Sobolev regularity of the coefficients. We mainly focus on uniqueness issues, which are settled by means of energy or $L^2$ estimates, formally satisfied by any weak solution,  under suitable bounds on the divergence of the driving coefficients: such an approach could be hardly considered novel, as it was already present in  \cite{diperna-lions}, for transport equations. However, our main contribution consists in a novel and systematic approach to the estimate of the error terms arising in the approximation procedure, to obtain so-called commutator inequalities: see Section \ref{sec:commutators} for a brief account of the method as well as complete proofs of our crucial resuts. It turns out that, essentially by means of the same technique, we are able to deal with Sobolev derivations (Lemma \ref{lemma:commutator-estimate}), Sobolev diffusions (Lemma \ref{lem:commutator-diffusion-rd}) as well as with time-dependent elliptic diffusions (Lemma \ref{lemma:commutator-time}). Such a technique, which ultimately consists in choosing a Markov semigroup as a smoothing operator and relying on duality arguments as well as an interpolation \emph{\`a la} Bakry-Em\'ery,  has also the advantage of being completely ``Eulerian'' and ``coordinate free''. Let us point out also in this case that it was first developed in \cite{ambrosio-trevisan} to deal with an analogue problem for derivations in metric measure spaces. 

In conclusion, we state and prove two well-posedness results: Theorem \ref{thm:wp-degenerate}, for diffusions \eqref{eq:intro-sde} having possibly degenerate coefficients, assuming first order Sobolev regularity for the drift $b$ and second order Sobolev regularity for the infinitesimal covariance $a = \sigma \sigma^*$ (together with uniform bounds on their divergence); Theorem \ref{thm:wp-degenerate}, for the bounded elliptic case, i.e.\ $ \lambda \abs{v}^2 \le a (v, v)  \le \Lambda \abs{v}^2$ for every $v \in \R^d$, with $t \mapsto a_t$ Lipschitz, where (roughly speaking) regularity assumptions can be reduced of one order (i.e., no assumption on $b$, and first order Sobolev regularity for $a$). We regard such results as chief examples of the strength and versatility of our techniques for commutator estimates, and we point out that other interesting results could arise in different situations, such as perturbations of elliptic generators which enjoy some ultra- (or hyper-) contractivity features, as well as the case of BV-regular coefficients, that we do not address here.

\smallskip
\noindent {\bf Acknowledgments.}
The author has been partially supported by PRIN10-11 grant from MIUR
for the project Calculus of Variations and is a member of the GNAMPA group of the Istituto 
Nazionale di Alta Matematica (INdAM).

The author thanks his PhD advisors L.\ Ambrosio and M.\ Pratelli, for many discussions before and during the writing of this paper, as well as the thesis referees D.\ Bakry and M.\ R\"ockner for their useful comments and constructive criticism, which were taken into great account also while developing this work, in particular with respect to comparison with existing literature.

\section{Diffusion processes and their equivalent descriptions}\label{chap:equivalent-descriptions-rd}

In this section, we study abstract correspondences between ``Eulerian'' and ``Lagrangian'' descriptions for multidimensional diffusion processes, in particular with respect to well-posedness results. The main ideas involved are not entirely novel, but they widely extend those from \cite{figalli-sdes}: 
here we obtain results under minimal regularity and integrability assumptions. As already remarked in the introduction, on a technical side, a crucial tool is the superposition principle for diffusions, Theorem \ref{thm:sp}, whose proof is deferred to Appendix~\ref{chap:superposition-rd}. 

In Section \ref{sec:definitions-superposition}, we introduce diffusion operators in $\R^d$, Fokker-Planck equations, martingale problems and flows; in Section \ref{sec:correspondence-rn} we study their equivalences.

\subsection{Definitions and basic facts} \label{sec:definitions-superposition}

Throughout, we use the following notation, for $v$, $w \in \R^d$ ($d\ge 1$) and $A$, $B \in \R^{d \times d}$, 
\[  v \cdot w = \sum_{i=1}^d v^i w^i, \quad \abs{v}^2 = v\cdot v,  \quad  (v \otimes w)^{i,j} := v^i w^j, \quad \text{for $i,j \in \cur{1,\ldots d}$,}\]
\[ A: B = \sum_{i,j=1}^d A^{i,j} B^{i,j},\quad  \abs{A}^2 = A:A,  \quad A(v,w) = A : (v \otimes w),\]
and the following notation for differential calculus on $(0,T)\times \R^d$ ($T \in [0,\infty)$):
\[ f_t(\cdot) = f(t, \cdot),\,\,  \partial_t f = \frac{\partial f}{\partial t }, \,\,  \partial_i f =  \frac{\partial f}{\partial x^i}, \, \, \partial_{i,j} f = \frac{\partial^2 f}{\partial x^i\partial x^j}, \,\,  \text{for $t \in [0,T]$, $i,j \in \cur{1,\ldots, d}$,}\]
\[ \nabla f = (\partial_i f)_{i=1}^d, \, \nabla^2 f = (\partial_{i,j} f)_{i,j=1}^d, \, \, \text{thus}\, \,  b \cdot \nabla f = \sum_{i=1}^d b^i \partial_i f \, \text{and } \, a: \nabla^2 f = \sum_{i,j=1}^d a^{i,j} \partial_{i,j} f,\]
as well as the notation $\scrL^d$ for Lebesgue measure on $\cR^d$ and $\nabla^*$ for the distributional adjoint of $\nabla$ (i.e., $\nabla^*b= -\div b$ on vector fields). 

We write $\scrM(\R^d)$ for the space of signed (real-valued) Borel measures on $\R^d$ (with finite total variation), $\scrM^+(\R^d) \subseteq \scrM(\R^d)$ for the cone of finite non-negative measures and $\scrP (\R^d) \subseteq \scrM^+(\R^d)$ for the convex set of Borel probability measures on $\R^d$. We say that a curve $\nu = (\nu_t)_{t \in (0,T)} \subseteq \scrM(\R^d)$ is Borel if, for every Borel set $A \subseteq \R^d$, the curve $t \mapsto \nu_t(A)$ is Borel; we let $\abs{\nu} = (\abs{\nu_t})_{t \in (0,T)}$ be the curve of total variation measures. A curve $\nu = (\nu_t)_{t \in (0,T)} \subseteq \scrP(\R^d)$ is narrowly continuous if, for every $f \in C_b(\R^d)$,  $t \mapsto \int f d\eta_t$ is continuous.

Most of the quantities that we consider below are integrated with respect to the variable $t \in (0,T)$, with respect to $\scrL^1|_{[0,T]}$: when $\nu = (\nu_t)_{t \in (0,T)} \subseteq \scrM (\R^d)$ is a Borel curve, we write  $\abs{\nu} dt$ for the Borel measure on $(0,T) \times \R^d$,  $A \mapsto \abs{\nu}(A) = \int_0^T \abs{\nu_t}(A_t) dt$, for $A \subseteq (0,T)\times \R^d$ Borel.
For $p$, $q \in [1,\infty]$ and a Borel curve $\nu = (\abs{\nu_t})_{t \in (0,T)} \subseteq \scrM^+ (\R^d)$, the space $L^p_t L^q_x(\nu)$ is naturally defined and endowed with the Banach norm
\[  \nor{ f}_{L^p_tL^q_x(\nu) } := \nor{ \nor{ f(t, x) }_{L^q_x(\nu_t) } }_{L^p_t(dt) }  < \infty,\]

On the space $C([0,T]; \R^d)$ (naturally endowed with the $\sup$ norm and its Borel $\sigma$-algebra), we let $e_t: \gamma \mapsto \gamma_t := \gamma(t)\in \R^d$ be the evaluation map at $t \in [0,T]$. The natural filtration on $C([0,T]; \R^d)$ is the increasing family of $\sigma$-algebras $\cF = (\cF_t)_{t \in [0,T]}$, with $\cF_t := \sigma( e_s: s \in [0,t])$. Given $\eeta \in \scrP ( C([0,T]; \R^d))$, we always let $\eta_t := (e_t)_\sharp \eeta$ be the $1$-marginal law at $t \in [0,T]$. Notice that the family $\eta := (\eta_t)_{t \in [0,T]} \subseteq \scrP(\R^d)$ is narrowly continuous.

We let throughout $\Algebra = C^{1,2}_b((0,T)\times\R^d))$ (respectively, $\Algebra_c = C^{1,2}_c((0,T)\times\R^d))$) be the space of uniformly bounded (respectively, compactly supported) and continuously differentiable functions, once with respect to $t \in (0,T)$ and twice with respect to $x \in \R^d$, with uniformly bounded derivatives (as usual, the superscript $(1,2)$ counts the number of derivatives with respect to $(t,x)$, other superscripts may appear, with natural meaning). We prefer the ``abstract'' notation $\Algebra$ and large parts of the theory can be developed when ``test'' functions are replaced by other classes (e.g.\ as developed throughout the monograph \cite{kurtz-stockbridge}). We endow $\Algebra$ with the norm
\begin{equation*}%\label{eq:norm-c12} 
\nor{f}_{C^{1,2}_{t,x}} = \sup_{(t,x) \in (0,T)\times\R^d} \cur{ |f(t,x)|+ |\partial_t f (t,x)| + |\nabla f (t,x)| + | \nabla^2 f(t,x)| }.\end{equation*}
Notice that, by uniform continuity, any $f \in \Algebra$ extends to $[0,T]\times \R^d$.

Throughout, we always let
\begin{equation}\label{eq:a-b-rn} a = (a^{i,j})_{i,j = 1}^d: (0,T)\times\R^d \to \operatorname{Sym}_+(\R^d), \quad b =(b^i)_{i =1}^d: (0,T)\times\R^d \to \R^d,\end{equation}
be Borel, where $\operatorname{Sym}_+(\R^d)$ is the space of symmetric, non-negative definite $d \times d$ matrices.

We define diffusion operators in $\R^d$, measure-valued weak solutions to Fokker-Planck equations and martingale problems on the interval $[0,T]$. Most of these notions are classical (for a brief historical account, see e.g.\ the introduction of \cite{stroock-varadhan}): for the sake of clarity we provide the definitions and prove some simple facts.

\begin{definition}[diffusion operator]
We let  $\cL (= \cL(a,b))$ be the linear differential operator
\[  \Algebra \ni f \quad \mapsto  \quad \cL f = \frac 1 2 a: \nabla^2f + b \cdot \nabla f,  \]
%\sum_{i,j =1}^d  a^{i,j}(t,x) \frac{\partial^2 f}{\partial x^i \partial x^j}(t,x) + \sum_{i=1}^d b^{i}(t,x) \frac{\partial f}{\partial x^i} (t,x)
with values in Borel maps on $(0,T)\times \R^d$.
\end{definition}

We write $\cL_t f := (\cL f)_t$, for $t \in (0,T)$. As usual, the coefficients $b$, $a$ are referred  as the drift of $\cL$ and the infinitesimal covariance of $\cL$. If $a = 0$, then $\cL$ reduces to a linear first-order operator, i.e.\ a derivation, and we say that we are in the deterministic case.

%\begin{equation*}\int \bra{\abs{a}^p + \abs{b}^p } d\abs{\nu} \bra{ := \int_0^T\int  \bra{\abs{a_t}^p + \abs{b_t}^p }d\abs{\nu_t}  dt } < \infty.\end{equation*}
%Similarly, we say that $a$, $b \in L^p_{loc}(\abs{\nu})$ if, for every compact set $B \subseteq \R^d$, it holds
%\begin{equation}\label{eq:locally-lp}
%\[ \int_{(0,T)\times B} \bra{ \abs{a}^p + \abs{b}^p}  d\abs{\nu} < \infty.
%\end{equation}
%However, we say that $\cL$ is bounded if it holds
%\begin{equation}
%\label{eq:bounded-coefficients}
 %\int_0^T \sup_{x\in \R^d} \cur{\abs{a_t} (x) + \abs{b_t}(x)} dt <\infty,\end{equation}
%and locally bounded if, for bounded set $B \subseteq \R^d$, it holds
%\begin{equation}\label{eq:locally-bounded-coefficients} \int_0^T\sup_{x\in B} \cur{\abs{a_t} (x) + \abs{b_t}(x)} dt <\infty.\end{equation}

Given a diffusion operator $\cL$, we let the ``Eulerian'' description of evolution of particles ``driven'' by $\cL$ consist of weak solutions of Fokker-Planck (or forward Kolmogorov) equations, in duality with $\Algebra$. Although our main interest lies in solutions to FPE's that are narrowly continuous curves of probability measures, we introduce more general measure valued solutions, as they are useful, e.g.\  the space of solutions becomes linear.

\begin{definition}[weak solutions of FPE's]%\label{def:weak-fpe}
A Borel curve $\nu = (\nu_t)_{t \in (0,T)} \subseteq \scrM(\R^d)$ is a weak solution of the Fokker-Planck equation (FPE)
\begin{equation}\label{eq:fpe-rn}\partial_t \nu_t = \cL^*_t \nu_t, \quad \text{on $(0,T) \times \R^d$,} \end{equation}
if it holds
\begin{equation}\label{eq:fpe-integrability} \int_0^T \int \bra{ \abs{a_t} + \abs{b_t}} d |\nu_t| dt < \infty\end{equation}
and, for every $f \in \Algebra_c$, it holds
\begin{equation}\label{eq:weak-fpe-rn} \int_0^T \int \sqa{\partial_t f (t,x)+ \cL f  (t,x) } d\nu_t (x) \,  dt = 0.
\end{equation}
\end{definition}

With the notation introduced above, condition \eqref{eq:fpe-integrability} can be restated as $a$, $b \in L^1_{t,x}(\abs{\nu})$. In what follows, we frequently omit to specify  the operator $\cL$, that we regard as fixed.%We remark some elementary properties of weak solutions, whose simple proofs can be found at the beginning of \cite[\S 8.1]{ambrosio-gigli-savare-book}, in the deterministic case.

% slight modifications of the argument in this chapter would allow to deal with general (non-negative) measure valued solutions \footnote{and even $\sigma$-finite solutions?}.

\begin{remark}%[narrowly continuous representative and extension of the weak formulation to $\Algebra$]
\label{rem:extension-weak-formulation}
A density argument akin to \cite[Lemma 8.1.2]{ambrosio-gigli-savare-book} allows for proving that any solution $\nu = (\nu_t)_{t \in (0,T)} \subseteq \scrP(\R^d)$ to \eqref{eq:fpe-rn} admits a unique narrowly continuous representative $\tilde{\nu} =( \tilde{\nu})_{t \in [0,T]}$, with $\nu_t = \tilde{\nu}_t$, for $\scrL^1$-a.e.\ $t \in (0,T)$. Thanks to this fact, we may also say that the solution $\nu$ starts from $\nu_0$ (or that $\nu_0$ is the initial law of $\nu$). Moreover, for every $f \in \Algebra$, it holds
\begin{equation}\label{eq:narrowly-continuous-fp} \int f_{t_2} d\tilde{\nu}_{t_2} - \int f_{t_1} d\tilde{\nu}_{t_1} = \int_{t_1}^{t_2} \int \sqa{ \partial_t f  + \cL_t f }d\nu_t dt, \quad \text{for $t_1$, $t_2 \in [0,T]$, $t_1 \le t_2$,}\end{equation}
\end{remark}

Actually, one first proves that \eqref{eq:narrowly-continuous-fp} holds for $f \in \Algebra_c$ and then extends by density to $f \in \Algebra$. Since this last step  requires the introduction of useful cut-off functions, we sketch it here, for later use. For $R \ge 1$, we fix $\chi_R: \R^d \to [0,1]$, a smooth function with $\chi_R(x) = 1$, for $\abs{x} \le R$, $\chi_R (x) = 0$, for $\abs{x} \ge 2R$, such that $\abs{ \nabla \chi_R } \le 4R^{-1}$ and $\abs{\nabla^2 \chi_R} \le 4 R^{-2}$. Given $f \in \Algebra$, we let $f_R = f \chi_R \in \Algebra_c$, for which we assume that \eqref{eq:narrowly-continuous-fp} holds. The chain rule entails
%\begin{equation}\label{eq:chain-rule-diffusions}
\[\cL_t f_R  = (\cL_t f) \chi_R + f_t \cL_t \chi_R + a_t( \nabla f_t, \nabla \chi_R ), \quad \text{for $t \in (0,T)$,}\]
%\end{equation}
hence the bound
\[ |\cL_t f_R | \le \abs{\cL_t f} + \abs{f_t} \abs{\cL_t \chi_R } +  \abs{a_t} \abs{\nabla f_t} \abs{\nabla \chi_R} \le C \nor{f_t}_{C^2_b} \sqa{ \abs{a_t} + \abs{b_t} }.\]
Letting $R \to \infty$, by dominated convergence, we extend the validity of \eqref{eq:narrowly-continuous-fp}. 

%For our present purposes, solutions of FPE's could be also directly defined via \eqref{eq:narrowly-continuous-fp}, but we prefer the more common distributional approach.

Next, we introduce solutions of the martingale problem, following \cite[Chapter 6]{stroock-varadhan}. In particular, we argue directly on the ``canonical'' space $\Omega = C([0,T]; \R^d)$, endowed with the evaluation process $e_t(\gamma) = \gamma(t)$, $t \in [0,T]$,  and its natural filtration.

\begin{definition}[solution of MP's]%\label{def:martingale}
A probability $\eeta \in  \scrP ( C([0,T]; \R^d))$ is a solution of the martingale problem (MP) (associated to $\cL$) if it holds
\begin{equation}\label{eq:martingale-integrability}   \int \sqa{ \int_0^T \bra{\abs{b_t}\circ e_t + \abs{a_t } \circ e_t} dt} d\eeta  < \infty\end{equation}
and, for every $f \in \Algebra$, the process
\begin{equation}\label{eq:martingale} [0,T] \ni t \mapsto f_t \circ e_t - \int_0^t \sqa{\partial_t f_s + \cL_s f}\circ e_s ds\end{equation}
is a martingale with respect to the natural filtration on $C([0,T]; \R^d)$.
\end{definition}

Recall the notation $\eta_t = (e_t)_\sharp \eeta \in \scrP(\R^d)$, $t \in [0,T]$, thus $\eta_0$ is the initial law of $\eeta$. As for FPE's, we usually omit to specify $\cL$, regarded as fixed.  Let us remark that a density argument shows that it makes no difference to require that \eqref{eq:martingale} is a martingale only for $f \in \Algebra_c$.

For any solution $\eeta$ of the MP, the integrability assumption \eqref{eq:martingale-integrability}, which is equivalent to $a$, $b \in L^1(\eta)$, entails that the process $[0,T] \ni t \mapsto \int_0^t \sqa{\partial_t f_s + \cL_s f}\circ e_s ds$ is well defined, up to a $\eeta$-negligible set, as continuous and  progressively measurable process. In particular, it belongs to $L^\infty_{loc}(\eeta, (\cF_t)_t)$, i.e.\ there exists an increasing sequence of stopping times $\tau_n$, $\eeta$-a.s.\ converging towards $T$, such that $\int_0^{\tau_n} \sqa{\partial_t f_s + \cL_s f}\circ e_s ds \in L^\infty(\eeta)$, for every $n \ge 1$: it is sufficient to let
\[ \tau_n := T \land \inf \cur{ t \in [0,T]: \int_0^t\bra{ \abs{b_s} \circ e_s +  \abs{a_s} \circ e_s} ds \ge n}.\]

We prefer throughout not to enlarge the filtration $\cF$ with $\eeta$ negligible sets. This causes virtually no harm in the exposition, e.g.\ a martingale $M = (M_t)_{t \in [0,T]}$ must be understood in the sense that it holds $\E[ M_t | \cF_s ] = M_s$, $\eeta$-a.s.\ for every $s \le t$ (and the $\eeta$-negligible set could not belong to $\cF_s$).

%\begin{remark}[the deterministic case]\label{rem:deterministic-definition}
When $a=0$, solutions to the MP reduce to probability measures concentrated on absolutely continuous solutions to the ordinary differential equation
\[ \frac{d}{dt}\gamma_t = b_t(\gamma_t), \quad \text{for $\scrL^1$-a.e.\ $t \in (0,T)$.}\]
Indeed, arguing as in \cite[Lemma 3.8]{figalli-sdes}, it turns out that the martingale \eqref{eq:martingale} is constant. More generally, the quadratic variation process of \eqref{eq:martingale} is $t \mapsto \int_0^t a_s(\nabla f_s, \nabla f_s) ds$: this plays a crucial role in estimates for the modulus of continuity of the canonical process, see e.g.\ Corollary \ref{coro:regularity-paths-martingale-general}.
%\end{remark}

%\begin{remark}[solutions to MP's induce solutions to FPE's] \label{rem:from-mp-to-fpe}
By integration of \eqref{eq:martingale} with respect to $\eeta$ (i.e., taking expectation) we deduce that any solution $\eeta$ of the MP induces, by means of its $1$-marginals $(\eta_t)_{t \in (0,T)}$ a narrowly continuous solution of the FPE \eqref{eq:fpe-rn}. 
%\end{remark}
A converse statement is provided by the following theorem, whose proof is deferred in  Appendix \ref{chap:superposition-rd}; in the next section, it  plays a crucial role to connect various well-posedness results for FPE's and MP's. 

\begin{theorem}[superposition principle]\label{thm:sp}
Let $\nu = (\nu_t)_{t \in [0,T]} \subseteq \scrP (\R^d)$ be a narrowly continuous solution of \eqref{eq:fpe-rn}. Then, there exists  $\eeta$ which is a solution to the MP (associated to the same diffusion operator $\cL$) such that, for every $t \in [0,T]$, it holds $\eta_t = \nu_t$.
\end{theorem}

In what follows, we refer to $\eeta$ above as a superposition solution for $\nu$.

We refer to this result as the superposition principle for diffusions: the terminology originates in the deterministic literature of ODE's, see \cite{ambrosio-bv}: the solution $\eeta$ can be non-trivially distributed among the possibly non-unique solutions to the ODE, thus introducing some ``randomness'' in an otherwise deterministic setting; these probability measures are nevertheless superpositions of deterministic paths. In the setting of diffusion operators, solutions are already expected to be random, thus the term is justified only by extension, although it would be interesting, at least in some cases, to be able to distinguish between the two ``sources of randomness'': this would require us to introduce concepts such as strong and weak solutions.

As remarked in the introduction, Theorem \ref{thm:sp} is a quite general result, only the integrability condition \eqref{eq:fpe-integrability} being required, which is some sense minimal to give sense to FPE's and MP's (although one may slightly relax it by dealing with local martingale problems). Our result extends \cite[Theorem 2.6]{figalli-sdes}, where only uniformly bounded coefficients are considered; let us mention that results in a similar spirit -- that of L.C.\ Young's theory of generalized curves -- appear quite often in the literature, e.g.\ Echeverria's theorem \cite[Theorem 4.9.17]{ethier-kurtz} (see  \cite{kurtz-stockbridge} for extensions) in the framework of martingale problems in spaces of c\`adl\`ag paths, or Smirnov's decomposition of 1-currents \cite{smirnov-sp} (see also \cite{pao-ste-12} for an alternative approach, valid also in the case of metric currents). Our strategy of proof extends that of \cite[Theorem 2.6]{figalli-sdes} and should be regarded as a (non-trivial) counterpart of \cite[\S 8.1 and \S 8.2]{ambrosio-gigli-savare-book} in the setting of multi-dimensional diffusions: although rather natural, the derivation is not immediate from the available literature (both from deterministic and stochastic), due to non-trivial technical points. The major difficulty in our proof is to provide estimates for the modulus of continuity of the canonical process (a problem that would appear also if we wanted to deduce it from Echeverria's theorem).

Next, we investigate some stability properties enjoyed by solutions of MP's and FPE's, with respect to suitable operations: their proofs are straightforward, so we omit them.

Clearly, all the definitions above can be given with respect to any interval $[S,T]$ in place of $[0,T]$ (when it is not mentioned, we always refer to the interval $[0,T]$): solutions are then well-behaved with respect to the natural restriction map 
\[  C([0,T]; \R^d) \ni \gamma \mapsto \gamma|_{[S,T]} = (\gamma_t)_{t \in [S,T]} \in  C([S,T]; \R^d).\]

\begin{proposition}\label{prop:restriction}
Let $S$, $T \in \R$, with $0 \le S \le T$, and let $\eeta \in \scrP(C([0, T]; \R^d))$ be a solution of the MP. Let $\rho: C([S,T]; \R^d) \to [0,\infty)$ be a uniformly bounded probability density (with respect to $\eeta$), measurable with respect to $\cF_{S}$.

Then, the push-forward $(\rho \eeta) |_{[S,T]} := \bra{|_{[S,T]}}_\sharp( \rho \eeta ) \in \scrP(C([S,T]; \R^d)$ is a solution to the MP associated to $\cL$ on $[S,T]$.
\end{proposition}

The analogous property for FPE's is obvious: if $(\nu_t)_{t\in (S,T)}$ is a solution of \eqref{eq:fpe-rn}, its restriction $(\nu_t)_{t \in (S,T)}$ is a solution of the FPE on $(S,T) \times \R^d$.

%\begin{proof}
%It is sufficient to fix any $f \in C^{1,2}_c((t_2,T);\R^d)$, let $t \in [t_2, T]$ and $g: C([t_2, T]; \R^d) \to \R$ be any bounded function, measurable with respect to $\sigma(e_r: t_2 \le r \le t)$, %and prove
%\begin{equation}
%\label{eq:key-stability}
%\int \sqa{ f_{T} - \int_{t_2}^{T} (\partial_t   + \cL_s) f_s \, ds} g\,  d \pi_\sharp( \rho \eeta ) = \int \sqa{ f_{t} -  \int_{t_2}^{t} (\partial_t   + \cL_s) f_s \, ds} g\,  d\pi_\sharp( \rho \eeta),\end{equation}
%(here and below, for simplicity of notation, we omit to write $e_s$). The key point is to consider $f$ as a function belonging to $C^{1,2}_c((t_1,T)\times\R^d)$, letting $f_s = 0$ for $s \in (t_1,t_2]$. The assumption on $\eeta$ gives that
%\[ [t_1,T] \ni t \mapsto f_t - \int_{t_1}^t (\partial_t   + \cL_s) f_s \, dr\]
%is a martingale on the space $C([t_1,T];\R^d)$ endowed with the probability $\eeta$ and the natural filtration. Since $f_t = 0$ for $t \in (t_1,t_2]$, it holds, for $t \in (t_2, T)$,
%\[  f_t  - \int_{t_1}^t (\partial_t   + \cL_r) f_r  \, dr = f_t  - \int_{t_2}^t (\partial_t   + \cL_r) f_r  \, dr.\]
%On the other hand, as $(g\circ \pi) \rho$ is $\cF_{t_2}$-measurable, it holds
%\[ \int \sqa{f_{T}  - \int_{t_1}^{T} (\partial_t   + \cL_s) f_s  \, ds} (g \circ \pi)\,  \rho d\eeta = \int \sqa{f_{t} - \int_{t_1}^{t} (\partial_t   + \cL_s) f_s  \, ds}(g \circ \pi)\,  \rho d\eeta.\]
%These two identities entail \eqref{eq:key-stability}.
%\end{proof}

%\end{document}

Solutions of FPE's and MP's are clearly stable with respect to convex combinations, as a consequence of Fubini's theorem. %Again, the correspondent statement for weak solutions to Fokker-Planck equations holds as well and its proof is also straightforward.

\begin{proposition}%\label{prop:integration-mp}
Let $(Z, \cA, \bar{\nu})$ be a probability space and let $(\eeta_z)_{z \in Z} \subseteq \scrP (C[0,T];\R^d)$ be a Borel family, such that, for $\bar \nu$-a.e.\ $z \in Z$, $\eeta_z$ is a solution of the MP (associated to a fixed diffusion operator $\cL$). Moreover, let
\begin{equation}\label{eq:integrability-prop-convex-combination} \int_Z \int_0^T \int \bra{ \abs{b_t} + \abs{a_t}} d\eta_z dt d\bar{\nu}(z) < \infty\end{equation}
hold. Then, $A \mapsto \eeta(A) = \int \eeta_z(A)\, d\bar{\nu}(z)$ is a solution of the MP (associated to $\cL$).
\end{proposition}

A somewhat converse result, for disintegration with respect to the initial law, is a  consequence of stability of martingales under conditional expectations with respect to the $\sigma$-algebra $\cF_0$. % and the fact that it is sufficient to restrict to test functions in a countable, dense set in $\cA$.

\begin{proposition}\label{prop:disintegration-mp}
Let $\eeta$ be a solution of the MP and let $(\eeta(x))_{x \in \R^d}$ be a regular conditional probability for $\eeta$ with respect to $e_0$. Then, for $\eta_0$-a.e.\ $x \in \R^d$, $\eeta(x)$ is a  solution of the MP associated to $\cL$, with initial law $\delta_x$.
\end{proposition}

%\begin{proof}
%The integrability assumption entails that, for every $f \in \Algebra$, $t \in [0,T]$, the function $M_t := f _t \circ e_t - \int_0^t (\partial_t +\cL_s) f \circ e_s \, ds$ belongs to $L^1(\eeta)$. To check the martingale property, it is sufficient to apply Fubini's theorem: for $t \in [0,T]$ and any bounded $\cF_t$-measurable function $g$, it holds
%\[ \int M_T g \, d\eeta= \int  \sqa{ \int M_T g\,  d\eeta_z } d\bar \nu (z) =  \int  \sqa{ \int M_t g\,  d\eeta_z } d\bar \nu( z) = \int M_t g \, d\eeta.\]
%\end{proof}

We conclude this section by introducing a suitable notion of flow associated to a diffusion operator, roughly consisting of Borel families of solutions of the MP, for a (large, in some sense) set of initial conditions in $\R^d$. Our aim is to study flows in the DiPerna-Lions sense (as extended by Figalli to MP's), thus, we  introduce the concept of ``regular flow'', where regularity is usually some growth and/or absolute continuity condition on the $1$-marginals, providing a selection criterion, yielding uniqueness in otherwise ill-posed problems. %When specialized to the deterministic setting, one recovers a notion of ``generalized'' flows, consisting in %to the  point of view lies between the usual notion of flows associated to vector fields and measurable selections of % of probabilities $\bra{\eeta(s,x)}_{(s,x)\in [0,T]\times \R^d} \subseteq C($, where for the martingale problem (see also Remark \ref{rem:flows-smooth}).
To study this notion in sufficient generality, we formulate such regularity conditions in terms of some set $\cR := \cR_{[0,T]}$ of narrowly continuous (probability curves that are) solutions of  \eqref{eq:fpe-rn}, which describe the ``admissible'' class of dynamics.  With this notation, we refer to any $\nu \in \cR$  as a $\cR$-regular solution of \eqref{eq:fpe-rn}, and we say that solution to the MP  is $\cR$-regular if the curve of its $1$-marginals is a $\cR$-regular solution of \eqref{eq:fpe-rn}. We also let $\cR_0 \subseteq \scrP(\R^d)$ be the set of all initial laws of the solutions belonging to $\cR_{[0,T]}$, which we regard as the set of initial distribution of mass that we are allowed to transport.% Then, we introduce the concept of $\cR$-regular martingale flow as follows.

\begin{definition}[$\cR$-MF]%\label{def:mf}
A Borel family $(\eeta(x))_{x \in \R^d} \subseteq \scrP (C([0,T];\R^d))$ is said to be a $\cR$-regular martingale flow ($\cR$-MF) (associated to $\cL$) if the initial law of $\eeta(x)$ is $\delta_x$, for every $x \in \R^d$, and, for every $\bar{\nu} \in \cR_0$, the probability measure $\int \eeta(x) d\bar{\nu}(x)$ is a $\cR$-regular solution to the MP (associated to $\cL$).
\end{definition}

We remark that we are not imposing that, for every $x \in \R^d$, $\eeta(x)$ is a $\cR$-regular solution to the MP associated to $\cL$; the requirement is only \emph{in average}, with respect to every admissible initial density $\bar{\nu} \in \cR_0$. Of course, from this condition and Proposition \ref{prop:disintegration-mp} we obtain that $\eeta(x)$ is a solution of the MP, for $\bar{\nu}$-a.e.\ $x \in \R^d$, for every $\bar{\nu} \in \cR_0$. For example, if we let $\cR_{[0,T]}$ be the set of all narrowly continuous solutions of \eqref{eq:fpe-rn}, then we operate no selection at all, and $\cR$-MF's are  Borel selections $(\eeta(x))_{x \in \R^d}$ of solutions of the MP, with $\eeta(x)$ starting at $\delta_x$, for every $x \in \R^d$. The DiPerna-Lions theory is obtained if we let $\cR$ be the set of all narrowly continuous  solutions $\nu_t = u_t \mathscr{L}^d \scrP(\R^d)$ of  \eqref{eq:fpe-rn} with $\nor{u}_{L^\infty_{t,x}} < \infty$.

We state (without implicitly assuming) some further properties of $\cR$-regular solutions of MP's and FPE's that are useful in the next section. The first property is a stability property with respect to pointwise domination: for every $\tilde{\nu}$, $\nu$, narrowly continuous solution of \eqref{eq:fpe-rn} such that, for some $C \ge 0$, 
\begin{equation}\label{eq:stability-cr-ambrosio} \text{ $\tilde{\nu_t} \le C \nu_t$, for every $t \in [0,T]$ and $\nu \in \cR_{[0,T]}$, then $\tilde{\nu} \in \cR_{[0,T]}$.}\end{equation}
A useful property is stability with respect to convex combinations, i.e., for any $\bar{\nu} \in \scrP(Z)$, %using the notation of Proposition \ref{prop:integration-mp},
\begin{equation}\label{eq:stability-convex-combination}
\text{ if, $\bar{\nu}$-a.e.\ $z \in Z$,  $\eeta_z$ is $\cR$-regular and \eqref{eq:integrability-prop-convex-combination} holds, then $\int \eeta_z d\bar{\nu}(z)$ is $\cR$-regular.}\end{equation}
A reasonable converse should be stability with respect to disintegration, but there are several formulations: given any $\cR$-regular $\eeta$, writing $(\eeta(x))_{x \in \R^d}$ for a regular conditional probability with respect to $e_0$, we may require that
\begin{equation}\label{eq:stability-convex-conditioning-bounded}
\text{ for any $\bar{\nu} \in \scrP(\R^d)$ with $\bar{\nu} \le C \eta_0$ for some $C>0$, then $\int \eeta(x) \bar{\nu} (x)$ is $\cR$-regular,}\end{equation}
or alternatively that
\begin{equation}\label{eq:stability-convex-conditioning}
\text{ for any $\bar{\nu} \in \cR_0$ with $\bar{\nu} \ll \eta_0$, then $\int \eeta(x) \bar{\nu} (x)$ is $\cR$-regular,}\end{equation}
or even that 
\begin{equation}\label{eq:stability-convex-conditioning-strong}\text{ for $\eta_0$-a.e.\ $x \in \R^d$, $\eeta(x)$ is a $\cR$-regular solution to the MP,}\end{equation}
which is a rather strong condition: it formally implies the others whenever \eqref{eq:stability-convex-combination} holds true. % they are related, 
Let us also notice that it does not hold when we deal with the DiPerna-Lions class introduced above, while \eqref{eq:stability-convex-conditioning-bounded} as well as \eqref{eq:stability-convex-conditioning} hold true. Moreover, an application of Theorem \ref{thm:sp} shows that condition \eqref{eq:stability-convex-conditioning-bounded} is equivalent to \eqref{eq:stability-cr-ambrosio}.% An example where condition \eqref{eq:stability-convex-conditioning} may not hold while \eqref{eq:stability-convex-conditioning-bounded} holds is the case % \eqref{eq:stability-convex-combination} holds.

%Finally, the following stability property:
%\begin{equation}\label{eq:stability-cr-ambrosio} \text{ for every $\bar{\nu}^1$, \bar{\nu}^2 \in \cR_0$ $\bar{\nu} \ll \eta_0$

%The correspondent property for $(\cR_{[s,T]})_{s \in [0,T]}$ is similar, with $s$ in place of $0$ (for every $s \in [0,T]$).

Due to technical reasons, we must introduce a slight extension of all the notions above, taking into account a family $(\cR_{[s,T]})_{s \in [0,T]}$, where  each $\cR_{[s,T]}$ consists of narrowly continuous solutions of the FPE associated to $\cL$, on $[s,T]$. Then, we  let $\cR_s$ be the set of all $1$-marginals at time $s$ for solutions belonging to $\cR_{[s,T]}$, and we refer to $\cR$-regular solutions of FPE's and MP's on $[s,T]$, by natural extension of the definitions given on the interval $[0,T]$. We also assume that
\begin{equation}\label{eq:stability-cr-restriction} \text{ for any $r$, $s \in [0,T]$, with $r \le s$, $\nu \in \cR_{[r,T]}$,  then $(\nu_t)_{t \in [s,T]} \in \cR_{[s,T]}$,}\end{equation}
In particular, for any $\nu \in \cR_{[r,T]}$, one has $\nu_s \in \cR_s$. We also accordingly extend the notion of $\cR$-MF by considering a family $(\eeta(s,x))_{s \in [0,T], x \in \R^d}$, where $(\eeta(s,x)_{x \in \R^d}$ is a $\cR$-MF, for every $s \in [0,T]$ (notice that we are not requiring joint measurability of $(s,x) \mapsto \eeta(s,x)$).   

\begin{remark}[Markov property]\label{rem:markov}
With the notation introduced above, we can state the Markov property via Chapman-Kolmogorov equations, for a $\cR$-MF $(\eeta(s,x))_{s \in [0,T], x \in \R^d}$,
\begin{equation}\label{eq:chap-kolmogorov} \eta(s,x)_t = \int \eta(s,y)_t \, \eta(r,x)_s, \quad \text{ $\bar{\nu}$-a.e.\ $x\in \R^d$, for every $\bar{\nu} \in \cR_r$}\end{equation}
for every $r$, $s$, $t \in [0,T]$ with $r \le s \le t$. %(notice that this is reasonable since $\int \eta(s,x)_r(dy) \nu(dx) \in \cR_s$, a consequence of \eqref{eq:stability-cr-restriction}).

%and let-regular flows as families $(\eeta(s,x))_{s \in [0,T], x\in \R^d}$ associated to some family $\cR = (\cR_{[s,T]})_{s \in [0,T]}$, 
We obtain this property as a consequence of uniqueness, arguing e.g.\ as in \cite[Proposition 3.10]{figalli-sdes}. However, let us remark that it could be be of independent interest to study regular flows that are also Markov, extending e.g.\ the approach in \cite[Chapter 12]{stroock-varadhan}. Finally, much less is known about the strong Markov property for DiPerna-Lions flows, i.e., the validity of \eqref{eq:chap-kolmogorov} with stopping times in place of deterministic times -- perhaps one has to introduce some notion of ``regular'' stopping times. % seems to require for the joint map $(s,x) \mapsto \eeta(s,x)$ to be Borel. We say that a flow is strong if such a condition holds: this notion is technical and we prefer not do not address strong flows: throughout this thesis, we restrict the attention to existence and uniqueness issues for martingale flows (once uniqueness is proved, one may investigate whether the flow is strong). \fr% the problem being that existence of a strong martingale flow is harder to settle.
\end{remark}

\subsection{Equivalence between FPE's, MP's and flows}\label{sec:correspondence-rn}

The superposition principle provided by Theorem \ref{thm:sp} allows for establishing a neat correspondence between ``Eulerian'' and ``Lagrangian'' descriptions, transferring well-posedness results both ways. Such a connection is firmly established in the deterministic case, see e.g.\ \cite[\S 4]{ambrosio-crippa-lecture-notes}, and in the stochastic setting has been investigated e.g.\ in \cite[\S 2]{figalli-sdes}, in case of a DiPerna-Lions theory, and in \cite[\S 4]{ethier-kurtz}, for the classical theory (i.e., not in a selection sense). In this section, we provide a complete equivalence between well-posedness results for $\cR$-regular solutions of FPE's and MP's. % in Lemma \ref{lem:transfer-uniqueness-strong}. %provided that the superposition principle holds. We postpone the investigation of its validity, for general diffusions in $\R^d$, in Chapter \ref{chap:superposition-rd}.%; in the next section we investigate the abstract equivalence between the notions above.% its introduction by . %We remark that a crucial role is played by the validity of the superposition principle, for a sufficiently large class of solutions. %At the same time, we believe that this discuss highlights the crucial role % Since the available proofs are only based on convexity (or linearity) and the validity of a superposition principle, nothing 

%We fix throughout all this section two Borel maps $a$, $b$ as in \eqref{eq:a-b-rn} and consider the associated diffusion $\cL$.
%Moreover, we fix a class of measures $\L \subseteq \scrM(C((0,T); \R^d)$ and assume that the superposition principle holds for any weak solution $\nu = (\nu_t) \subseteq \cL$.

{\bf FPE's $\Leftrightarrow$ MP's.}
Equivalence between existence result is straightforward, by lifting any solution $\nu$ of the FPE, we obtain existence of solutions of the MP, so we focus on uniqueness. A simple result which transfers ``uniqueness'' is the following one: the non trivial implication \emph{ii) $\Rightarrow$ i)} follows from lifting two different solutions $\nu^1$, $\nu^2$ (see also \cite[Theorem 2.3]{figalli-sdes}).

\begin{lemma}\label{lem:transfer-weak}
Let $\bar{\nu} \in \cR_0$. Then, the following conditions are equivalent:
\begin{enumerate}[i)]
\item there exists at most one $\cR$-regular solution $\nu$ of \eqref{eq:fpe-rn} with $\nu_0 = \bar{\nu}$.
\item if $\eeta^1$, $\eeta^2$ are $\cR$-regular solutions of the MP with $\eta_0^1 = \eta_0^2 = \bar{\nu}$, then $\eta_t^1 = \eta^2_t$, for $t \in [0,T]$.
\end{enumerate} 
\end{lemma}
%
%\begin{proof}
%\emph{i) $\Rightarrow$ ii)}. It is sufficient to notice that the $1$-marginals provide $\cR$-regular solutions to FPE associated to $\cL$. \emph{ii) $\Rightarrow$ i)}. Given any $\cR$-regular solution $\nu$ to the FPE, the superposition principle provides some solution $\eeta$ to the MP, which is necessarily $\cR$-regular because its $1$-marginals are given by $\nu \in \cR$.
%\end{proof}

A stronger uniqueness result, for processes, can be obtained arguing as in \cite[Theorem 6.2.3]{stroock-varadhan} or \cite[Proposition 5.5]{figalli-sdes}. Let us point out that here there appears a small gap with the deterministic literature, since a different argument \cite[Theorem 9]{ambrosio-crippa-lecture-notes} shows uniqueness for MP's assuming only \eqref{eq:stability-cr-ambrosio}, while we must consider also intermediate $s \in [0,T]$ (since the  argument employed therein uses some conditioning which may not preserve the martingale property in general, but it does when the martingale is deterministic).%\footnote{This gap between the deterministic and stochastic theories certainly requires deeper investigations.}. 

\begin{lemma}[transfer of uniqueness]%\label{lem:transfer-uniqueness-strong}
Let $\cR = (\cR_{[s,T]})_{s\in [0,T]}$ satisfy \eqref{eq:stability-cr-restriction} and \eqref{eq:stability-cr-ambrosio}, with $s$ in place of $0$, for $s \in [0,T]$. Then, the following conditions are equivalent:
\begin{enumerate}[i)]
\item for every $s \in [0,T]$ and $\bar{\nu} \in \cR_s$, there exists at most one $\nu \in \cR_{[s,T]}$ with $\nu_s = \bar{\nu}$.
\item for every $s \in [0,T]$, if $\eeta^1$, $\eeta^2$ are $\cR$-regular solutions of the MP on $[s,T]$, with $\eta_s^1 = \eta_s^2$, then $\eeta^1 = \eeta^2$.
\end{enumerate} 
\end{lemma}

\begin{proof}%[Proof of Lemma \ref{lem:transfer-uniqueness-strong}]
\emph{ii) $\Rightarrow$ i)}. As in Lemma \ref{lem:transfer-weak}, $\nu \in \cR_{[s,T]}$ with $\nu_s = \bar{\nu}$ we consider a ($\cR$-regular) superposition solution $\eeta$: the uniqueness assumption entails that its $1$-marginals are uniquely identified. \emph{i) $\Rightarrow$ ii)}. The proof relies (implicitly) on the Markov property. Let $s \in [0,T]$ and $\eeta^1$, $\eeta^2$ be solutions of the MP on $[s,T]$, with $\eta^1_s = \eta^2_s$. To deduce that $\eeta^1 = \eeta^2$, we show that, for every $n \ge 1$, the $n$-marginals of $\eeta^1$ an $\eeta^2$ coincide, i.e., for any $s \le t_1 < \ldots < t_n \le T$ and $A_1, \ldots, A_n \subseteq \R^d$ Borel, it holds
\begin{equation}\label{eq:induction} \eeta^1( e_{t_1} \in A_1, \ldots, e_{t_n} \in A_n)  = \eeta^2( e_{t_1} \in A_1, \ldots, e_{t_n} \in A_n).\end{equation}

We argue by induction on $n \ge 1$, the case $n = 1$  being a consequence of \emph{i) $\Rightarrow$ ii)} in Lemma \ref{lem:transfer-weak} and property \eqref{eq:stability-cr-restriction},  i.e.\ we use the fact that $(\eta^i_t)_{t\in[s,T]}$ for $i \in \cur{1,2}$ are $\cR$-regular solutions, with $\eta^1_s = \eta^2_s$. To perform the step from $n$ to $n+1$, we argue as follows. For fixed $s \le t_1 < \ldots < t_n < t_{n+1}\le T$ and $A_1, \ldots, A_n, A_{n+1} \subseteq \R^d$ Borel sets, we let
\[ \rho := \frac{ \prod_{i=1}^n \chi_{A_i} (e_{t_i})}{\eeta^1( e_{t_1} \in A_1, \ldots, e_{t_n} \in A_n)} : C([s,T];\R^d) \to [0,\infty),\]
i.e., the density of $\eeta^1$ conditioned with respect to  $\bigcap_{i=1}^n \cur{e_{t_i} \in A_i}$. We assume that the denominator above is not null: otherwise there is nothing to prove. Notice also that the inductive assumption gives $(e_{t_n})_\sharp (\rho \eeta^1) = (e_{t_n})_\sharp (\rho \eeta^2)$, since it amounts to \eqref{eq:induction} with $A_n \cap B$ in place of $A_n$, for every $B \subseteq \R^d$ Borel.
%\[ \eeta^1( e_{t_1} \in A_1, \ldots, e_{t_n} \in (A_n \cap B) )  = \eeta^2( e_{t_1} \in A_1, \ldots, e_{t_n} \in (A_n \cap B)), \quad \text{for all $B\subseteq \R^d$ Borel.}\]
%true as well because of the induction assumption.

%Notice that the $1$-marginals of $(\eeta^j)^{|t_n}$ are smaller than the correspondent for $\eeta^j$ and the inductive assumption gives that the law of $(e_{t_1}, \ldots, e_{t_n})$ is the same w.r.t.\ $\eeta^1$ or $\eeta^2$, thus
%\[ (\eeta^1)^{|t_n}_{t_n} = (e_{t_n})_\sharp (\rho \eeta^1) = (e_{t_n})_\sharp (\rho \eeta^2) = (\eeta^2)^{|t_n}_{t_n}.\]
For $i \in \cur{1,2}$, we let $\eeta^i_\rho$ be the push-forward of the measure $\rho \eeta^i$ with respect to the natural restriction from $[s,T]$ to $[t_n, T]$,  and notice that both are $\cR$-regular solutions of the MP on $[t_n, T]$, with identical laws at $t_n$,
\[(\eta^1_\rho)_{t_n} = (e_{t_n})_\sharp (\rho \eeta^1) = (e_{t_n})_\sharp (\rho \eeta^2)=(\eta^2_\rho)_{t_n},\]
by Lemma \ref{prop:restriction} and \eqref{eq:stability-cr-restriction}. By the implication \emph{i) $\Rightarrow$ ii)} in Lemma \ref{lem:transfer-weak}, we deduce in particular that
$(\eta^1_\rho)_{t_{n+1}} =  (\eta^2_\rho)_{t_{n+1}}$, thus
\[ \frac{ \eeta^1( e_{t_1} \in A_1, \ldots, e_{t_n} \in A_n,  e_{t_{n+1}} \in A_{n+1}) }{ \eeta^1( e_{t_1} \in A_1, \ldots,  e_{t_n} \in A_n)}  = \frac{ \eeta^2( e_{t_1} \in A_1, \ldots, e_{t_n} \in A_n,  e_{t_{n+1}} \in A_{n+1}) }{ \eeta^2( e_{t_1} \in A_1, \ldots, e_{t_n} \in A_n)},\]
hence we deduce the case $n +1$ of \eqref{eq:induction}.
\end{proof}

{\bf MP's $\Leftrightarrow$ flows.} %e investigate the connection between well-posedness for $\cR$-regular martingale problems and flows. 
In this case,  both notions are ``Lagrangian'', thus there  is no need of the superposition principle here: most of the argument are just consequences of convexity and disintegration of measures.% and the main technical difficulty is to provide existence results: this can be achieved under suitable assumptions, e.g.\ via measurable selection theorems or using regular conditional probabilities.

%We introduce the following notation: for $(s,\bar{\nu}) \in [0,T] \times \scrP(\R^d)$, we write $C_{s,\bar{\nu}}(\cL)\subseteq \scrP(C([0,T]; \R^d))$ for the set of solutions $\eeta$ to the martingale problem associated to $\chi_{[0,s]}\cL$, with $\eta_s = \bar \nu$.

Although our actual well-posedness results are in the DiPerna-Lions case, where uniqueness  is understood up to $\mm$-a.e.\ equivalence, where $\mm$ is some ``reference'' $\sigma$-finite Borel measure on $\R^d$ (i.e., $\mm = \scrL^d$),  for the sake of completeness, we provide a  result assuming \eqref{eq:stability-convex-conditioning-strong}.% disintegration \eqref{  in a everywhere sense, which in particular applies when $\cR$ is the class of all.  %It would be also interesting to investigate intermediate results, where well-posedness is understood in a \emph{quasi}-everywhere sense, with respect to some capacity.

%For $\nu \in \scrP(\R^d)$ introduce the notation $\cR_\nu$ for the set of $\cR$

\begin{proposition}%\label{lemma:well-posedness-flows}
%Assume that a $\cR$-MF associated to $\cL$ on $[0,T]$ exists. Then, %Let $\cL = \cL(a,b)$ be a diffusion with $a$, $b$ Borel as in \eqref{eq:a-b-rn}.
% and let $C_{s,x} := C_{s,\delta_x}(\cL)$ is compact for every $(s,x) \in [0,T] \times \R^d$, with
%\[ [0,T] \times \R^d \ni (s,x) \mapsto C_{s,x}  \in \mathscr{K}( \scrP( C([0,T]; \R^d)) )\]
%Borel, where the target space is that of compact sets of $\scrP(C([0,T]\times\R^d))$ endowed with the Hausdorff distance, see e.g.\ \cite[Chapter 12, \S 1]{stroock-varadhan}. Assume also that
%\[ \sup_{x \in \R^d, \eta \in C_{s,x}} \int_s^T \int \abs{\cL_t f } d\eta_t \, dt < \infty,\]
%for every $s \in [0,T]$, $f \in \Algebra = C^{1,2}_c((0,T)\times \R^d)$. 
Consider the following conditions:
\begin{enumerate}[i)]
\item for every $\bar{\nu} \in \cR_0$, there exists a unique $\cR$-regular solution $\eeta^{\bar{\nu}}$ to the MP, with initial law $\bar{\nu}$, and the map $\bar{\nu} \mapsto \eeta^{\bar{\nu}}$ is Borel;
%\item for every $x \in \R^d$, there exists a $\cR$-regular solution $\eeta(x)$ to the MP associated to $\cL$ on $[0,T]$, with initial law $\delta_x$, and the map $x \mapsto \eeta(x)$ is Borel;
\item for every $\bar{\nu} \in \cR_0$ and $\cR$-MF's $(\eeta^1(x))_{x \in \R^d}$, $(\eeta^2(x))_{x \in \R^d}$, one has $\eeta^1 = \eeta^2$, $\bar{\nu}$-a.e.\ on $\R^d$. 
\end{enumerate}
Then, it always holds \emph{i) $\Rightarrow$ ii)}, while \emph{ii) $\Rightarrow$ i)} holds true provided that some $\cR$-MF exists and both \eqref{eq:stability-convex-combination} and \eqref{eq:stability-convex-conditioning-strong} hold.
\end{proposition}

\begin{proof} \emph{i) $\Rightarrow$ ii)} is straightforward, since regular conditional probabilities are essentially unique (a $\cR$-MF is in particular a regular conditional probability of $\int \eeta(x) d\bar{\nu}(x)$ with respect to $e_0$).

To show the implication \emph{ii) $\Rightarrow$ i)}, let $\bar{\nu} \in \cR_0$ and $\eeta^1$, $\eeta^2$ be $\cR$-regular solutions of the MP, with initial law $\bar{\nu}$. By disintegrating with respect to $e_{0}$ and \eqref{eq:stability-convex-conditioning-strong} we may assume that $\bar{\nu} = \delta_{\bar{x}}$, for some $\bar{x} \in \R^d$. Let $(\eeta(x))_{x \in \R^d}$ be a $\cR$-MF (here we use the existence assumption) and define
\[ \eeta^i(x) := \chi_{\cur{ x \neq \bar{x} }} \eeta(x) + \chi_{\cur{ x = \bar{x} }} \eeta^i, \quad \text{ for $i \in \cur{1,2}$,}\]
which are two different $\cR$-MF's since, for any $\bar{\mu} \in \cR_0$, it holds, by \eqref{eq:stability-convex-combination}, $\int \eeta^{i}(x) d\bar{\mu}(x) =  \int_{\cur{x \neq \bar{x}}} \eeta(x) d\bar{\mu}(x) + \bar{\mu}(\bar x) \eeta^i \in \cR$.
%To prove that \eqref{eq:chap-kolmogorov} holds, it is enough to notice that, if we let $\pi: C([s,T];\R^d) \mapsto C([r,T];\R^d)$ be the natural projection, by Proposition \ref{prop:restriction}, then $\pi_{\sharp}(\eeta(s,x)) \in C_{r,\bar{\nu}}$ (to be rigorous, we have to extend it trivially on $[0,r]$), where $\bar{\nu} = \eta(s,x)_r$, by definition. By uniqueness, we deduce
%[ \pi_{\sharp}(\eeta(s,x)) = \int_{\R^d} \eeta(r,y) \, \eta(s,x)_r(dy), \quad  \text{as measures on $C([r,T];\R^d)$,}\]
%which entails  \eqref{eq:chap-kolmogorov}.
\end{proof}

%Thanks to stability with respect to convex combinations, proved in Proposition \ref{prop:integration-mp}, we replace the first statement in Lemma \ref{lemma:well-posedness-flows} with the following one, obtaining the same conclusions.

%An identical proof shows that if
%\[ [0,T] \times \R^d \ni (s,x) \mapsto C_{s,x} \in \mathscr{K}(\scrP( C([0,T]\times\R^d) ))\]
%is Borel, then the unique martingale flow, if it exists, is strong in the sense introduced in Remark \ref{rem:markov}.

The result above is rather unsatisfactory in terms of existence of $\cR$-MF's, which seems a delicate problem, in general. For example, existence may follow if one assumes the validity of assumption \emph{i)}, \eqref{eq:stability-convex-combination}, \eqref{eq:stability-convex-conditioning-strong} and that $\cR_0$ is a  Borel of probability measures. Then, for every $x \in \R^d$ such that $\delta_x \in \cR_0$, there exists a unique $\cR$-regular solution of the MP, and by suitable definition for $x$ not in such a set, we obtain a $\cR$-MF (which is then unique in the sense above). An easier existence result follows if we assume the domination condition
\begin{equation}\label{eq:domination} \text{ for some $\sigma$-finite measure $\mm$, it holds $\bar{\nu} \ll \mm$, for every $\bar{\nu} \in \cR_0$,}\end{equation}
as in the DiPerna-Lions case. We also assume that $\mm$ is minimal in the sense that, for every $A\subseteq \R^d$ Borel with $\mm(A)>0$, there exists some $\bar{\nu} \in \cR_0$ with $\bar{\nu}(A)>0$.

\begin{proposition}
Let \eqref{eq:domination} hold, and consider the following conditions:
\begin{enumerate}[i)]
\item for every $\bar{\nu} \in \cR_0$, there exists a unique $\cR$-regular solution $\eeta^{\bar{\nu}}$ to the MP, with initial law $\bar{\nu}$, and the map $\bar{\nu} \mapsto \eeta^{\bar{\nu}}$ is Borel;
%\item for every $x \in \R^d$, there exists a $\cR$-regular solution $\eeta(x)$ to the MP associated to $\cL$ on $[0,T]$, with initial law $\delta_x$, and the map $x \mapsto \eeta(x)$ is Borel;
\item there exists a  $\cR$-MF's $(\eeta(x))_{x \in \R^d}$ and $\cR$-MF's are $\mm$-a.e.\ unique, i.e.\ if $(\eeta^1(x))_{x \in \R^d}$ and $(\eeta^2(x))_{x \in \R^d}$ are $\cR$-MF's, then $\eeta^1 = \eeta^2$, $\mm$-a.e.\ in $\R^d$.
\end{enumerate}
If \eqref{eq:stability-convex-combination} and \eqref{eq:stability-convex-conditioning-bounded} holds, then \emph{i) $\Rightarrow$ ii)}. If  \eqref{eq:stability-convex-combination} and \eqref{eq:stability-convex-conditioning} are satisfied, then \emph{ii) $\Rightarrow$ i)}.
\end{proposition}

\begin{proof}
\emph{i) $\Rightarrow$ ii)}. We have only to settle existence of some $\cR$-MF, as uniqueness is trivial.  For any probability $\bar{\nu} = {u}\mm \in \cR_0$ we consider the unique $\cR$-regular solution of the MP $\eeta^{{u}}$ with initial law $\bar{\nu}$ and a regular conditional probability with respect to $e_0$, $(\eeta^{{u} }(x))_{x \in \R^d}$. Then, for any ${v}\mm \in \cR_0$, it holds
\begin{equation} \label{eq:identity-mf-wellposedness} \eeta^{{u}}(x) = \eeta^{{v}}(x), \quad \text{$\mm$-a.e.\ $x \in X$ such that ${u} (x) >0$ and ${v} (x) >0$.}\end{equation}
Indeed, it is sufficient to show that, for every $\veps >0$ and every $\rho \mm$ probability density, concentrated on $\cur{ {u} > \veps, {v}  > \veps }$, with $\rho$ uniformly bounded, it holds
%\begin{equation}\label{eq:key-mp-mf}
\[ \int \eeta^{{u}}(s)  \rho(x) d\mm (x) = \int \eeta^{{v}}(x)  \rho(x) d\mm (x).\]%\end{equation}
This, in turn, follows from uniqueness and \eqref{eq:stability-convex-conditioning-bounded}: both members above are $\cR$-regular solutions to the MP, with initial law $\rho \mm \le  \veps^{-1} C  {u} \mm$. % By definition of disintegration of measure, one can rewrite both sides above as
%\[  \int \eeta^{\overline{u}}(s,x)  \rho(x) d\mm (x) = (\rho \circ e_s) \eeta^{\overline{u}} \quad \text{and} \quad \int \eeta^{\overline{v}}(s,x)  \rho(x) d\mm (x) = (\rho \circ e_s) \eeta^{\overline{v}}.\]
%Since it holds
%\[ \rho \le \nor{\rho}_\infty \chi_{\cur{\overline{u} > \veps }} <  \veps^{-1}\nor{\rho}_\infty  \overline{u}, \quad \text{ $\mm$-a.e.\ in $X$,}\]
%we obtain $\rho \circ e_s \le \nor{\rho}_\infty/\veps$, $\eeta^{\overline{u}}$-a.s., and similarly
% $\rho \circ e_s \le \nor{\rho}_\infty/ \veps$, $\eeta^{\overline{v}}$-a.s.. Moreover, as $\rho \circ e_s$ is clearly $\cF_s$-measurable, by Proposition \ref{prop:restriction-abstract}, the claim is proved.

Next, we notice that there must exists some ${u} \mm \in \cR_0$ equivalent to $\mm$, i.e., such that ${u} >0$ $\mm$-a.e.\ in $\R^d$, since $\mm$ is equivalent to the supremum of all the measures in $\cR_0$ (appropriately rescaled). Then, we define $\eeta(x) := \eeta^{ u }(x)$, for $x \in \R^d$. To conclude that $\eeta(x)$ is a $\cR$-MF, we use \eqref{eq:identity-mf-wellposedness}: given any probability ${v} \mm \in \cR_0$, it holds
\[ \int \eeta(x)  v (x) d\mm = \int \eeta^{{v}}(x)  v (x) d\mm = \eeta^{ v}.\]

To prove \emph{ ii) $\Rightarrow$ i)}, existence of $\cR$-regular solutions to the MP, given the existence of a $\cR$-MF is trivial, so we focus on uniqueness. We let $\tilde{\eeta}$,  be a $\cR$-regular solution of the MP with some initial law and show that it must coincide with the one induced by the (unique) $\cR$-MF $(\eeta(x))_{x \in \R^d}$, i.e., $\tilde{\eeta} = \int\eeta(x) d\tilde{\eta}_0(x)$.  To this aim, we let ${u} \mm \in \cR_0$ be a probability measure equivalent to $\mm$, and consider the measure
\[ \overline{\eeta} := \frac 1 2 \tilde{\eeta} + \frac 1 2 \int \eeta(x) {u} (x)d\mm(x),\]
which is a $\cR$-regular solution to the $MP$ by \eqref{eq:stability-convex-combination}, whose initial law is again equivalent to $\mm$. By disintegration with respect to $e_{0}$, we obtain a Borel family of probability measures $(\overline{\eeta}(x))_{x \in  \R^d}$, which, by \eqref{eq:stability-convex-conditioning}, provides a $\cR$-MF and so by uniqueness it coincides with $\eeta(x)$, for $\mm$-a.e.\ $x \in \R^d$, yielding
\[   \int \eeta(x) \sqa{ \frac 1 2d\tilde{\eta}_0(x) + \frac 1 2 {u} (x) d\mm(x) }=  \frac 1 2 \tilde{\eeta} + \frac 1 2 \int \eeta(x) {u} (x)d\mm(x)\]
from which we conclude.
\end{proof}

We end this section with some remarks on standard consequences of uniqueness: the Markov property and stability with respect to approximation.

As in Remark \ref{rem:markov}, we consider $\cR$-regular flows with respect to a family $(\cR_{[s,T]})_{ s \in [0,T]}$ such that \eqref{eq:stability-cr-restriction} holds.

\begin{proposition}[Markov property]%\label{prop:markov}
Assume that uniqueness holds for $\cR$-regular MP's, in the sense that, for every $s \in [0,T]$, $\bar{\nu} \in \cR_s$, there exists a unique $\cR$-regular solution to the MP on $[s,T]$, with initial law $\bar{\nu}$. %and$\cR$-MF's $(\eeta^1(s,x))_{s \in [0,T], x \in \R^d}$, $(\eeta^2(s,x))_{s \in [0,T], x \in \R^d}$, then $\eeta^1(s,x)= \eeta^2(s,x)$, $\bar{\nu}$-a.s.\ $x\in \R^d$.
Then, for every $\cR$-MF $(\eeta(s,x))_{s\in [0,T], x\in \R^d}$, \eqref{eq:chap-kolmogorov} holds true, for every $r$, $s$, $t \in [0,T]$, with $r \le s \le t$.
\end{proposition}

%Notice that the $\cR$-MF is unique, in the sense of Proposition \ref{lemma:well-posedness-flows}.
The proof is straightforward from the following identity between measures on $C([s,T];\R^d)$:
\[ ( |_{[s,T]} )_\sharp \sqa{ \int \eeta(r,x) \bar{\nu}(dx)} = \int  \eeta(s,y) \sqa{\int \eta(r,x)_s \bar{\nu}(dx)}(dy),\]
which, in turn, holds true because both terms define $\cR$-regular solutions of the MP on $[s,T]$, with initial law $\int \eta(r,x)_s \bar{\nu}(dx)$: this is obvious for the right hand side, while for the left hand side it is a consequence Proposition \ref{prop:restriction} and condition \eqref{eq:stability-cr-restriction}.

%Next, we show a non-quantitative stability result (see also \cite[Theorem 3.7]{figalli-sdes}). In this case, we study only the case when \eqref{eq:domination} holds. 

Another well understood, but rather technical, property that sometimes follows from existence and uniqueness is a non-quantitative version of stability with respect to approximations, which in this setting would read as follows.

\begin{proposition}[stability] For $n \ge 1$, let $a^n$, $b^n$ be Borel maps as in \eqref{eq:a-b-rn}, let $\cL^n := \cL(a^n, b^n)$ and let $\eeta^n$ solve the MP associated to $\cL^n$. If
\begin{enumerate}[i)]
\item there exists a unique $\cR$-regular solution $\eeta$ of the MP associated to $\cL=\cL(a,b)$ with $\eta_0 = \bar{\nu}$,
\item it holds $\eta_0^n \to \bar{\nu}$ narrowly, $a^n \to a$ and $b^n \to b$ pointwise as $n \to \infty$,
\item for some convex, l.s.c\ functions $\Theta_1$, $\Theta_2$ as in Theorem \ref{theorem:basic-burkholder} it holds
\begin{equation*}%\label{eq:stability-estimate}
\limsup_{n \to \infty} \int_0^T \int  \Theta_1\bra{\abs{b^n_t} }+ \Theta_2\bra{\abs{a^n_t} } d\eta_t^ndt \le \int_0^T \int  \Theta_1\bra{\abs{b^n_t} }+ \Theta_2\bra{\abs{a^n_t} } d\eta_tdt,\end{equation*}
\item every limit point in $C([0,T]; \scrP(\R^d)) $ of  $(\eta^n)_{n\ge 1}$ belongs to $\cR$,
\end{enumerate}
then $\eeta^n \to \eeta$ narrowly in $\scrP ( C([0,T]; \R^d) )$.
\end{proposition}

Notice that we do not require that $\eeta^n$ are $\cR$-regular: in general it does not even make sense, since $\cR$ is a class of solutions to the FPE associated to $\cL$, not to $\cL^n$. A proof of the result above would not be difficult, but it would require us to combine some technical results, such as those established in Section \ref{sec:tightness} and \cite[Lemma 23]{ambrosio-crippa-edinburgh} to establish that $(\eeta^n)$ is a tight sequence and any limit point provides a $\cR$-regular solution to the MP associated to $\cL$; the conclusion is then straightforward from uniqueness. If \eqref{eq:domination} holds, then $\mm$-a.e.\ convergence in place of pointwise convergence of the coefficients is sufficient,  if we also restrict to solutions $\eeta^n$ whose marginals are absolutely continuous (as done, e.g.\ in \cite[Theorem 3.7]{figalli-sdes}).% We also mention the work in progress \cite{ambrosio-stra-trevisan}, where stability is investigated for (deterministic) DiPerna-Lions flows defined on abstract metric measure spaces.

\section{Well-posedness results}\label{sec:well-posedness-diperna-lions-figalli}

In this section, we state and prove two results (Theorem \ref{thm:wp-degenerate} and Theorem \ref{thm:wp-elliptic}) about existence and uniqueness for solutions of the FPE \eqref{eq:fpe-rn}, belonging to suitable classes of probability measures. In particular, as we are interested in the DiPerna-Lions theory, we deal with absolutely continuous with respect to the $d$-dimensional Lebesgue measure, $\mu_t = u_t \mathscr{L}^d$, satisfying some bounds on their density $u: [0,T] \times \R^d \to \R$. Besides such integrability conditions on the solution, we require (Sobolev) regularity assumptions on the coefficients $a$, $b$.

In Section \ref{sec:formal}, we give some formal derivation of the energy estimates  which eventually lead to well-posedness for FPE's; in Section \ref{sec:sobolev}, we introduce the notation for Sobolev spaces and related basic facts; in Section \ref{sec:statements} we state our results and compare their with some (related) existing literature; the technical heart of the matter is developed in Section \ref{sec:commutators}, where crucial commutator inequalities are proved; in Section \ref{sec:proof-well-posedness} we give a proof of main our results.

\subsection{Energy estimates and renormalized solutions}\label{sec:formal}

As in the classical DiPerna-Lions theory (as well as in \cite{figalli-sdes}), we rely on energy inequalities satisfied by an absolutely continuous solution $u = (u_t)_{t\in [0,T]}$ of \eqref{eq:fpe-rn} (i.e.\  $\mu_t = u_t \scrL^d$). Let us briefly sketch a formal derivation, where we assume all the quantities involved being smooth (solutions and coefficients). 
%
%In the deterministic case, i.e.\ when $\cL$ reduces to a derivation, an important remark in \cite{diperna-lions} is that uniform bounds for the density $u$ formally propagate in time if $\div b$ is uniformly bounded, via Gronwall inequality. Moreover, if one is interested in solutions that are forward in time, the assumption $\div b^- \in L^1_tL^\infty_x$ is still sufficient. Variants of this arguments can be devised, in the elliptic case, where bounds on $\div b$ can be dropped in favour of bounds on $\abs{b}$, also exploiting the validity of $d$-dimensional Sobolev inequalities.

%Before we address rigorous proofs of these bounds, leading to existence and uniqueness for FPE's, we sketch in Section \ref{sec:formall:strategy} the general approximation scheme that we follow.

%\subsubsection{General case}%\label{sec:formal:general}

The main idea is to write the equation satisfied by $t\mapsto \int \beta(u_t(x)) dx$, where 
$\beta: \R \mapsto \R$ is a smooth function (from a Lagrangian viewpoint, this amounts in choosing, as a test function, an expression involving the density of the solution $u$ itself). The chain rule gives $\partial_t \beta(u) = \beta'(u) \cL^*(u)$ and, 
%Let $\beta: [0,T]\times \R \mapsto \R$ be any smooth function and formally compute the evolution of the ``energy'' $t \mapsto \int \beta(t,u_t)d\mm$ (for brevity, we omit to write $t$)
%\[ \frac{d}{dt} \int \beta(u_t) u_t d\scrL^d = \int \sqa{\beta'(u_t)u_t + \beta(u_t)} \partial_t u d\scrL^d = \int \cL_t\sqa{\beta(u_t) +\beta'(u_t) u_t} u_t d\scrL^d,\]
%for $t \in (0,T)$ 
%(we used chain rule for $t \mapsto \beta(u_t)$), we look for a (distributional) inequality for $\partial_t \beta(u)$.
 by linearity, we consider separately the drift and diffusion terms. Straightforward calculus gives
\[ \beta'(u) (b \cdot \nabla)^* u = \beta'(u) \nabla^* (b u) = \beta'(u) u \nabla ^* b + \beta'(u) b \cdot \nabla u =   (b \cdot \nabla)^* \beta(u) + \sqa{ \beta'(u)u -\beta(u) } \nabla^* b.\]
For the diffusion part, we first notice the identity $(a : \nabla^2 )^* u = -\nabla^* ( a \cdot \nabla u  ) + \nabla^* ( ( \nabla ^*a ) u  )$, where $\nabla^*a$ is the vector field $(\nabla^*a)_i = - \sum_{j=1}^d \partial_j a_{i,j}$, for $i \in \cur{1, \ldots , d}$. Therefore, we obtain
\begin{equation*}
\begin{split}
 \beta'(u) & ( a : \nabla^2 )^* u =  \\
=& -\beta'(u)\nabla^* ( a \cdot \nabla u  ) + \beta'(u)\nabla^* ( ( \nabla ^*a ) u  ) \\
=&  -\beta'(u)\nabla^* ( a \cdot \nabla u  ) + \nabla^* ( ( \nabla ^*a ) \beta(u)) +  \sqa{ \beta'(u)u -\beta(u) }  (\nabla^*)^2a \\
=&  -\nabla^* ( a \cdot \nabla \beta(u ) ) - \beta''(u)a (\nabla u, \nabla u) + \nabla^* ( ( \nabla ^*a ) \beta(u)) +  \sqa{ \beta'(u)u -\beta(u) }  (\nabla^*)^2a \\
=&  (a : \nabla^2 )^* \beta(u) - \beta''(u) a (\nabla u, \nabla u) + \sqa{ \beta'(u)u -\beta(u) } (\nabla^*)^2 a
\end{split}\end{equation*}
where $(\nabla^*)^2 a = \nabla^*( \nabla^*a ) = \sum_{i,j} \partial_{i,j}^2 a_{i,j}$. Summing up, we have the identity
\begin{equation}\label{eq:renormalized-0} \partial_t \beta(u) =  \cL^* (\beta(u) ) - \frac{\beta''(u)}{2} a (\nabla u, \nabla u) - \sqa{ \beta'(u)u -\beta(u) } \bra{ \div \cL }, \end{equation}
where $\div \cL := - \nabla^*b - (\nabla^*)^2 a/2$. By integrating over $\R^d$,  (formally $\cL 1 = 0$), we deduce
\[ \partial_t \int \beta(u) d \scrL^d = - \int \frac{\beta''(u)}{2} a( \nabla u, \nabla u) d\scrL^d - \int \sqa{ \beta'(u)u -\beta(u) } \bra{ \div \cL } d \scrL^d.\]
If $\beta$ is convex with $\beta(0) =0$, so $\beta''(z) \ge 0$ and $\beta'(z)z -\beta(z)\ge 0$,  for   $z \in \R$, we obtain
\begin{equation}\label{eq:energy-gronwall} \partial_t \int \beta(u) d \scrL^d \le  \int \sqa{ \beta'(u)u -\beta(u) } \bra{\div \cL}^- d \scrL^d, \end{equation}
which is the key inequality we employ to show existence as well as uniqueness, for suitable choices of $\beta$. For example, letting $\beta(z) = |z|^+$, we deduce that if $u_0 \ge 0$, then $u_t \ge 0$ for $t \in [0,T]$ (thus, for simplicity, we assume that $u_t \ge 0$ in what follows). In particular, to deduce uniqueness for solutions in $L^\infty_t(L^r_x)$ (for some $r >1$), we show that the difference between any solutions $u$, $v$ satisfies \eqref{eq:energy-gronwall}, with $\beta(z) = \abs{z}^r$, and Gronwall lemma entails a uniform bound with respect to $t \in (0,T)$ for $\nor{ u_t - v_t}_{L^r_x}$. Let us also notice that, with the choice $\beta(z) = \abs{z}^2$, keeping track of the non-negative terms dropped above, we would obtain a bound for the ``Sobolev energy'' $\int_{\R^d}  a_t(\nabla u_t, \nabla u_t) d\scrL^{d} dt$ and, for $\beta(z) = |z|^{r}$, with $r >2$, of the energy 
\[ r(r-1) \int_0^T \int_{\R^d}  u^{r-2}_t a_t(\nabla u_t, \nabla u_t) d\scrL^{d} dt =  \frac{r(r-1)}{(r/2-1)^2 }\int_0^T \int_{\R^d} a_t(\nabla u_t^{r/2}, \nabla u_t^{r/2}) d\scrL^{d} dt\]

In the elliptic case, i.e.\ if there exists some constant $\lambda >0$ with $a(v,v) \ge \lambda  \abs{v}^2$ for every $v \in \R^d$, uniformly in $(0,T)\times \R^d$, we  would deduce that any weak solution $u$ actually belongs to the Sobolev space $L^2_t(W^{1,2}_x)$. Moreover, if we have no bounds on $\div \cL$ but only on $(\nabla^*)^2a $, and $b \in L^1_t(L^\infty_x)$, we may still deduce some bound with respect to the energy $z \mapsto |z|^{2}$,
\[  2 \int u_t b_t\cdot \nabla u_t d\scrL^d \le \frac{\lambda}{2} \int |\nabla u_t|^2 d\scrL^d + \frac{4 \nor{b_t}_\infty}{\lambda} \int |u_t|^2 d\scrL^d,\]
so that
\[ \partial_t \int |u_t|^2 d\scrL^d \le  \int |u_t|^2 \sqa{ \bra{(\nabla^*)^2 a_t}^+ + \frac{ 4 \nor{b_t}_\infty } {\lambda} }d \scrL^d - \frac{\lambda}{2} \int |\nabla u_t|^2 d\scrL^d,\]
and again Grownwall inequality leads to a bound for $L^2_x$, uniform in $t \in [0,T]$. Similarly, if $r > 2$, we use the inequality $2 ab \le a^2 + b^2$ thus, for every $\veps >0$,  the term $r \int u_t^{r-1} b_t\cdot \nabla u_t d\scrL^d$ (assume for simplicity that $u$ is non-negative) is estimated with
\[  \frac{r}{r/2 -1} \int u_t^{r/2} b_t\cdot \nabla u_t^{r/2} d\scrL^d \le \frac{\veps}{2} \int |\nabla u_t^{r/2}|^2 d\scrL^d + \frac{r^2 \nor{b_t}_\infty}{2(r/2-1)^2 \veps} \int |u_t|^r d\scrL^d,\]
and letting $\veps = \lambda$, we may conclude again by a Grownall argument that
\begin{equation}\label{eq:gronwall-elliptic-r} \sup_{t \in [0,T]} \nor{u_t}_{L^r_x} \le \nor{u_0}_{L^r_x} \exp\cur{ \bra{1 - \frac 1 r }\nor{ ( (\nabla^*)^2 a )^+ }_{L^1_tL^\infty_x}  + \frac{r}{2(r/2-1)^2 \lambda} \nor{b }_{L^1_tL^\infty_x} }\end{equation}

Let us finally remark that if we integrate \eqref{eq:renormalized-0} with respect to some function $f \in \cA$, with $f \ge 0$ (instead of $f =1$), we would deduce 
\begin{equation} \label{eq:renormalized-1}  \partial_t \int f \beta(u) d \scrL^d  \le  \int ( \partial_t + \cL f )   \beta(u)  d\scrL^d + \int f \sqa{ \beta'(u)u -\beta(u) }  \bra{\div \cL}^-  d\scrL^d.\end{equation}

The inequality above is so useful that weak solutions $u$ of the FPE, which also satisfy \eqref{eq:renormalized-1}  for every $f \in \cA$, $f \ge 0$, for (many) smooth convex functions $\beta$, are called \emph{renormalized solutions} \cite[Definition 4.9]{figalli-sdes}. There are abstract results connecting well-posedness for FPE's and the fact that every weak solution is renormalized, e.g.\ \cite[Lemma 4.10]{figalli-sdes} (but see also \cite{bouchut-crippa-06} for a somewhat converse result, in the deterministic framework); here, for brevity, we limit ourselves to a direct proof of uniqueness of FPE's from the validity of \eqref{eq:renormalized-0}, e.g.\ with the special choice $\beta(z) = |z|^r$.

\subsection{Sobolev spaces}\label{sec:sobolev}

Before we state and prove our main results, we briefly introduce Sobolev spaces associated to the operators $\partial_t$ and $\cL$, together with some useful facts; we use throughout a compact notation extending that in Section \ref{sec:definitions-superposition}.

For $p$, $q \in [1,\infty]$, the space $W^{1,p}_t(L^q_x)$ is defined as the space of functions $u \in L^p_t(L^q_x)$ such that the distributional derivative $\partial_t u$ is represented by a (unique) $g \in L^p_t(L^q_x)$, 
\[ \int_{0}^T \int _{\R^d} (\partial_t f)  u_t d\scrL^d dt =  - \int_0^T \int_{\R^d} f_t g_t   d\scrL^d dt, \quad \text{for every $f \in \cA_c$.}\]
We endow $W^{1,p}_t(L^q_x)$ with the Banach norm $\nor{u}_{L^p_{t}L^q_x }+ \nor{\partial_t u }_{ L^p_{t}L^q_x}$. A standard mollification argument, with respect to the variable $t \in (0,1)$, gives that $\cA$ is dense in  $W^{1,p}_t(L^q_x)$, for $p$, $q < \infty$ (for a proof of this and the following results, we refer e.g.\ to \cite[\S III.1]{showalter}). In particular, the chain rule for $\partial_t$ extends to $W^{1,p}_t(L^q_x)$, thus
\[ \partial_t (f \beta(u) ) = (\partial_t f) \beta(u) +  f \beta'(u) \partial_t u, \quad \text{ for every $f \in \cA$, $u \in W^{1,p}_t(L^q_x)$, $\beta \in C^1_b(\R)$.}\]
Another straightforward consequence of the density of $\cA$ is the fact that any $u \in W^{1,p}_t(L^q_x)$ enjoys an absolutely continuous representative, i.e.\ there exists some $\tilde{u} \in AC^p( [0,T]; L^q(\scrL^d) )$ such that $\tilde{u}_t = u_t$, for $\scrL^1$-a.e.\ $t \in (0,T)$. In particular the map: $T_0(u) := \tilde{u}_0$ (trace at $0$) is linear and continuous from $W^{1,p}_t(L^q_x)$ to $L^q_x$. Moreover, $t \mapsto \tilde{u}_t$ is strongly differentiable at $\scrL^1$-a.e.\ $t \in (0,T)$ and it holds $\frac{d}{dt} \tilde{u} = \partial_t u$.%Finally, let us remark that, if $u \in H^{1,1}(\partial_t)$, then $t \mapsto \tilde{u}_t \scrL^d$ is a narrowly continuous curve of probability measures.

We associate to the diffusion operator $\cL$ some ``Sobolev spaces''. An important role in our deductions is played by $D^p(\cL)$ (for $p \in [1,\infty)$), defined as the abstract completion of $\cA$ with respect to the norm $\nor{ f }_{D^p(\cL)}:= \nor{f}_{L^1_tL^p_{x} } + \nor{\cL f}_{L^1_tL^p_{x} }$, which is well defined whenever $a$, $b \in L^1_tL^p_{x}$ (actually, a more consistent notation for $D^p(\cL)$ would be $L^1_t(D^p(\cL_t))$). Let us remark however that, without further regularity assumptions,  the extended operator $D^{p}(\cL) \ni f \mapsto \cL f \in L^{1}_t (L^p_x)$ may be multi-valued, but the assumptions on $\cL$ that we impose in our results entail that the extension is single-valued.

A useful fact is the following: if $f \in W^{1,1}_t(L^p_x)  \cap D^{p}(\cL)$, then one can provide a sequence $(f_n)_{n\ge 1} \subseteq \cA$ converging towards $f$ both in  $W^{1,1}_t(L^p_x)$ and $D^{p}(\cL)$. Indeed, it is sufficient to consider first a sequence $(g_n)_{n\ge 1} \subseteq  \cA$  converging towards $f$ in $D^{p}(\cL)$, let $\rho$ be a smooth probability density on $\R$, and consider the approximation $g_{n,m} := g_n * \rho_m$ (where we let $\rho_m(t) = m^{-1} \rho(t/m)$, $t \in \R$, and we carefully extend $g_n$ to a continuous function outside the set $[0,T]\times \R^d$). For $(n,m) \to \infty$, the sequence $g_{n,m}$ converges towards $f$ in $D^{p}(\cL)$, because $g \mapsto g * \rho_m$ is a contraction in $D^{p}(\cL)$, as convolution with respect to $t$ and the operator $\cL$ commute; for fixed $m\ge 1$, the sequence $g_{n,m}$ converges towards $f * \rho_m$, because $g \mapsto g * \rho_m$ is continuous from $L^1_t(L^p_x)$ into $W^{1,1}_t(L^p_x)$, with norm smaller than $\nor{\rho_m}_\infty$. Moreover, as $m \to \infty$, $f * \rho_m$ converges towards $f$ in $W^{1,1}_t(L^p_x)$, since $f \in W^{1,1}_t(L^p_x)$ (this is exactly the standard mollification argument providing density of $\cA$ in $W^{1,1}_t(L^p_x)$). By a diagonal argument, we finally extract a sequence $(f_n)_{n\ge 1}$ as required. %Let us also notice that a similar argument holds true also for different choices of exponents, e.g. $W^{1,p}_t(L^q_x) \cap D^{r}(\cL)$: 
As a consequence, if $u \in L^\infty_tL^r_x$ ($r>1$) is a narrowly continuous solution of \eqref{eq:fpe-rn}, with $a$, $b \in L^1_tL^p_x$, then the weak formulation \eqref{eq:weak-fpe-rn} extends to $f \in W^{1,1}_t(L^{r'}_x) \cap D^{r'}(\cL)$:
\begin{equation}\label{eq:weak-fpe-rd-hp} \int_0^T \int  \sqa{ (\partial_t  + \cL_t) f } u_t d\scrL^d dt = \int f_T u_T d\scrL^d - \int f_0 u_0 d\scrL^d,\end{equation}
where by $f_T \in L^{r'}_x$ and $f_0\in L^{r'}_x$ we mean the continuous representative of $f$ evaluated at $T$ and $0$.

Similarly, we introduce the space $D^{p}(\cL, a \nabla \otimes \nabla )$ as the abstract completion of $\cA$ with respect to the norm $\nor{ \abs{f} + \abs{\cL f } + a(\nabla f, \nabla f) }_{L^1_tL^p_{x} }$. Clearly, this is a space than $D^{p}(\cL)$, but is useful because the following chain rule holds, for $D^{p}(\cL, a \nabla \otimes \nabla )$, and $\gamma \in C^2(\R)$, with $\gamma'$ and $\gamma''$ uniformly bounded:
 % \cL ( f)  \gamma( u ) + a( \nabla f, \nabla u) \gamma'(u) + 
\[ \cL ( \gamma( u ) ) =
\gamma'(u) \cL(u) + \gamma''(u) a(\nabla u ,\nabla u).\]
As in the previous case, it might be that $f \mapsto \cL(f)$ and $f \mapsto a(\nabla f, \nabla f)$ are not single-valued, but the identity above holds true with the natural interpretation (and in our results we introduce assumptions ensuring that these are well-defined functions).

We always consider the divergence $\div \cL$ be defined in the sense of distributions, i.e., as the linear operator $\cA_c \ni f \mapsto \int_0^T \int \cL_t f d\scrL^{d} dt$ (provided that $a$, $b$ are locally integrable). We say that $\div \cL \in L^1_{t}(L^p_x)$ if there exists a (necessarily unique) $g \in L^1_{t}(L^p_x)$ such that
\[
\int_{(0,T)\times\R^d} \cL f d\scrL^{1+d} =  - \int_{(0,T)\times\R^d} f g d\scrL^{1+d}, \quad \text{for every $f \in \cA_c$.}\]
Similarly, $\div \cL^{-} \in L^1_{t}(L^p_x)$ if, for some $g \in L^1_{t}(L^p_x)$, the inequality $\le$ in place of equality above holds, for every $f \in \cA$, with $f \ge 0$. If $\div\cL^- \in L^1_{t}(L^p_x)$, then we can prove the following inequality, for $u \in D^{1,p}(\cL, a \nabla \otimes \nabla)$, and $\beta \in C^3(\R)$ convex, with bounded derivatives as well as $\beta'(z)z - \beta(z)$ bounded:
\begin{equation}
\label{eq:ibp-diffusion} \int_{(0,T)\times \R^d}  \cL ( \beta'(u) ) u d\scrL^{1+d} \le \int_{(0,T)\times \R^d} \sqa{ \beta'(u)u - \beta(u)} \div \cL^- d \scrL^{1+d}.\end{equation}
Indeed, let $\rho$ be a smooth convolution kernel on $\R^d$ and consider the diffusion operator $\cL^m$ with smooth coefficients $a* \rho^m$ and $b* \rho^m$ (where we let $\rho^m(x) = m^{-d} \rho(x/m)$). If we also assume $u \in \cA$, then the inequality above holds true by the derivation as in Section \ref{sec:formal} above. The general case follows by approximation, letting first $m \to \infty$ and then choosing $u^n \in \cA$ converging towards $u$ in  $u \in D^{p}(\cL, a \nabla \otimes \nabla)$.

Besides these spaces associated with $\cL$, let us recall some features of standard Sobolev spaces and the smoothing properties of the standard heat semigroup $(P^\alpha)_{\alpha \ge 0}$ on $\R^d$. For $p \in [1,\infty]$, we consider spaces
\[ W^{1,p}_x := \cur{ f \in L^{p}_x \, : \, \nabla f \in  L^p_x }, \quad  W^{2,p}_x = \cur{ f \in  W^{1,p}_x  \,  : \,  \nabla^2 f \in  L^p_x },\]
endowed with the usual norms.

%These are for which we have the following representation formula:
%\begin{equation}\label{eq:heat-semigroup} \sfP^\alpha f (x) = \int_{\R^d} f(x + \sqrt{2 \alpha} y) \frac{e^{ -\abs{y}^2/2} }{\sqrt{(2\pi)^n} } dy, \quad \text{ for $x \in \R^d$.}\end{equation}

A crucial fact for our deductions are quantitative inequalities for the smoothing effect of the heat semigroup $(\sfP^\alpha)_{\alpha\ge 0}$, which can be deduced by straightforward computations from the heat kernel in $\R^d$. Of course, $\sfP^\alpha$ is a contraction semigroup in $W^{1,p}_x$ as well as $W^{2,p}_x$; moreover, integration by parts and H\"older inequality give
\begin{equation}\label{eq:smoothing-heat-lpgamma} \sqrt{\alpha} \nor{ \nabla \sfP^\alpha f}_{L^p_x} \le c \nor{ f }_{L^p_x} \quad \text{for every $\alpha \in (0,\pinfty)$,} \end{equation}
with $c$ depending on $p \in [1,\infty]$ only (possibly also on the dimension $d$). Such inequalities, called $L^p-\Gamma$ in \cite{ambrosio-trevisan}, play a fundamental role for our approach to continuity equations in metric measure spaces: their validity in abstract setups as well as in Riemannian manifolds follow e.g.\ from uniform lower bounds on the Ricci curvature.

Arguing similarly, it holds for $p \in [1,\infty]$, $i, j \in \cur{1, \ldots d}$,
\begin{equation}\label{eq:smoothing-heat-lpdelta} \alpha \nor{ \partial_{i,j}^2 \sfP^\alpha f}_{L^p_x} \le c \nor{ f }_{L^p_x} \quad \text{for every $\alpha \in (0,\infty)$.} \end{equation}

Let us also notice that, as $\alpha \downarrow 0$, the left hand side in the two inequalities above are infinitesimal, for a standard density and uniform boundedness argument applies.

Finally, another property that we occasionally use below is that, for $p \in (1,\infty)$, one has $W^{2,p}_x := \cur{ f \in L^p_x\, :\, \Delta f \in L^p_x }$, because of the $L^p_x$-boundedness of the second order Riesz transform $f \mapsto \nabla^2 \Delta^{-1}f$, see e.g.\ \cite{gilbarg-trudinger}.

\subsection{Well-posedness: statement of results}\label{sec:statements}

We are in a position to state our main well-posedness results, which we split in two theorems: the first one deals with possibly degenerate  diffusions, with Sobolev regular coefficients.

\begin{theorem}[degenerate case]\label{thm:wp-degenerate}
Let $p \in (1,\infty]$, $r \ge  2p/(p-1)$, and $a$, $b$ be as in \eqref{eq:a-b-rn}, with
\[ a \in L^1_t(W^{2,p}_x),\quad  b \in L^1_t(W^{1,p}_x), \quad \text{and} \quad \div \cL^{-} \in L^1_tL^\infty_x.\]
Then, for every probability density $\bar{u} \in L^{r}_x$, there exists a unique narrowly continuous solution $u = (u_t)_{t \in [0,T]}$ of the FPE \eqref{eq:fpe-rn} with $u_0 = \bar{u}$ and $u \in L^\infty_t(L^r_x)$.
\end{theorem}

Actually, the technique employed provides (existence and) uniqueness even without the assumption that $\bar{u}$ is a probability density. As a straightforward consequence of the result above and the equivalence established in the previous section, if we let $\cR$ be the class of narrowly continuous solutions $u$ of the FPE \eqref{eq:fpe-rn},  with $u \in L^\infty_t(L^r_x)$,  we deduce existence and uniqueness for $\cR$-regular martingale problems as well as for  $\cR$-regular martingale flows. The unique regular flow satisfies the Chapman-Kolmogorov equations \eqref{eq:chap-kolmogorov}.

Our second statement deals with non-degenerate (elliptic) diffusions, i.e.\ if it holds, for some $\lambda >0$,  $a(v, v) \ge \lambda \abs{v}^2$, for every $v \in \R^d$, a.e.\ in $(0,T) \times \R^d$. In such a case, we can remove one order of Sobolev regularity assumption from both coefficients, but we introduce Lipschitz regularity for  $t \mapsto a_t$. 

\begin{theorem}[bounded elliptic case]\label{thm:wp-elliptic}
Let $p \in [2,\infty]$, $r \ge 2p/(p-2) \in [2,\infty]$, and $a$, $b$ be as in \eqref{eq:a-b-rn}, with $a \in L^\infty_t(L^\infty_x)$ and elliptic,
\[ \partial_t a \in L^\infty_{t,x}, \quad a \in L^1_t(W^{1,p}_x), \quad b \in L^1_tL^\infty_x \quad \text{and} \quad ( (\nabla^*)^2 a)^{-} \in L^1_t(L^\infty_x).\] Then, for every probability density $\bar{u} \in L^{r}_x$, there exists a unique narrowly continuous solution $u = (u_t)_{t \in [0,T]}$ of the FPE \eqref{eq:fpe-rn} with $u_0 = \bar{u}$  and $u \in L^\infty_t(L^r_x)$.  
\end{theorem}

Also in this case, as a straightforward consequence of the equivalence between Eulerian and Lagrangian descriptions,  we obtain well-posedness for $\cR$-regular martingale problems as well as $\cR$-regular flows, with $\cR$ as in the previous case.

\begin{remark}[comparison with existing literature]
The literature on the subject of Fokker-Planck equations and martingale problems is so vast and growing that we limit ourselves to a direct comparison only with very closely related and recent works. In particular, we stress some aspects  which are different from the results appearing in \cite{figalli-sdes}, \cite{lebris-lions}. 

In \cite{lebris-lions}, the approach is mostly Eulerian, dealing with FPE's in divergence form
\[ \partial_t u_t + \nabla^*( u_t b) = \frac 1 2 \nabla^* ( \sigma \sigma^* \nabla u_t), \quad \text{on $(0,T)\times \R^d$,}\]
with $\sigma: \R^d \to \R^{d\times k}$. The main result in \cite{lebris-lions} provides existence and uniqueness for the equation above, provided that
\[
b \in L^1_t(W^{1,1}_{loc}), \quad \frac{b}{1 + |x|} \in L^1_t(L^1_x + L^\infty_x), \quad \nabla^* b \in L^1_t(L^\infty_x)\]
\[
\sigma \in L^2_t(W^{1,2}_{loc}), \, \frac{\sigma}{1 + |x|} \in L^2_t(L^2_x + L^\infty_x).
\]
To compare these assumptions, we must notice as in \cite[\S 5.1]{lebris-lions} that with our notation $a = \sigma \sigma^*$ and the drift is actually $b - \frac 1 2 \nabla^* a$. In view of this correspondence, it might seem that Theorem \ref{thm:wp-degenerate} follows from their weaker assumptions: this follows in principle from a result of the type $\sigma := a^{1/2} \in L^2_t(W^{1,2}_{loc})$, if $a \in L^1_t(W^{2,2}_{loc})$, extending the well-known result \cite[Lemma 3.2.3]{stroock-varadhan} that $a^{1/2}$ is Lipschitz whenever $a \in C^2$. However, their conclusions are in fact weaker, and actually insufficient in order to obtain correspondent Lagrangian results: they prove existence and uniqueness in the class of narrowly continuous probability densities $u \in L^\infty_t(L^\infty_x)$ such that $\sigma \nabla u \in L^2_t(L^2_x)$: the latter (weak) regularity condition then prevents from a straightforward application of the results in Section \ref{sec:correspondence-rn}. In conclusion, our result has (apparently) stronger regularity conditions on the coefficients, but draws stronger results and leads directly to well-posedness of regular martingale problems and flows.

The problem arising from the condition $\sigma \nabla u \in L^2_t(L^2_x)$, which prevents a Lagrangian theory, is well understood in \cite{figalli-sdes}, where much effort is put in showing, for the bounded elliptic case, uniqueness in the class of narrowly continuous probability densities $u \in L^2_t(L^2_x)$ \cite[Theorem 4.3]{figalli-sdes}. When compared with the assumptions of Theorem \ref{thm:wp-elliptic}, an evident difference is that we require a first order condition $a \in L^1_t(W^{1,p}_x)$, while no such requirement appear in \cite[Theorem 1.3]{figalli-sdes}, besides (with our notation) $\nabla^*a$, $\div \cL^- \in L^\infty_t(L^\infty_x)$. The technique we employ -- approximation by the semigroup associated to the Dirichlet form $f \mapsto \int a(\nabla f, \nabla f)$ -- is the same as Figalli's one, and in the elliptic case the novelty is more conceptual, providing a much cleaner derivation of commutator estimates, essentially by the same abstract arguments in the elliptic and the degenerate case. However, in the possibly degenerate case, our results are stronger, compare e.g.\ with  \cite[Theorem 1.4]{figalli-sdes}, as we allow for much more general diffusion coefficients, and possibly unbounded terms -- obtaining as well Lagrangian counterparts.

In more recent years, further developments along these research lines appeared in the literature, as well as different techniques (e.g.\ Crippa-DeLellis' technique \cite{crippa-delellis-08} was extended to SDE's in \cite{zhang-10, rockner-zhang-uniqueness-fpe}): of course, novelties and improvements appear in these developments, but to the author's knowledge that they address different aspects (such as strong solutions, equations with jumps, quantitative estimates, etc.) and there is no substantial overlap with our two results above.

We also point out that the theory of measure-valued solutions (i.e., not necessarily absolutely continuous) Fokker-Planck equations, at least in the elliptic case, is well-developed and some results may be compared with ours. For example, \cite[Proposition 3.1]{bogachev-daprato-rockner-stannat} entail uniqueness if, for some $p \ge d+2$, $a \in L^\infty_t H^{1,p}_x$ is elliptic, $b\in L^p_t L^p_x$ and $t \mapsto a_t$ is H\"older continuous (locally uniformly in $x$).  It is immediate to see that there is no inclusion between such class of coefficients and that of Theorem \ref{thm:wp-elliptic}, and in particular the hypothesis of our result are dimension-free (indeed, we are specializing a theory tailored for infinite dimensional spaces). However, the uniqueness class is smaller in our case, since we restrict from the very beginning to absolutely continuous solutions, which is nevertheless sufficient to entail a reasonable Lagrangian theory. Let us point out some recent developments \cite{bogachev-daprato-rockner-2008, bogachev-rockner-shapo-2011, bogachev-rockner-shapo-2013} and in particular \cite{bogachev-rockner-shapo-2015} which also contains a survey of known results and methods for the degenerate case. Finally, we point out the monograph in preparation \cite{bogachev-rockner-krylov-shapo-2015}, which contains a detailed study and a vast bibliography on the subject.
\end{remark}

Let us briefly discuss some features of the two theorems above and their proof. First, existence of weak solutions in the hypothesis stated above is a much easier task than uniqueness: for example, one can argue by  approximation via convolution of the coefficients (and the initial law) with a smooth kernel, so that the estimates on the coefficients are preserved, and one gains enough regularity (e.g.\ $C^2$ coefficients) so existence is available even at the Lagrangian level. Then, we have enough regularity so that the deductions which lead to inequality \eqref{eq:renormalized-0}  apply and by a Gronwall argument we deduce a bound in $L^\infty_t(L^r_x)$, in terms of $\div \cL^-$ only, and uniform in the approximation (in the elliptic case, we argue with \eqref{eq:gronwall-elliptic-r} instead). By extracting a weakly convergent sequence and by strong convergence of the approximations of the coefficients, we deduce that any weak limit point in $L^\infty_t(L^r_x)$ is a weak solution to the FPE \eqref{eq:fpe-rn}. In the elliptic case, we deduce as well existence for a solution $u \in L^\infty_t(L^r_x) \cap L^2_t(W^{1,2}_x)$. Let us also recall the approach \cite[Theorem 4.3]{figalli-sdes}, which is completely Eulerian (i.e., it relies on PDE's techniques only), and has the advantage of yielding easily uniqueness, for solutions belonging to such a (smaller) space, which does not allow for applications of the theory developed in Section \ref{sec:correspondence-rn}.% Since our study of FPE's is related to the Lagrangian description, these results can be only useful for the existence part, while for uniqueness they provide intermediate steps towards well-posedness in a larger class, e.g.\ for $u \in L^\infty_tL^2_x$.

In order to establish uniqueness of solutions, our aim is to rigorously establish \eqref{eq:renormalized-1} and \eqref{eq:gronwall-elliptic-r}, for (difference of) solutions $u \in L^\infty_t(L^r_x)$. As already remarked, the main problem is related to the regularity of $u$, in order to employ the standard calculus rules. Our strategy relies the well-known smoothing scheme, which dates back at least to \cite{diperna-lions}: for $\alpha \in (0,1)$ we introduce some linear operator $\sfP^\alpha$, acting on functions defined on $(0,T) \times \R^d$ such that, by defining $u^\alpha := \sfP^\alpha u$, we obtain an approximation of $u$ sufficiently regular to rigorously obtain \eqref{eq:renormalized-1}. Of course, the price that we pay is that $u^\alpha$, in general, is not a solution of \eqref{eq:fpe-rn} and one has to carefully estimate the ``error terms'' thus appearing: our novel contribution indeed provides a systematic approach to such inequalities.

To be more precise, in the cases that we consider, the operators $(\sfP^\alpha)_{\alpha \ge 0}$ form a strongly-continuous Markov symmetric semigroup on $L^2( (0,T)\times \R^d, \scrL^{d+1} )$, so that, in particular, $\sfP^\alpha$ preserves all $L^{p}_tL^q_x$ spaces, for $p, q \in [1,\infty]$. If we also prove that $\sfP^\alpha$ maps $W^{1,1}_t(L^{r'}) \cap D^{r'}(\cL)$ into itself, we may write, for $f$ belonging to such space,
\begin{equation}\label{eq:weak-formulation-extended} \int_0^T \int \sqa{(\partial_t + \cL _t) f} u^\alpha_t d \scrL^d dt = \int f^\alpha_T u_T d\scrL^d - \int f^\alpha_0 u_0 d\scrL^d + \int_{(0,T)\times \R^d}     u  \sqa{ \sfP^\alpha,  (\partial_t + \cL _t) } f d\scrL^{1+d},\end{equation}
since the weak formulation \eqref{eq:weak-fpe-rd-hp} extends by density of $\cA$ in $W^{1,1}_t(L^{r'}) \cap D^{r'}(\cL)$. The commutator term appears as an algebraic way to highlight the identity as an equation for $u^\alpha$, and all the issue is to show that it is infinitesimal as $\alpha \downarrow 0$.%The commutator term
%\%[ \int_{(0,T)\times \R^d}  u  \sqa{ \sfP^\alpha,  (\partial_t + \cL _t) } f d\scrL^{1+d} = %\int_{(0,T)\times \R^d}  u \sfP^\alpha (\partial_t + \cL _t) f  - d\scrL^{1+d}

Next, we prove that $\sfP^\alpha$ has a ``smoothing effect'', %in the sense that $u^\alpha \in W^{1,1}_t(L^{r'}) \cap D^{r'}(\cL, a \nabla \otimes \nabla)$,
in a sense that we can choose $\beta'(u^\alpha)$ as a test function, and apply the chain rule with respect to $\partial_t$ and \eqref{eq:ibp-diffusion}, so
\[
\partial_t \int \beta(u^\alpha_t) d\scrL^d  \le \int_{\R^d} \sqa{ \beta'(u_t^\alpha)u_t^\alpha  - \beta(u_t^\alpha)} \div \cL^-_t d \scrL^{d} + \int_{\R^d}  u_t \sqa{ \sfP^\alpha,  (\partial_t + \cL _t) } \beta'(u^\alpha_t)  d \scrL^{d}, 
\]
$\scrL^1$-a.e.\ $t \in (0,T)$ and in the sense of distributions on $(0,T)$. Finally, we let $\alpha \downarrow 0$, and by strong convergence of $u^\alpha$ towards $u$ in $L^1_t(L^r_x)$, we are able to conclude, provided
\[ \int_{\R^d}  u_t  \sqa{ \sfP^\alpha,  (\partial_t + \cL _t) } \beta'(u^\alpha_t)  d \scrL^{d} \le \varepsilon(\alpha) \to 0, \quad  \text{ in $L^1(0,T)$ as $\alpha \downarrow 0$}.\]

\subsection{Commutator inequalities}\label{sec:commutators}

In this section, we estimate the ``error terms''  involving the commutator between $\sfP^\alpha$ and $\partial_t + \cL$. Our general strategy is a further development of that first introduced in \cite{ambrosio-trevisan}, in the framework of continuity equations in metric measure spaces, and it is completely ``coordinate free'' and depends  on an interpolation argument \emph{\`a la} Bakry-\'Emery, namely
\begin{equation*}\begin{split}
 \int  u  \sqa{ \sfP^\alpha,  \partial_t + \cL } f d\scrL^d & = \int \sqa{ \sfP^\alpha u(\partial_t + \cL)   f  - u (\partial_t + \cL) \sfP^\alpha f} d\scrL^d \\
 & = \int \int_0^\alpha \frac{d}{ds}\sqa{ (\sfP^s u) (\partial_t + \cL)  \sfP^{\alpha-s} f} ds  d\scrL^d \\
 &= \int_0^\alpha  \int \sfP^s u \sqa{ \sfDelta, \partial_t + \cL  } \sfP^{\alpha-s} f d\scrL^d ds,
 \end{split}\end{equation*}
where we let $\sfDelta$ be the generator of $(\sfP^\alpha)_{\alpha \ge 0}$. It turns out that the commutator between $\sfDelta$ and $\partial_t + \cL$, reflecting the ``relative regularity'' between the chosen approximation and the target diffusion, depends upon natural quantities such as Sobolev regularity of the coefficients.

In principle, this method provides very general results but, for the ease of exposition, we address separately only three cases, which are of particular interest: the case of a commutator between the Euclidean heat semigroup and a Sobolev derivation, which is a specialization of \cite[Lemma 5.8]{ambrosio-trevisan} in the Euclidean case; that of a commutator between the Euclidean heat semigroup and a second-order Sobolev diffusion, which is apparently novel, that we settle by performing a ``second order'' interpolation argument; and finally that of the commutator between $\partial_t$ and a non-degenerate diffusion acting only the variable $x \in \R^d$, with $t \mapsto a_t$ Lipschitz, which provides an alternative approach to Step 2 in \cite[Theorem 4.3]{figalli-sdes}.

We let throughout $q \in (1,\infty]$, $r$, $s \in (1, \infty)$, with $q^{-1} + r^{-1} + s^{-1}  =1$ (one can deal with endpoint case at the price of more delicate approximations).

\begin{lemma}\label{lemma:commutator-estimate}
Let $b \in L^1_t(W^{1,q}_x)$,  $u\in L^\infty_t(L^r_x)$ and $f \in L^\infty_x(W^{1,s}_x)$. It holds
\begin{equation} \label{eq:commutator-sobolev-derivation-bilinear} \int_0^T \abs{ \int u_t [ \sfP^\alpha, b_t \cdot \nabla ] f_t d\scrL^d } dt \le  c \nor{\nabla b}_{L^1_tL^q_x} \nor{u}_{L^\infty_tL^r_x}  \nor{f}_{L^\infty_tL^s_x},\, \text{ for $\alpha \in (0,1)$,}\end{equation}
where $c \in \R$ is some constant (depending on the dimension $d$ only).
\end{lemma}

Actually, the proof below shows that $\nabla b$ can be replaced with the symmetric part of the derivative (also called deformation) $D^{sym} b := ( \nabla b + (\nabla b)^{\tau} )/2$, where $\tau$ denotes the transpose operator.

As a consequence of \eqref{eq:commutator-sobolev-derivation-bilinear}, the commutator operator  $L^\infty_x(W^{1,s}_x) \ni f\mapsto [ \sfP^\alpha, b \cdot \nabla ] f \in L^1_tL^{r'}_x $ extends to a linear continuous operator on $L^\infty_tL^s_x$. Moreover, a standard density and uniform boundedness argument entails that, for $f \in L^\infty_tL^s_x$,
\begin{equation*}%\label{eq:strong-convergence-commutator-derivation} 
[ \sfP^\alpha, b \cdot \nabla ] f  \to 0, \quad \text{ strongly in $L^1_t(L^{r'}_x)$ as $\alpha \downarrow 0$.}\end{equation*}

\begin{proof}

It is sufficient to argue assuming that $b$, $u$ and $f$ are sufficiently smooth, e.g., $u \in \cA_c$, $f \in \cA$, as well as $b^i \in \cA$, for $i \in \cur{1, \ldots, d}$, as the general inequality will follow by approximation (e.g.\ by convolution with a smooth kernel). Moreover, we argue at $t \in (0,T)$ fixed and then integrate over the interval $(0,T)$: thus we omit to specify $t \in (0,T)$ in what follows.

The curve $s \mapsto F(s) =  \int u^s  b \cdot \nabla f^{\alpha-s} d\scrL^d$ is then $C^1_b(0,\alpha)$, with
\[ \frac{d}{ds} F (s)   = \int (\Delta u^s) b \cdot \nabla f^{\alpha-s} - u^s  b \cdot \nabla (\Delta f^{\alpha-s}) d\scrL^d,\]
By straightforward integration by parts, we obtain the following alternative expression for the right hand side above:
\[ \frac{d}{ds} F (s) = \int \sqa{ \bra{ (\nabla b + (\nabla b)^\tau ) \nabla  u^s ,  \nabla f^{\alpha-s}} + (\nabla^* b)\bra{  \nabla u^s \cdot \nabla f^{\alpha-s} -  u^s \Delta f^{\alpha-s} }} d\scrL^d.\]
If $\nabla^* b = 0$, the conclusion is immediate, since we may estimate $\abs{F(\alpha) - F(0)} \le \int_0^\alpha \abs{  \frac{d}{ds} F (s) } ds$ and, by H\"older inequality, 
\begin{equation*}\begin{split} \abs{ \int u [ \sfP^\alpha, b \cdot \nabla ] f d\scrL^d } & \le  2 \int_0^\alpha \nor{D^{sym}b}_{L^q_x}\nor{ \nabla  u^s }_{L^r_x} \nor{ \nabla  f^{\alpha-s} }_{L^s_x} ds \\ & \le  \nor{D^{sym}b}_{L^q_x} \nor{ u }_{L^r_x} \nor{f }_{L^s_x} \int_0^\alpha \frac{2ds}{\sqrt{ s(\alpha-s)}} \\ & 
\le 2\pi \nor{D^{sym}b}_{L^q_x} \nor{ u }_{L^r_x} \nor{f }_{L^s_x}.
\end{split}
\end{equation*}
by \eqref{eq:smoothing-heat-lpgamma} and using $\int_0^1 (s (1-s))^{-1/2} ds = \pi$.

The general case $\nabla^* b \in L^q_x$ is slightly more involved: let us first notice that the term $(\nabla^* b) \nabla u^s \cdot \nabla f^{\alpha-s}$ can be estimated as above, adding a contribution $\pi \nor{ \nabla^* b}_{L^q_x}$ to the inequality. Finally, to estimate the contribution of $(\nabla^* b) u^s \Delta f^{\alpha-s}$ we do not put the absolute value inside integration with respect to $s \in (0,\alpha)$, but exchange integration with respect to $x$ and $s$,  exploiting the identity
\[ \int_0^\alpha (\nabla^* b) u^s \Delta f^{\alpha-s} ds = -(\nabla^* b) \int_0^\alpha u^s  \frac{d}{ds} f^{\alpha-s} ds. \]
Next, to integrate by parts only ``half of the derivative'' with respect to $s$, we simply add $(\nabla^* b) u^\alpha $ times the quantity
%\begin{equation} \label{eq:taylor-first-order}
\[  f^0 - f^\alpha - \int_0^\alpha \frac{d}{ds} f^{\alpha-s}  = 0,\]%\end{equation}
thus
\[ \abs{\int_0^\alpha u^s  \frac{d}{ds} f^{\alpha-s} ds} \le \abs{  u^\alpha \bra{ f^0 - f^\alpha  }} + \int_0^\alpha \abs{ (u^s - u^\alpha) \Delta f^{\alpha-s} } ds,\]
which, once integrated with respect to $x \in \R^d$, by H\"older inequality and \eqref{eq:smoothing-heat-lpdelta} is bounded from above by
\[  \nor{ \nabla^* b}_{L^q_x}\nor{ u }_{L^r_x} \nor{ f}_{L^s_x} \bra{ 2   + c \int_0^\alpha \frac{ds } { \sqrt{s (s -\alpha)}} }.\]
This settles an analogue of \eqref{eq:commutator-sobolev-derivation-bilinear} at fixed $t \in (0,T)$, and by integration with respect to $t \in (0,T)$, we obtain \eqref{eq:commutator-sobolev-derivation-bilinear}.
%\begin{equation}\label{eq:commutator-integrated} \int_0^T \abs{ \int u_t [ \sfP^\alpha, b_t \nabla ] f_t d\scrL^d } dt \le  c \int_0^T \nor{D^{sym}b_t}_{L^q_x} \nor{u_t}_{L^r_x}  \nor{f_t}_{L^s_x}dt,\, \text{ for $\alpha \in (0,1)$,} \end{equation}
%and then \eqref{eq:commutator-sobolev-derivation-bilinear}.
%To remove the smoothness assumptions on $b$, $u$ and $f$, it is sufficient to consider a sequence of functions $u^n \in \cA_c$, $f^n \in \cA$, as well as $b^n$ with components in $\cA$ which converge towards $u$, $f$ and $b$ with respect to the appropriate topologies (e.g., we may let $f^n$ and $b^n$ be obtained by convolution with a smooth kernel on $\R^{1+d}$).
%The convergence result is an easy consequence: indeed, the convergence holds true for fixed $f \in L^\infty_t(H^{1,s}(\R^d) )$ and by density and uniform boundedness it extends also for $f \in L^{\infty,s}_{t,x}$.
\end{proof}

The constant $c$ can be even independent of the dimension $d$ of the underlying space, provided that assume some bound directly on $\nor{ \nabla^* b }_{L^q_x}$, and use a refined, dimension independent estimate for $\nor{ \Delta f^{\alpha-s} }_{L^s_x}$: these are the key observation that lead to well-posedness on possibly infinite dimensional spaces, as developed in \cite{ambrosio-trevisan}.

 \begin{lemma}\label{lem:commutator-diffusion-rd} 
Let $a \in L^1_t(W^{2,q}_x)$, $u\in L^\infty_t(L^r_x)$ and $f \in L^\infty_x(W^{2,s}_x)$. For $\alpha \in (0,1)$, it holds 
\begin{equation}
 \label{eq:commutator-taylor-second-order-rd} \int_0^T \abs{ \int u_t [\sfP^\alpha, a_t : \nabla^2] f_t d\scrL^d  - \alpha \int u_t [\Delta, a_t : \nabla^2] \sfP^\alpha f_t d\scrL^d}dt \le c \nor{\nabla ^2 a}_{L^\infty_tL^q_x}  \nor{u}_{L^\infty_tL^r_x} \nor{f}_{L^\infty_xL^s_t}\end{equation}
where $c$ is some constant (depending on $d$ only). Moreover, for $u \in L^\infty_t( L^r_x \cap L^s_x)$, it holds
\begin{equation}\label{eq:commutator-diffusion-infinitesimal-rd}\abs{ \int u [\sfP^\alpha, a : \nabla^2](\sfP^\alpha u) d\scrL^d } \to 0, \quad \text{in $L^1(0,T)$, as $\alpha \downarrow 0$.}\end{equation}
\end{lemma}

\begin{proof}
To establish \eqref{eq:commutator-taylor-second-order-rd}, the underlying idea is to formally rewrite $a : \nabla^2 f = a: (\nabla^2 \Delta^{-1}) \Delta f$ and exploit the boundedness of the Riesz transform $\nabla^2 \Delta^{-1}$ in $L^s_x$, together with a second order interpolation along the heat semigroup. To make computations more transparent, we argue on coordinates, i.e., we fix $i$, $j \in \cur{1, \ldots, d}$ and consider the commutator
\[ [ \sfP^\alpha, a^{i,j} \partial^2_{i,j}] f = \sfP^\alpha (a^{i,j} \partial^2_{i,j} f) - a^{i,j} \partial^2_{i,j} (\sfP^\alpha f).\]
As in the proof of the previous lemma, we may also let $u \in \cA_c$, $f$ and $a^{i,j}$ be sufficiently regular, e.g.\ $f$, $a^{i,j} \in C^4_b((0,T)\times \R^d)$, and argue at fixed $t \in (0,T)$. We consider the curve
\[ [0,\alpha] \ni s \mapsto F(s) := \int u^s a^{i,j} \partial^2_{i,j} f^{\alpha -s} d\scrL^d,\]
which is $C^1_b(0,\alpha)$, with
\[ F'(s) =  \int u^{s} [\Delta, a^{i,j} \partial^2_{i,j} ] f^{\alpha- s} d\scrL^d = \int u^{s} [\Delta, a^{i,j} ] \partial^2_{i,j} f^{\alpha- s} d\scrL^d,\]
since the Laplacian and partial derivatives commute. We write $h^{\alpha-s} := \partial^2_{i,j} f^{\alpha- s} = (\partial^2_{i,j} f)^{\alpha- s}$ (since derivatives and heat semigroup commute), let $b := \nabla a^{i,j}$ and integrate by parts, obtaining
\[ F'(s) =   2\int u^s b \cdot \nabla h^{\alpha-s}d\scrL^d + \int u^s (\Delta a^{i,j}) h^{\alpha-s} d\scrL^d. \]
Differentiating once more, since $F \in C^2_b(0,\alpha)$,  we obtain 
\[F''(s) = 2\int u^s [\Delta, b \cdot \nabla ] h^{\alpha-s} d\scrL^d  + \int u^s [\Delta, (\Delta a^{i,j})] h^{\alpha-s} d\scrL^d. \]
We introduce a second order interpolation based on the Taylor expansion
%\begin{equation}\label{eq:second-order-taylor-rd}
 \[F(\alpha) - F(0) - \alpha F'(0) = \int_0^\alpha F''(\sigma) (\alpha - \sigma) d\sigma,\]
 %\end{equation}
and we notice that the left hand side gives, up to integration on $(0,T)$,  the left hand side of \eqref{eq:commutator-taylor-second-order-rd}.

Let us notice first how we would conclude in case $\nabla^* b = \Delta a^{i,j} = 0$, and then address the general case. As in the previous lemma, we obtain the identity
\[ \int u^{s} [\Delta, b \cdot \nabla ] h^{\alpha - s} d\scrL^d = - 2 \int \bra { (\nabla^2 a^{i,j})  \nabla u^{s},  \nabla h^{\alpha-s}} d\scrL^d \]
and we estimate
\begin{equation*}\begin{split}% \label{eq:second-derivative-crucial-rd}
\abs{F''(s)} &\le  4\nor{\nabla^2 a^{i,j}}_{L^q_x} \| \nabla u^s \|_{L^r_x}  \|\nabla h^{\alpha -s}\|_{L^s_x} \\
&\le  \frac{c }{\sqrt{s (\alpha - s)^3}} \nor{\nabla^2 a^{i,j}}_{L^q_x}\nor{u}_{L^r_x} \nor{ f }_{L^s_x},
\end{split}
\end{equation*}
where $c$ is some constant. Integrating with respect to $s \in (0,\alpha)$ and exploiting the factor $(\alpha -\sigma)$ to compensate the bound the norm of $h^{\alpha-s}$, we deduce \eqref{eq:commutator-taylor-second-order-rd}.

To address the general, we bound separately the terms
\begin{equation}
\label{eq:first-second-block-rd}
 \int_0^\alpha \int u^s [\Delta, b \cdot \nabla ] h^{\alpha-s} d\scrL^d\,  (\alpha -s) ds \quad \text{and} \quad \int_0^\alpha \int  u^s [\Delta, (\Delta a^{i,j}) ] h^{\alpha-s} d\scrL^d \, (\alpha -s) ds.\end{equation}

To deal with former, we isolate a ``leading term'' which involves $\nabla^2 a^{i,j}$ and we bound the remaining terms it by adding and subtracting suitable quantities, with the only difficulty that we must take into account the second order expansion. Precisely, after arguing as in the case $\Delta a^{i,j} = 0$, we are left with estimating
\begin{equation}
\label{eq:what-is-left-diffusions} \int_0^\alpha \int u^s (\Delta a^{i,j}) \Delta  h^{\alpha-s} (\alpha - s) d\scrL^d ds,  \end{equation}
and to this aim we add and subtract
\begin{equation}\label{eq:add-subtract-second-order-rd}  \int_0^\alpha \int u^\alpha (\Delta a^{i,j}) \Delta  h^{\alpha-s} (\alpha - s) d\scrL^d ds =  \int_0^\alpha \int u^\alpha (\Delta a^{i,j}) R_{i,j} \Delta^2 f^{\alpha-s} (\alpha - s) d\scrL^d ds,\end{equation}
where we let $R_{i,j}f := \partial^2_{i,j} \Delta^{-1}f$ be the second-order Riesz transform along the directions $i$, $j$. The difference between the \eqref{eq:what-is-left-diffusions} and \eqref{eq:add-subtract-second-order-rd} is easily bounded and to conclude, we exploit the identity 
\[
\int_0^\alpha \Delta^2 f^{\alpha-s} (\alpha - s)  ds = \int_0^\alpha(\alpha - s)  \partial^2_s f^{\alpha-s}   ds = - f^\alpha + f + \alpha \Delta f^\alpha.
\]
and use the fact that the latter quantity is uniformly bounded (and that $R_{i,j}$ is a bounded operator).

To estimate the second expression in \eqref{eq:first-second-block-rd}, we notice that
%\[  %splitting
%\[ \int u^\sigma [\Delta, (\Delta a)]  h^{\alpha - \sigma} d\scrL^d  = \int  (\Delta u^\sigma) (\Delta a) h^{\alpha - \sigma} d\scrL^d - \int  u^\sigma (\Delta a) \Delta h^{\alpha - \sigma} d\scrL^d,\]
%and recognizing in the right hand side the first term of \eqref{eq:A-B-first-block}, whose contribution is already controlled. We conclude by estimating
%\begin{equation*}\begin{split}
%\abs{ \int  (\Delta u^\sigma) (\Delta a) h^{\alpha - \sigma} d\scrL^d  } &\le \nor{\Delta a}_q \nor{ \Delta u^\sigma}_r \nor{h^{\alpha - \sigma}}_s\\
%&\le c^\Delta_rc^\Delta_s
\[ \int u^s [\Delta, (\Delta a^{i,j}) ] h^{\alpha -s} d\scrL^d = \frac{d}{ds} \int u^{s} (\Delta a^{i,j}) \partial^2_{i,j} f^{\alpha -s} d\scrL^d,\]
thus we integrate by parts with respect to $s \in (0,\alpha)$,
\[ \int_0^\alpha \frac{d}{ds} u^{s} (\Delta a^{i,j}) \partial^2_{i,j} f^{\alpha -s}  (\alpha-s) ds= - \alpha  u (\Delta a^{i,j}) \partial^2_{i,j} f^{\alpha}  + \int_0^\alpha u^{s} (\Delta a^{i,j}) \partial^2_{i,j} f^{\alpha -s}.\]
The first term in the right hand side above is bounded by $c\nor{\Delta a^{i,j}}_{L^q_x} \nor{u}_{L^r_x} \nor{f}_{L^s_x}$. We write
\[ \int_0^\alpha \int u^{s} (\Delta a^{i,j}) \partial^2_{i,j} f^{\alpha -s} d\scrL^d ds = \int_0^\alpha \int u^{s} (\Delta a^{i,j}) R_{i,j} \Delta f^{\alpha - s} d\scrL^d ds.\]
To conclude, we argue once more by adding and subtracting
\[ \int_0^\alpha  \int u^{\alpha} (\Delta a^{i,j}) R_{i,j} \Delta f^{\alpha - \sigma} d\scrL^d ds=  \int u^{\alpha} (\Delta a^{i,j}) R_{i,j} (f^\alpha - f)d\scrL^d,\]
and estimating the differences involved. This settles the validity of \eqref{eq:commutator-taylor-second-order-rd}, for smooth functions and at fixed $t \in (0,T)$. By integration and a density argument, the general case is deduced at once.

Next, we  prove \eqref{eq:commutator-diffusion-infinitesimal-rd}, which follows from the fact that $\alpha \int u [\Delta, a : \nabla^2] u^{2\alpha} d\scrL^d$ is infinitesimal, as $\alpha \downarrow 0$: indeed, a standard uniform boundedness and density argument gives that the left hand side in \eqref{eq:commutator-taylor-second-order-rd} is infinitesimal as $\alpha \downarrow 0$. To show it, we initially argue in the  case of smooth functions $u$, $f$, and for fixed $i$, $j\in \cur{1, \ldots, d}$, we let $b = \nabla a^{i,j}$ and integrate by parts
\begin{equation*}\begin{split}
  \int u [\Delta, a^{i,j}\partial^2_{i,j}] f^\alpha d\scrL^d  & = - 2\int (b \cdot \nabla u) \partial^2_{i,j}  f^\alpha d\scrL^d -  \int u (\Delta a^{i,j}) \partial^2_{i,j} f^\alpha d\scrL^d \\
& = - 2 \int (\partial^2_{i,j} f) \sfP^\alpha (b \cdot \nabla u)  d\scrL^d -  \int u (\Delta a^{i,j}) \partial^2_{i,j} f^\alpha d\scrL^d \\
 & = -2 \int (\partial^2_{i,j} f) \cur{ [\sfP^\alpha, b \cdot \nabla  ] u + (b \cdot \nabla u^\alpha)} d\scrL^d -  \int u (\Delta a^{i,j}) \Delta f^\alpha d\scrL^d.
\end{split}\end{equation*}
Although the intermediate steps require some regularity for $u$, by the commutator estimate for Sobolev derivations established in the previous lemma, the resulting identity extends by continuity to $u \in L^\infty_t(L^r_x)$, $f \in L^\infty_t( W^{2,s}_x)$. Next, we specialize to the case $f := u^\alpha$. By the strong convergence provided by Lemma \ref{lemma:commutator-estimate} and uniform boundedness of $\alpha \partial^2_{i,j} u^\alpha$ in $L^\infty_1(L^r_x)$, we  have
\[ \alpha\abs{\int (\partial^2_{i,j} u^\alpha_t) [\sfP^\alpha, b_t\nabla ] u_t  d\scrL^d} \to 0, \quad \text{in $L^1(0,T)$ as $\alpha \downarrow 0$.}\]
Similarly, it holds (recall that the left hand side in \eqref{eq:smoothing-heat-lpdelta} is infinitesimal)
\[ \alpha \abs{ \int u (\Delta a^{i,j}) \Delta  f^\alpha d\scrL^d } \le \nor{\Delta a^{i,j}}_{L^q_x} \nor{u}_{L^r_x} \| \alpha \partial^2_{i,j} u^{2\alpha}\|_{L^s_x} \to 0.\]
Finally, in order to handle the term $\alpha \int (\partial^2_{i,j} u^\alpha) (b \nabla u^\alpha) d\scrL^d$, the choice $f = u^\alpha$ and the symmetry of $a$ are crucial: we integrate by parts once, and since $b = \nabla a^{i,j}$, we obtain
\[
 \int (\partial^2_{i,j} u^\alpha) b\cdot  \nabla u^\alpha d\scrL^d = - \sum_{k=1}^d \int \sqa{ \partial_i u^\alpha (\partial^2_{j,k}a^{i,j}) \partial_k u^\alpha + (\partial_i u^\alpha) (\partial_k a^{i,j}) \partial^2_{k,j}  u^\alpha }d\scrL^d.
\]
The first term, when multiplied by $\alpha$, is clearly bounded and infinitesimal  as $\alpha \downarrow 0$, so we focus on the last one. To show that it is bounded, we recall that $a$ is symmetric and we are actually interested in bounds for the whole sum on $i$, $j\in \cur{1, \ldots d}$; thus, by coupling the symmetric terms, it is sufficient to prove that
\[  \alpha \int \partial_i u^\alpha(\partial_k a^{i,j}) \partial^2_{k,j}  u^\alpha + \partial_j u^\alpha(\partial_k a^{i,j}) \partial^2_{k,i}  u^\alpha d\scrL^d \]
is infinitesimal. This symmetric expression can be explicitly rewritten  as
\[  \frac{\alpha}{2} \int (\partial_k a^{i,j}) \partial_k \sqa{ (\partial_i u^\alpha + \partial_j u^\alpha)^2 -(\partial_i u^\alpha)^2 - (\partial_j u^\alpha)^2} d\scrL^d,\]
and at this stage we integrate by parts once more, obtaining a bound in terms of
\[ \alpha \nor{ \nabla ^2 a}_{L^q_x} \nor{ \nabla u^\alpha}_{L^r}  \nor{ \nabla u^\alpha}_{L^s},\]
which is sufficient to conclude (recall that the left hand side in \eqref{eq:smoothing-heat-lpgamma} is infinitesimal).
\end{proof}

%n the next lemma, we assume that $a$ is uniformly bounded and non-degenerate, i.e.\ it holds, for some $\lambda >0$,  $a(v, v) \ge \lambda \abs{v}^2$, for every $v \in \R^d$, $\scrL^{1+d}$-a.e.\ in $(0,T) \times \R^d$. 

Finally, we deal with the bounded elliptic case: if $a$ is bounded and elliptic, then the form $L^2_t(W^{1,2}_x) \in f \mapsto \int a(\nabla f, \nabla f)$ is Dirichlet form, with associated Markov semigroup $\sfP^\alpha_a$ and (self-adjoint) generator $\sfDelta_a f = \div( a \nabla f)$, on its ``abstract'' domain $D(\sfDelta_a)$ (as given by the general theory of Dirichlet forms). When we choose $\sfP^\alpha_a$ as ``smoothing  operator'', the main difficulty is to prove that it preserves regularity with respect to $t \in (0,T)$, thus we need some estimate for the commutator  $[\sfP^\alpha_a, \partial_t]$, which we initially define in following the weak sense, for $u \in \cA_c$,  $f \in \cA$:
\[ \int_{(0,T)\times \R^d} u [ \sfP^\alpha_a, \partial_t] f d \scrL^{1+d} := \int_{(0,T)\times \R^d} \sqa{ (\sfP^\alpha_a u) \partial_t f + (\partial_t u)  \sfP^\alpha_a f} d \scrL^{1+d}.\]

\begin{lemma}\label{lemma:commutator-time}
Let $a$ be bounded and elliptic, with $\partial_t a \in L^\infty_{t}(L^\infty_x)$. Then, for every $\alpha \in (0,1)$, $u \in \cA_c$, $f \in  \cA$, it holds
\begin{equation}\label{eq:commutator-time-bilinear}\abs{ \int u [\sfP^\alpha_a, \partial_t ] f d\scrL^{1+d} } \le c \nor{\partial_t a}_{L^\infty_t L^\infty_x} \nor{u}_{L^2_tL^2_x} \nor{f}_{L^2_tL^2_x},\end{equation}
where $c$ is a constant (depending only on the ellipticity constant $\lambda$).\end{lemma}

Thanks to this lemma and a density argument, for $f \in W^{1,2}_t(L^2_x)$, we deduce that $\sfP^\alpha_a f \in W^{1,2}_t(L^2_x)$, and the ``strong'' commutator $[\sfP^\alpha_a, \partial_t ] f  := \sfP^\alpha_a \partial_t f -  \partial_t \sfP^\alpha_a f$ is well defined and it belongs to $L^2_{t}(L^2_x)$. Moreover, the usual uniform boundedness arguments shows that, for $u\in L^2_{t}(L^2_x)$ and any family $(f_\alpha)_{\alpha \ge 0} \subseteq  L^2_t(W^{1,2}_x)$ converging in $L^2_{t,x}$, it holds
\[ \int_{(0,T)\times \R^d} u [\sfP^\alpha_a, \partial_t] f_\alpha d \scrL^{1+d} \to 0, \quad \text{as $\alpha \downarrow 0$.}\]

\begin{proof}
We provide the following analogue of \eqref{eq:commutator-time-bilinear}, where $\partial_t$ is replaced by $\sigma^{-1} (\sfT^\sigma - \mathsf{I})$, where $\sfT^\sigma f(t,x) = f(t+\sigma, x)$, and $\mathsf{I}$ is the identity operator (we also choose $\sigma \neq 0$ small enough, to avoid boundary terms, thanks to the assumption $u \in \cA_c$):
\[ \abs{\int_{(0,T)\times \R^d} u [\sfP^\alpha_a, \sigma^{-1} (\sfT^\sigma - \mathsf{I}) ] f d\scrL^{1+d}} \le c \nor{\partial_t a}_{L^\infty_t L^\infty_x} \nor{u}_{L^2_tL^2_x} \nor{f}_{L^2_tL^2_x},\]
(which $c$ depending on $\lambda$ only). Once this is is settled, we may let $\sigma \to 0$ and pass to the limit in the weak formulation. Let us notice that the identity operator plays no role above, and everything reduces to estimate $\sigma^{-1}\int u [\sfP^\alpha_a, \sfT^\sigma ] f d\scrL^{1+d}$. By first-order interpolation along the semigroup $\sfP^s_a$, for $s \in (0,\alpha)$, it is sufficient to bound the infinitesimal commutator
%Notice that $f \mapsto \sigma^{-1}\sfT_\sigma f$ is a linear continuous operator mapping $L^2 \cap L r$ into itself, with norm smaller than $\sigma^{-1}$. % commutator%Moreover, if $ f$ is bounded in $\Lbm 2 + \Lbm {s'}$, uniformly in $\delta$, then $\partial_t f \in \Lbm 2 + \Lbm{s'}$.
%We argue by duality and establish the bound
%\begin{equation}\label{eq:commutator-time-discrete-bilinear} \abs{\int  f [\sfP^\alpha_a, \delta_t] u d\mmt } \le c \nor{\partial_t \aa}_{r,s}  \nor{u}_{r} \nor{ f}_{s},\end{equation}
%for every $u \in \tilde{\V^r}$, $f \in  \tilde{\V^s}$, where $c$ is some constant depending upon $\lambda$.
%and that
%\[ \int u [\sfP^\alpha_a, \sigma_\delta ] f d\mm\]
%We prove the claim using Lemma \ref{lemma:commutator-estimate}, where we let $\sfP := \sfR$ and $\sfA =\sigma^{-1} \sfT_\sigma$. Indeed, the infinitesimal commutator \eqref{eq:commutator-infinitesimal} reads as $\sigma^{-1}$ times
\[  \int u^{s} [\sfDelta_a, \sfT^\sigma ] f^{\alpha-s} d\scrL^{1+d} =  \int \bra{ (\sfT^{\sigma}a) \nabla u^{s} , \nabla \sfT^{\sigma} f^{\alpha-s}) } - \bra{ a\nabla u^{s} , \nabla \sfT^\sigma f^{\alpha-s} } d\scrL^{1+d},\]
where we performed integration by parts with respect to the variable $x \in \R^d$ and the change of variables $t \mapsto t+\sigma$ in the first integral. We have therefore the bound (we are actually interpolating also along the semigroup $\sigma \to \sfT^\sigma$)
\[ \abs{ \int u^{s} [\sfDelta_a, \sfT^\sigma ] f^{\alpha-s} d\scrL^{1+d}} \le \int_0^\sigma \abs{ \partial_r \int (\sfT^{r}a) (\nabla u^{s} , \nabla \sfT^{\sigma} f^{\alpha-s})  d\scrL^{1+d}} dr \]
which by $\partial_r \sfT^{r}a = \sfT^{r} \partial_t a$ gives the thesis, after an application of H\"older inequality and using the smoothing effect in $L^2_t(L^2_x)$ of $\sfP_a$, i.e.\ $\nor{\nabla u^s}_{L^2_t L^2_x} \le (s\lambda)^{-1/2} \nor{u}_{L^2_tL^2_x}$. 
%Once \eqref{eq:commutator-time-bilinear} is settled, a first consequence is that $\partial_t \sfP^\alpha_a f = [\partial_t, \sfP^\alpha_a] f  + \sfP^\alpha_a (\partial_t f)\in L^2_{t,x}$ and convergence to zero follows from a density and uniform boundedness argument. 
\end{proof}

It would be natural to extend the argument above for more general exponents beyond the case above; the main issue being that a smoothing effect for $\sfP_a$ akin to \eqref{eq:smoothing-heat-lpgamma} is not ensured by Dirichlet form theory, when the exponent involved is different from $2$. It seems plausible however to replace $L^2_t(L^2_x)$ with $L^\infty_t(L^2_x)$ and require only $\partial_t a \in L^1_t(L^\infty_x)$ (as the semigroup acts only fiberwise).

\begin{remark}[trace semigroup at $t=0$]%\label{trace-semigroup}
Another consequence of Sobolev regularity of the lemma above is existence of a ``trace'' semigroup, e.g.\ at $t = 0$, defined as follows: for $f \in L^2_x$, consider a constant extension $f(t,x) = f(x)$  for $(t,x) \in (0,T)\times\R^d$, and let $\sfP^\alpha_0 f$ be the trace of the Sobolev function $\sfP^\alpha_a f$ at $t =0$. Alternatively, this can be obtained as the semigroup generated by the bilinear form given by the trace at $0$ of $a$.
\end{remark}

\subsection{Proof of well-posedness results}\label{sec:proof-well-posedness}

In this section, we address the proof of Theorem \ref{thm:wp-degenerate} and Theorem \ref{thm:wp-elliptic}. As already remarked, existence is easily settled by approximations, so we focus on uniqueness.

\begin{proof}[Proof of Theorem \ref{thm:wp-degenerate}]
Let $u$ be the difference between any two narrowly continuous solutions in $L^\infty_t(L^r_x)$ and let $\sfP^\alpha$ be the heat semigroup on $\R^d$, extended on $(0,T)\times \R^d$ by acting on each fiber $\cur{t} \times\R^d$, $t \in [0,T]$. For $\alpha >0$, $\sfP^\alpha$ maps $W^{1,1}_t(L^{r'}_x) \cap D^{r'}(\cL)$ into itself, as $f^\alpha := \sfP^\alpha f$ is $C^2_b$ with respect to the variable $x \in \R^d$, for $\scrL^1$-a.e.\ $t \in (0,T)$ (to approximate $\sfP^\alpha f$ with functions in $\cA$, we argue by convolution with a smooth kernel with respect to $t \in (0,T)$),  % Actually, the same reasoning shows that, for $\alpha >0$, $\sfP^\alpha$ maps $L^\infty_t(L^{r}_x)$ into $W^{1,1}_t(L^{r'}_x) \cap D^{r'}(\cL)$,
thus \eqref{eq:weak-formulation-extended} holds true for $f$ in such a space:
\begin{equation}\label{eq:weak-formulation-extended-2} \int_0^T \int \sqa{(\partial_t + \cL _t) f} u^\alpha_t d \scrL^d dt = \int f^\alpha_T u_T d\scrL^d - \int f^\alpha_0 u_0 d\scrL^d + \int_{(0,T)\times \R^d}     u  \sqa{ \sfP^\alpha,  (\partial_t + \cL _t) } f d\scrL^{1+d}.\end{equation}

For $\alpha >0$, we also have $u^\alpha \in D^{r'}(\cL)$: to use $u^\alpha$ as a test function, we deduce that $u^\alpha \in W^{1,1}_t(L^{r'}_x)$, which follows directly from the equation satisfied by $u^\alpha$. Indeed, \eqref{eq:weak-formulation-extended-2} for $f \in \cA_c$ entails that the distributional derivative $\partial_t u^\alpha$ coincides with the distribution $\cL^* u^\alpha$, which is represented by a function, namely
\[ \cL^* u^\alpha = \nabla^* (b u^\alpha) + \frac 1 2 (\nabla^*)^2 (a u^\alpha) = - (\div \cL) u^\alpha + \cL u^\alpha  + 
\frac 1 2 (\nabla^*a) \cdot (\nabla u^\alpha) \in L^1_tL^{r'}_x.\]
Therefore, $\partial_t u^\alpha \in L^1_t(L^{r'}_x)$, and $u^\alpha$ admits an absolutely continuous continuous representative, which must coincide with the one that we would obtain by acting directly to the narrowly continuous representative $u_t$ with the heat semigroup $\sfP^\alpha$, at every $t \in [0,T]$:  it holds in particular $u^\alpha_0 = 0$, since $u_0 = 0$. Moreover, the curve $t \mapsto \int_{\R^d} (u^\alpha_t)^2 d\scrL^d$ is absolutely continuous, with distributional and $\scrL^1$-a.e.\ derivative 
$\frac{d}{dt} \int (u^\alpha)^2_t d\scrL^d = 2 \int (\partial_t u^\alpha) u^\alpha d\scrL^d.$

We are in a position to let $u^\alpha$ in the weak formulation \eqref{eq:weak-formulation-extended-2}, to obtain
\[ \int_0^T \int \sqa{(\partial_t + \cL_t) u^\alpha} u^\alpha_t d \scrL^d dt = \int (u^\alpha_T)^2  d\scrL^d  + \int_{(0,T)\times \R^d}   u  \sqa{ \sfP^\alpha,  \cL _t} u^\alpha d\scrL^{1+d}.\]
If we choose instead a test function $t\mapsto f(t) u^\alpha_t$, with $f \in C^1_c[0,T)$ and we apply \eqref{eq:ibp-diffusion}, we eventually deduce the inequality
\[ \frac{d}{dt} \int (u^\alpha)^2_t d\scrL^d \le \nor{\div \cL_t^-}_{L^\infty_x} \int_{\R^d} (u_t^\alpha)^{2} d\scrL^d + \int_{\R^d}  u_t \sqa{ \sfP^\alpha,  (\partial_t + \cL _t)} u^\alpha_t d \scrL^{d}, 
\]
$\scrL^1$-a.e.\ $t \in (0,T)$ and in the sense of distributions on $(0,T)$. Gronwall lemma gives
\[ \nor{ u^\alpha }_{L^\infty_t L^2_x}^2 \le \exp\cur{ \nor{ \div \cL^-}_{L^1_tL^\infty_x}} \int_0^T \abs{ \int_{\R^d}  u_t \sqa{ \sfP^\alpha,  (\partial_t + \cL _t) } u^\alpha_t  d \scrL^{d} } dt .\]
As a consequence of Lemma \ref{lem:commutator-diffusion-rd}, we deduce $\nor{ u }_{L^\infty_t L^2_x} \le \liminf_{\alpha \downarrow 0 }  \nor{ u^\alpha }_{L^\infty_t L^2_x}  = 0$.\end{proof}

\begin{proof}[Proof of Theorem \ref{thm:wp-elliptic}]
%First, we argue that uniqueness holds in the smaller space $L^\infty_t(L^2_x) \cap L^2_t(W^{1,2}_x)$. We may follow e.g.\ the argument in \cite[Theorem 4.3]{figalli-sdes}, but let us argue by the same smoothing scheme as above, $\sfP$ being the heat semigroup acting on each fiber. Indeed, \eqref{eq:weak-formulation-extended-2} holds in this setting as well
%We aim at showing that any solution $u$ (as given in the statement) actually belongs to the space $L^2_t(W^{1,2}(\R^d))$, where well-posedness follows by well-known PDE results, see e.g.\ Step 2 in \cite[Theorem 4.3]{figalli-sdes}. %Moreover, as in we preliminarily subtract from $u$ the unique solution in $L^2_t(W^{1,2}(\R^d))$ of the equation
%\[ \partial_t v = -\div (bu) + ( \nabla^*)^2(av), \quad \text{ in $(0,T) \times \R^d$, $v_0 = u_0$,}\]
%so that we are reduced to a narrowly continuous solution belonging to $L^2_{t,x}$ to the FPE associated to the diffusion operator whose drift is identically zero and null initial condition.
In our smoothing scheme, we choose $\sfP^\alpha = \sfP_a^\alpha$ be the semigroup associated to the Dirichlet form $f \mapsto \int a(\nabla f, \nabla f) d\scrL^{1+d}$, as introduced in the previous section. A first step consists in showing that \eqref{eq:weak-formulation-extended-2} holds true, and we see it as a consequence of the fact that $\sfP^\alpha_a$ maps $W^{1,2}_t(L^2_x) \cap D^2(\cL)$ into itself: if $f \in W^{1,2}_t(L^2_x)$, then Lemma \ref{lemma:commutator-time} shows that $f^\alpha \in W^{1,2}_t(L^2_x)$ as well; to show $f^\alpha \in D^2(\cL)$, we rely on the assumption on $a \in L^1_t(W^{1,p}_x)$, and show that the smooth approximations obtained by means of the standard heat semigroup $\sfP^s( f^\alpha )$ converge towards $f^\alpha$ in $D^2(\cL)$, i.e.\ $\cL \sfP^s ( f^\alpha ) \to \cL ( f^\alpha )$ in $L^1_t(L^2_x)$ (this is the only point where we use the first order regularity assumption on $a$). Such convergence can be seen by the commutator lemma for Sobolev vector fields, Lemma \ref{lemma:commutator-estimate}, noticing that the claim convergence amounts to show
\[ [ \cL , \sfP^s ] f^\alpha \to 0 \text{ in $L^1_t(L^2_x)$,}\]
but since derivatives and the standard heat semigroup commute, it holds
\[ [ \cL , \sfP^s ] f^\alpha = \sum_{i,j=1}^d [ a_{i,j} \partial_i , \sfP^s] \partial_j f^\alpha +  \sum_{i=1}^d [ b_i, P^s] \partial_i f^\alpha\to 0\]
since $\partial_j f^\alpha \in L^\infty_tL^2_x$ and Lemma \ref{lemma:commutator-estimate} shows convergence towards $0$ in $L^1_tL^2_x$, as $s \downarrow 0$.

As a second step, we notice that we may let $u^\alpha$ be a test function in \eqref{eq:weak-formulation-extended-2} indeed, it holds $u^\alpha \in H^{1,2}(\cL, a(\nabla \otimes \nabla))$ by what we just proved, while the fact that $\partial_t u^\alpha$ is represented by some function in $L^1_tL^2_x$ follows from a duality argument: for a.e.\ $t \in (0,T)$ the linear functional $f \mapsto \int_{\R^d} u_t \cL_t f^\alpha$ is bounded in $L^2_x$. From \eqref{eq:weak-formulation-extended-2}, we have
\[ \partial_t \int (u^\alpha)^2_t d\scrL^d + 2 \lambda \int \abs{ \nabla u^\alpha }^2 d\scrL^d \le \int \sqa{ (u_t^\alpha)^{2}  (\nabla^*)^2 a_t + u_t b_t\nabla u^\alpha_t + u_t \sqa{ \sfP^\alpha_a,  \partial_t }u^\alpha_t }d \scrL^{d},\]
where we applied \eqref{eq:ibp-diffusion} only for the diffusion part $a: \nabla^2$, as we deal with the drift term separately, using the inequality \[ \abs{u_t b_t \cdot \nabla u^\alpha} \le  \lambda \abs{\nabla u^\alpha}^2 + {4\lambda}^{-1} \abs{ u_t}^2 \abs{b_t}^2, \]
to bound the contribution of the drift part. To conclude, we apply Gronwall inequality and finally let $\alpha \downarrow 0$, using \eqref{eq:commutator-diffusion-infinitesimal-rd} to deduce that the commutator term gives no contribution in the limit and uniqueness holds.
\end{proof}

\appendix
\section{The superposition principle for multidimensional diffusions}\label{chap:superposition-rd}

%In this chapter, we establish the superposition principle for rather general diffusions in $\R^d$, %  where ``lifting'' means that the $1$-marginal $\eta_t := (e_t)_\sharp \eeta$ coincides with $\nu_t$, for $\scrL^1$-a.e.\ $t \in (0,T)$ (actually, for every $t \in [0,T]$, when considering the $\nu$ is narrowly continuous).
%in particular under minimal regularity and no ellipticity assumptions on coefficients: this both settles the equivalence results in the previous chapter and provides a rigorous foundation on which we build our deductions in the metric measure space setting, in Chapter \ref{chap:superposition-mms}.%  exposition for readers that prefer not to deal at all with the abstract setting and at the same time a rigorous foundation on which

To prove Theorem \ref{thm:sp}, we follow a general scheme, whose structure is shared by many proofs of superposition principles appearing in the literature, see e.g.\ \cite[Theorem 8.2.1]{ambrosio-gigli-savare-book}, \cite[Theorem 12]{ambrosio-crippa-lecture-notes}, \cite[Theorem 4.5]{ambrosio-figalli-wiener}, \cite[Theorem 2.6]{figalli-sdes}, \cite[Theorem 7.1]{ambrosio-trevisan}, that we summarize below. The derivation is rather elementary, although the ``right'' underlying framework would that of Young (or random) measures. For simplicity, we let $T=1$ in this section. Let $\nu = (\nu_t)_{t \in [0,1]} \subseteq \scrP( \R^d)$  be a narrowly continuous weak solution of the FPE \eqref{eq:fpe-rn}. To deduce existence of a superposition solution for $\nu$, we perform the following steps.

\vspace{.5em}

\noindent {\bf Step 1} (approximation){\bf.}\ We build from $\nu$ a sequence of solutions $(\nu^n)_n$ of FPE's associated to diffusion operators $(\cL^n)_n$, for which the superposition principle is already known to hold, thus obtaining a sequence of superposition solutions $(\eeta^n)_n$ of MP's. Here, the difficulty is to exhibit a sufficiently good approximation, so that $\nu^n$ converge towards $\nu$, e.g., narrowly, and $\cL^n$ towards $\cL$, in a sense to be made precise, as $n \to \infty$.

\vspace{.5em}

\noindent {\bf Step 2} (tightness){\bf.}\ We prove that $(\eeta^n)_n \subseteq \scrP ( C([0,1]; \R^d))$ is tight, yielding a narrow limit point $\eeta$. By Ascoli-Arzel\`a criterion, this step reduces to show uniform bounds on the modulus of continuity of the canonical process $(e_t)_{t \in [0,1]}$ with respect to $\eeta^n$.

\vspace{.5em}

\noindent {\bf Step 3} (limit){\bf.}\ From convergence $\nu^n \to \nu$, $\cL^n \to \cL$, as $n \to \infty$, we conclude that $\eeta$ is a superposition solution for $\nu$. Here, the problem is to deal with convergence for possibly non-continuous functions, as they involve the coefficients $a$, $b$.

%The reasoning can be be split into three steps. we approximate $(\nu_t)_t$ by a family of measures that solve some related equations having cylindrical smooth coefficients. Then, we lift these approximations and show tightness in $C([0,T], \R^\infty)$. Finally, we show that any limit point provides some solution $\llambda$ as required.

%{.5em}

%In the next subsections, we investigate separately some specific results in order to deal with martingale problems associated to diffusion operators.

\subsection{Approximation}
\label{sec:approximation}

We  approximate the limit solution by means of mollification by convolutions or push-forwards via smooth maps (in probabilistic jargon, by conditioning with respect to some observables). %In this section we prove some useful features of these transformations. Let us point out that more general approximations can be devised, e.g.\ by general probability kernels. However, the commutator between $\cL$ and the chosen kernel plays an important role, and it seems difficult (perhaps worthless) at this stage to look for further variants, as the two strategies discussed above are sufficient for our purposes.

%Let us deal with the cylindrical approximation first. We consider the following general problem: if $\nu = (\nu_t)_t$ is a weak solution to FPE on $X$ and $\pi: X \to Y$ is Borel map ($Y$ is also Polish) is $(\pi_\sharp \nu_t)$ is a weak solution to some FPE on $Y$?

%A formal computation entails that by defining
%\[ \pi(\cL) f (t,x):=   \frac{ d \pi _{\sharp} \sqa{ \cL_t(f \circ \pi)  \nu_t} }{d \pi_{\sharp} \nu_t},\]
%then $\pi_{\sharp} \nu_t$ is a solution to FPE associated to $\pi(\cL)$. Indeed, it holds
%\begin{equation*}\begin{split}
%\%partial_t \int f d \pi_\sharp (\nu) &= \partial_t \int f\circ \pi d \nu \\
%&= \int  \cL_t(f \circ \pi)  \nu_t \\
%&= \int d pi _{\sharp} \sqa{ \cL_t(f \circ \pi)  \nu_t} \\
%&= \int \pi(\cL) f d \pi_\sharp \nu.
%\end{split}
%\end{equation*}

%Although being general, in view of the proof of a superposition principle for diffusions, this transformation is not useful, as the operator $\pi(\cL)$ is not a diffusion\footnote{a part from specific situations, e.g.\ when $\pi$ is injective}
%\vspace{1em}

{\bf Push forward via smooth maps.} This technique is inspired by the approach in \cite[Theorem 7.1]{ambrosio-trevisan}. Let $\pi \in C^2(\R^d; \R^d)$ $\pi = (\pi^1, \ldots, \pi^d)$, with uniformly bounded first and second derivatives. Then, it is possible to define a diffusion operator $\pi(\cL)$ on $\R^d$ such that $\pi_\sharp \nu := (\pi_\sharp \nu_t)_{t \in [0,1]}$ is a solution to the associated FPE (in duality with $\cA = C^{1,2}_b((0,T)\times \R^d)$). Indeed, the composition $f\circ \pi(t,x)  := f(t, \pi(x))$ belongs to $\cA$, and if we let $f \circ \pi$ in the weak formulation \eqref{eq:weak-fpe-rn}, the chain rule gives
\[ \cL (f \circ \pi) = \sum_{i=1}^d \cL(\pi^i) \sqa{ (\partial_i f) \circ \pi }+ \frac 1 2 \sum_{i,j=1}^d a(\nabla \pi^i, \nabla \pi^j) \sqa{ (\partial_{i,j} f) \circ \pi}.\]
We define, for $(t,x) \in (0,T) \times \R^d$,
\[ \pi(a)^{i,j}_t (x) :=  \E_{\nu_t}\sqa{ a(\nabla \pi^i, \nabla \pi^j) \, | \, \pi = x} = \frac{d \pi_\sharp\sqa{ a(\nabla \pi^i, \nabla \pi^j) \nu_t}}{ d \pi_\sharp \nu_t}(x), \quad \text{for $i,j \in \cur{1, \ldots, d}$,}\]
\[  \pi(b)^{i}_t (x) :=  \E_{\nu_t}\sqa{ \cL (\pi^i)\,  |\,  \pi = x} = \frac{d \pi_\sharp\sqa{ \cL(\pi^i)\nu_t}}{ d \pi_\sharp \nu_t}(x), \quad \text{for $i \in \cur{1, \ldots, d}$.}\]
%To be rigorous, we actually consider a Borel representative for the functions defined above.
Then, $\pi(\cL) := \cL(\pi(a), \pi(b))$ is a diffusion operator on $\R^d$ and the (narrowly continuous) curve of measures $\pi_\sharp \nu$ is a weak solution of the FPE $\partial_t \pi(\nu) = \pi(\cL)^* \pi(\nu)$, in $(0,T)\times \R^d$. Let us remark that $\pi(\cL)$ depends upon $\nu$, although it is not evident in the notation.

Since conditional expectations reduce norms and the derivatives of $\pi^i$ are uniformly bounded, integral bounds on $a$, $b$ are naturally transferred on $\pi(a)$, $\pi(b)$ (in particular, uniform bounds). However, local integrability conditions could be not preserved.

{\bf Mollification by convolutions.} This is a more standard technique, already employed e.g.\ in \cite[Theorem 8.2.1]{ambrosio-gigli-savare-book} and \cite[Theorem 2.6]{figalli-sdes}. Let $\rho \ge 0$ be a smooth probability density (with respect to$\mathscr{L}^d$), with full support. Then, the family of measures $\nu * \rho := (\nu_t * \rho)_{t \in [0,1]}$, solve a FPE associated to a suitably defined diffusion operator. Indeed, for $f \in \cA$, it holds $f * \rho \in \cA$ with
\[ \cL (f * \rho) = \sum_{i=1}^d b^i \partial_i (f * \rho) + \frac 1 2 \sum_{i,j=1}^d a^{i,j} \partial_{i,j} (f * \rho)  
=\sum_{i=1}^d b^i (\partial_i f) * \rho + \frac 1 2  \sum_{i,j=1}^d a^{i,j} (\partial_{i,j} f) * \rho,\]
since derivatives and convolution commute. We define
\[ (a^\rho)_t^{i,j} :=  \frac{ d(a^{i,j} \nu_t) * \rho}{ d( \nu_t * \rho)}, \quad (b^\rho)^{i}_t := \frac{  d(b^i \nu_t) * \rho}{ d (\nu_t * \rho)}, \quad \text{for $i,j \in \cur{1,\ldots, d}$.}\]
so $(\nu_t * \rho)_{t \in [0,1]}$ is a weak solution of the FPE associated to $\cL^\rho := \cL(a^\rho, b^\rho)$, as
\[ \partial_t \int  f  d (\nu * \rho) =   \int (\partial_t f) * \rho \, d \nu = \int \partial_t (f*\rho) \, d\nu  
= -\int  \cL (f*\rho) \,  d\nu  = -\int \cL^\rho f \, d (\nu * \rho).
\]

Integrability and regularity properties of $a^\rho$ and $b^\rho$ are collected by the following lemma, see \cite[Lemma 8.1.10]{ambrosio-gigli-savare-book} for a detailed proof.

\begin{lemma}\label{lemma:approximation}
Let $\rho$ be a smooth probability kernel on $\R^d$ with $\rho > 0$  everywhere and $\abs{\nabla^i \rho} \le C \rho$, for $i \in \cur{1,\ldots k}$,  for some constant $C \ge 0$. Let $\mu$, $\nu \in \scrM^+(\R^d)$, with $\mu \ll \nu$.
Then, it holds $\mu * \rho \ll \nu *\rho$, and the function
\[
 \frac{ d (\mu * \rho)}{ d (\nu * \rho)}(x) = \frac{ \int \rho(x-y) \, d\mu(y) }{\int \rho(x-y) \, d\nu(y) }, \quad \text{for $x \in \R^d$}
\]
provides a $C^k(\R^d)$ version of the density $d(\mu * \rho)/ d(\nu *\rho)$. Moreover, for every convex, lower semicontinuous function $\Theta: \R \to [0,\infty]$, it holds
\begin{equation}\label{eq:jensen} \int \Theta\bra{ \frac{ d(\mu * \rho)}{ d(\nu * \rho)} } d(\nu *\rho) \le \int \Theta\bra{ \frac{ d\mu}{ d\nu} } d\nu.\end{equation}

Similar conclusions hold when $\mu = (\mu_t)_{t \in [0,1]} \subseteq \scrM^+(\R^d)$ is a Borel curve and $\nu = (\nu_t)_{t\in [0,1]} \subseteq \scrP(\R^d)$ is narrowly continuous, with $\mu_t \ll \nu_t$ for every $t \in [0,1]$. In addition, it holds
\[ \sup_{t \in [0,1]} \nor{ \frac{d (\mu_t * \rho)}{ d (\nu_t * \rho)}}_{C^k_b(B)} < \infty, \]
for every open bounded set $B \subseteq \R^d$.
\end{lemma}

When applied to a  solution $\nu = (\nu_t)_{t \in [0,1]}$ of the FPE \eqref{eq:fpe-rn} we deduce that, if $a$, $b \in L^p(\nu)$, then $a^\rho$, $b^\rho \in L^p( \nu * \rho)$ (for $p \in [1,\infty]$) and $a^\rho_t$, $b^\rho_t$ are $C^k(\R^d)$, uniformly in $t \in [0,1]$, with uniformly bounded first and second (spatial) derivatives on compact sets of $[0,T] \times \R^d$.

\subsection{Tightness}\label{sec:tightness}

We prove a compactness criterion for solutions of martingale problems, under minimal integrability conditions on the coefficients. %By Ascoli-Arzel\`a theorem, tightness is achieved by estimating the modulus of continuity of solutions to martingale problems.
In the deterministic case, tightness is achieved by estimating the metric velocity of absolutely continuous curves which solve the ODE; in the stochastic case, we rely on analogous results for martingales, using Burkholder-Davis-Gundy inequalities and an argument reminiscent of L\'evy's modulus of continuity for the Brownian motion, to estimate the modulus of continuity of the canonical process (yielding in some cases H\"older regularity).

%We formulate a general result valid for semimartingales.

\begin{theorem}\label{theorem:basic-burkholder}
Let $\theta$, $\Theta_1$, $\Theta_2: [0,+\infty) \to [0,+\infty)$ be functions with $\Theta_1$, $\Theta_2$ convex, l.s.c., \[ \lim_{x \to \infty} \theta(x) = \lim_{x\to +\infty} \frac{\Theta_1(x)}{x} =  \lim_{x\to +\infty} \frac{\Theta_2(x)}{x} = \infty\]
and, for some constant $C \ge 0$, $\Theta_2(2x) \le C \Theta_2(x)$, for $x \ge 0$. Then, there exists some coercive function $\Psi : C([0,1]; \R) \to [0,+\infty]$ such that, for every filtered probability space $(\Omega, (\cF_t)_{t \in [0,1]}, \P)$ and progressively measurable processes $\vphi = (\vphi_t)$,  $\beta = (\beta_t)_t$, $\alpha = (\alpha_t)_t$ with %and, for some $p \in (1,\infty)$, let
%\begin{equation}\label{eq:lp-norms-bound-probabilistic}  \int_0^T \E\sqa{\abs{\ell_t}^{p}  + \alpha_t^{p} } dt < \infty.\end{equation}
\[ [0,1] \ni t \mapsto M_t := \vphi_t - \int_0^t \beta_s \, ds, \quad \text{and} \quad [0,1] \ni t \mapsto M_t^2 - \int_0^t \alpha_s\,  ds\]
$\P$-a.s.\ continuous local martingales, and $\alpha \ge 0$ $\P$-a.s.\ it holds
\begin{equation}\label{eq:coercivity-refined}  \E\sqa{ \Psi( \vphi ) } \le \E\sqa{ \theta(\vphi_0)  + \int_0^1 \sqa{ \Theta_1 \bra{\abs{\beta_t}} + \Theta_2 \bra{\alpha_t}} dt}.\end{equation}
\end{theorem}

\begin{proof}
We prove separately the existence of functions $\Psi_1$ and $\Psi_2$, depending respectively on $\Theta_1$, $\Theta_2$ only, taking integer and non-negative values, such that, if we let
\begin{equation} \label{eq:psi-definition}\Psi(\gamma) := \theta(\abs{\gamma_0}) + \inf_{\gamma^1 + \gamma^2 = \gamma} \cur{ \Psi_1(\gamma^1) + \Psi_2(\gamma^2)},\end{equation}
then $\Psi$ is coercive and \eqref{eq:coercivity-refined} holds. For every $\veps >0$, $i \in \cur{1,2}$ we let $\delta_{i} = \delta_{i,\veps}$ be the largest number in the form $1/n$, (with $n\ge 1$ natural) such that $\Theta_i(\veps^{i}/ \delta) \delta \ge \veps ^{-1}$: such a choice is possible because $\lim_{x \to +\infty} \Theta_i(x)/x = \infty$. %(notice also that $\veps \mapsto \delta_{i,\veps}$ is monotone increasing).
Then, we introduce the closed sets
\begin{equation}\label{eq:set-modulus-contuniuity} A_i(\veps) :=  \Big\{ \gamma \in C([0,1], \R) \, : \, \sup_{k =1, \ldots, \delta^{-1}_i} \, \sup_{s \in \sqa{(k-1) \delta_i, k\delta_i}}  |\gamma_s - \gamma_{(k-1) \delta_i} | \le \veps \Big\},\end{equation}
and we let $\Psi_{i} (\gamma) := \sum_{m\ge 0} (m +1) \chi_{ A_i(2^{-m})^c}$. By construction, $\Psi(\gamma) \le m$ entails $\gamma \in A_i(2^{-k})$, for every $k \ge m$.

To show that $\Psi$ defined by \eqref{eq:psi-definition} is coercive, it is sufficient to apply Ascoli-Arzel\`a criterion, noticing that $\gamma \in \cur{\Psi \le m}$ can be decomposed as the sum of two curves $\gamma_1+\gamma_2$, and $\gamma_i$ ($i\in \cur{1,2}$) admits the following modulus of continuity
\begin{equation}\label{eq:modulus-continuity} \omega_{i,m}(x) := \left\{ \begin{array}{ll}
         2^{1-k} & \text{if $x \in [\delta_{i,2^{-(k+1)}}, \delta_{i,2^{-k}})$ with $k \ge m$,}\\
         2^{1-m}/ \delta_{i,2^{-m}} & \text{if $x \in [\delta_{i,2^{-m}}, +\infty)$}.\end{array} \right. \end{equation}
             
To show that \eqref{eq:coercivity-refined} holds, we assume that the right hand side therein is finite. 
%\begin{equation}\label{eq:condition-finite-theta}
% \E\sqa{ \theta(\vphi_0)  + \int_0^1  \Theta_1 \bra{\abs{\beta_t}} + \Theta_2 \bra{\alpha_t}dt} < \infty. \end{equation}
 The assumptions entail therefore that $(M_t)_t$ is a $\P$-a.s.\ continuous local martingale, whose quadratic variation process is $t \mapsto \int_0^{t} \alpha_s ds$. If we let $\gamma^1_t := \int_0^t \beta_s ds$ and $\gamma^2_t = M_t$, for $t \in [0,1]$, then the left hand side in \eqref{eq:coercivity-refined} is smaller than
\[ \E\sqa{ \theta(\vphi_0) + \Psi_1(\gamma^1) + \Psi_2( \gamma^2  ) } \le \E\sqa{ \theta(\vphi_0)} + \sum_{m \ge 0} (m+1) \E\sqa{\chi_{ A_1(2^{-m})}  \circ \gamma^1 + \chi_{ A_2(2^{-m})}  \circ \gamma^2}.  \]
Next, we focus on the addends in the series above, writing for brevity $\veps$ in place of $2^{-m}$. For $i \in \cur{1,2}$, using \eqref{eq:set-modulus-contuniuity}, we have
\begin{equation*}\E\sqa{\chi_{ A_i(\veps)^c}  \circ \gamma^i}  = P\Big( \sup_{k =1, \ldots, \delta^{-1}_i} \, (\gamma^i)^*_k > \veps \Big) \le \sum_{k=1}^{\delta^{-1}_i} P\bra{ (\gamma^i)^*_k> \veps},\end{equation*}
where we write, $(\gamma^i)^*_k := \sup_{s \in \sqa{(k-1) \delta_i, k\delta_i}} |\gamma_s^i - \gamma_{(k-1) \delta_i}^i |$.

%\begin{equation*}\E\sqa{\chi_{ A_i(\veps)^c}  \circ \gamma^i}  = P\Big( \sup_{k =1, \ldots, \delta^{-1}_i} \, (\gamma^i)^*_k > \veps \Big) \le \sum_{k=1}^{\delta^{-1}_i} P\bra{ (\gamma^i)^*_k> \veps}.\end{equation*}
Let us focus on the case $i=1$ (thus we write $\delta = \delta_1$, $\Theta = \Theta_1$). Since $\abs{\gamma_s - \gamma_t} \le \int_s^t \abs{\beta_r} dr$, we estimate 
\[
 P\bra{ (\gamma^1)^*_k> \veps} \le \frac{ \E\sqa{ \Theta\bra{ \frac{1}{\delta} \int_{(k-1)\delta}^{k\delta} \abs{\beta_s} ds } } }{ \Theta\bra{ \veps / \delta }}
 \le \veps  \E\sqa{ \int_{(k-1)\delta}^{k\delta}  \Theta\bra{\abs{\beta_s}} ds  },
\]
where the last inequality is a consequence of Jensen's inequality and our preliminary choice for $\delta$. Summing upon $k \in \cur{1, \dots, \delta^{-1}}$, we conclude that $\E\sqa{\Psi \circ \gamma} \le c \int_0^1 \E\sqa{\Theta\bra{\abs{\beta_s}}}ds$, for some constant $c \ge 0$ (in this case, the constant does not even depend upon $\Theta$).

To deal with the case $i=2$ (again, we omit to specify $i$ in what follows), i.e., the martingale part, for each $k \in \cur{1, \dots, \delta^{-1}}$, we estimate from above,
\[  P\bra{ M^*_k> \veps}\le \frac{ \E\sqa{ \Theta\bra{ (M^*_k)^2 / \delta } } }{ \Theta\bra{ \veps^2 / \delta }} \le c_{\Theta} \frac{ \E\sqa{ \Theta\bra{ \frac{1}{\delta}\int_{(k-1)\delta}^{k\delta} \alpha_s ds } } }{ \Theta\bra{ \veps^2 / \delta }},\]
where $c_\Theta$ is some constant depending on $\Theta$ only: indeed, it is sufficient to apply Burkh\"older-Davis-Gundy inequalities, e.g.\ in the form \cite[Theorem 2.1]{lenglart-lepingle-pratelli}, to the martingale $M_s := \delta^{-1/2} M_{s+(k-1)\delta}$, $s\in [0,\delta]$ and the convex function with ``moderate growth'' $x \mapsto \Theta(x^2)$. By Jensen's inequality and our definition of $\delta_\veps$ we conclude that
\[ \frac{  \E\sqa{ \Theta\bra{ \frac{1}{\delta}\int_{(k-1)\delta}^{k\delta} \alpha_s ds } }}{\Theta\bra{ \veps^2 / \delta } } \le \veps \E\sqa{ \int_{(k-1)\delta}^{k\delta} \Theta \bra{\alpha_s} ds}.\]
As in the previous case, by summing upon $k \in \cur{0, \ldots, \delta^{-1}}$, we deduce that
\[ \E\sqa{\chi_{ A(\veps)^c}  \circ M} \le  \veps c_\Theta \E\sqa{ \int_0^1 \Theta \bra{\alpha_s} ds}\]
and so we deduce the desired bound for $\E\sqa{\Psi(M) }$. %by the product of a constant (which depends upon $\Theta$ only) times $\E\sqa{ \int_0^1 \Theta \bra{\alpha_s} ds}$.
\end{proof}

%The proof above also entails the following result on the modulus of continuity for $\vphi$.

\begin{corollary}
In the situation of the theorem above, let $\Theta_1(x) = \abs{x}^{p_1}$ and $\Theta_2(x) =  \abs{x}^{p_2}$, for $p_1$, $p_2 \in (1,\infty)$. Then, for every $r >0$ with $r< r(p_1, p_2) := \min\cur{ 1 - \frac {1}{p_1}, \frac{1}{2} \bra{1- \frac{1}{p_2}} }$, it holds
\[ \P\bra{ \limsup_{h \downarrow 0} \sup_{ \abs{t-s} \le h }  \frac{\abs{ \vphi_t - \vphi_s}}{ \abs{t-s}^{r} } =  0 } = 1.\]
\end{corollary}

\begin{proof}
It is sufficient to $\delta_i := \veps^{ 1/r }$, for $i \in \cur{1,2}$. Thanks to this choice, %it holds
%\[ \Theta_i\bra{ \frac{\veps^i }{\delta_i} } \delta_i  = \veps^{i p_i}\delta_i^{1 - p_i} =\veps^{i p_i}\veps^{ - (p_i-1)/r } = \veps^{i( p_i - (p_i-1)/(i r)  ) }, \]
%and the same deductions still holds, for $p_i - (p_i-1)/(ir) >0$ and 
the probabilities of $A_i (2^{-m})^c$ decay sufficiently fast as $m \to \infty$ so that, by Borel-Cantelli lemma, there exists $P$-a.s. some $m \ge 1$ such that the curve $(\vphi_t)_{t \in [0,1]}$ can be written as a sum of two curves having $\omega_{i,m}$, defined in \eqref{eq:modulus-continuity}, as a modulus of continuity. This entails $r$-H\"older estimates for $\vphi$: since the %The choice of $\delta_i$ entails that both these curves are $r$-H\"older, thus
%\[ P\bra{ \limsup_{h \downarrow 0} \sup_{ \abs{t-s} \le h } \frac{\abs{ \vphi_t - \vphi_s}}{ \abs{t-s}^{r} } < \infty } = 1.\]
%To replace $\infty$ with $0$ above, we use the fact that 
 condition on $r$ is open-ended, thus arguing with a $\tilde{r}$ slightly larger than $r$, the thesis follows.
\end{proof}

It is not clear if $\vphi$ in the previous lemma is actually $\P$-a.s.\ H\"older continuous with exponent $r(p_1, p_2)$: one might exploit the existence of functions $\tilde{\Theta}_i$ ($i \in \cur{1,2}$) with $\tilde{\Theta}_i(x)/\abs{x}^{p_i} \to \infty$ as $x \to \infty$, and the right hand side in \eqref{eq:coercivity-refined} still finite, but it seems not sufficient.

%When specialized to the martingale problem, Theorem \ref{theorem:basic-burkholder} entails the following criterion.

\begin{corollary}\label{coro:regularity-paths-martingale-general}
Let $a$, $b$ be Borel maps as in \eqref{eq:a-b-rn}, let $\eeta  \in \scrP(C([0,1]; \R^d))$ be a solution of the martingale problem associated to $\cL(a,b)$. For any $\theta$, $\Theta_1$ and $\Theta_2$, as in the theorem above, let $\Psi$ be the associated coercive functional. Then, for every $f \in C^{1,2}_b((0,T)\times\R^d)$, it holds
%\begin{equation}
%\label{eq:p_1-p_2-bound-martingales}
\begin{equation*} %\label{eq:bound-coercive-martingale} 
\int  \Psi( f_t \circ e_t ) d\eeta \le \int \theta(\abs{f_0}) d\eta_0 + \int_0^1  \sqa{ \Theta_1( \abs{ \cL_t f} ) + \Theta_2( a_t(\nabla f_t, \nabla f_t) ) } d\eta_t dt.\end{equation*}
\end{corollary}

\begin{proof}
We prove that $t \mapsto \vphi_t := f_t \circ e_t$ satisfies the assumptions of Theorem \ref{theorem:basic-burkholder}, with $\beta_t := (\partial_t + \cL_t)f_t  \circ e_t$ and $\alpha_t := a_t( \nabla f_t, \nabla f_t ) \circ e_t$. For simplicity of notation, we omit to write $e_t$ below (since its appearance is quite natural). Since both $f$ and $f^2 \in C^{1,2}_b((0,T)\times\R^d)$, both
\[ t \mapsto M^f_t := f_t - f_0  - \int_0^t (\partial_s+\cL_s) f\,  ds, \quad \text{ and } \quad  t \mapsto M_t^{f^2} := f_t^2 - f_0^2- \int_0^t (\partial_s+ \cL_s) f^2\,  ds\]
are martingales. By developing $(M_t^f)^2$, we see that
\[ t \mapsto (M_t^f)^2 - \int_0^t (\partial_s + \cL_s )f^2 \, ds - 2 \int (\partial_s+ \cL_s )f \sqa{ \int_s^t (\partial_r + \cL_r )f dr - f_t }ds\]
is also a martingale. We add and subtract the process $t \mapsto 2\int_0^t f_s (\partial_t +\cL)f_s ds$, thus
\[  t \mapsto (M_t^f)^2 - \int_0^t \alpha_s \, ds + 2\int_0^t (\partial_s +\cL_s) f \sqa{ f_t - f_s - \int_s^t (\partial_r + \cL_r) f dr} ds\]
is a martingale. To conclude, we notice that
\[ t \mapsto \int_0^t (\partial_s +\cL_s) f \sqa{ f_t - f_s - \int_s^t (\partial_r + \cL_r) f\,  dr}ds = \int_0^t \beta_s \bra{M_t^f - M_s^f} ds \]
is a local martingale (see also \cite[Theorem 1.2.8]{stroock-varadhan}). Indeed, by a stopping time argument, we are easily reduced to the case where $M^f$ is replaced by a martingale $M$ with $M_1 \in L^{\infty}(\P)$, thus $\int_0^t \beta_s \bra{M_t - M_s} ds \in L^1(\P)$, for $t \in [0,1]$. To prove that increments are orthogonal, we let $t \in [0,1]$ and show that
\[ \E\sqa{\int_0^1 \beta_s \bra{M_1 - M_s} ds \, |\,  \cF_t} = \int_0^t \beta_s \bra{M_t - M_s}ds.\]
By the integrability assumptions, we exchange between conditional expectation and integration. The thesis follows by direct consideration of the cases, $s \in [0,t]$ and $s \in (t,1]$.
\end{proof}

\subsection{Limit}\label{sec:weak-limit-martingales}

In the third step, we assume that the probability measures $(\eeta^n)_n$, obtained as superposition solutions for a suitable approximating sequence $(\nu^n)_n$ narrowly converge in $\scrP (C([0,T]; \R^d))$ towards some limit $\eeta$. The fact that $\eeta$ provides a superposition solution for $\nu$ is not straightforward, since we must deal with a limit in the weak formulation, where terms involving the coefficients $a$, $b$ appear (in general, they are not continuous).%The strategy that we sketch below relies on density arguments and depends explicitly on the chosen approximation procedure.

%Let us assume that $\Pi^k$ converge (in some sense to be made precise) towards a kernel $\Pi$. We would like to deduce that $\llambda$ solve the MP associated with $\Pi \cL$.
%be a sequence of probability kernels converging towards a kernel $\Pi$ (in a sense to be made precise below), let  solve the martingale problem with respect to $\Pi^k \cL$, and let $\llambda^k$ weakly converge towards $\llambda$. For $t \in [0,T]$, let $\mu_t^k = (e_t)_\sharp \llambda^k$, $\mu_t = (e_t)_\sharp \llambda$ and assume that $\mu_t^k = \Pi^k_\sharp \mu_t$.

Indeed, $\eeta \in P( C([0,1]; \R^d) )$ is a solution of the martingale problem associated to $\cL(a,b)$ if and only if the following property holds: for every $s$, $t \in [0,1]$ with $s\le t$, for every $f \in C^{1,2}_c([0,1]\times \R^d)$ (with $\nor{f}_{C^{1,2}} \le 1$) and for every bounded continuous and $\cF_s$-measurable function $g$ on $C([0,T]; \R^d)$ (with $\nor{g}_\infty \le 1$) it holds
%\begin{equation}
%\label{eq:martingale-solution-by-density}
\[
 \int g \sqa{ f_t \circ e_t - f_s \circ e_s - \int_s^t \sqa{(\partial_t +\cL_r) f }\circ e_r\,   dr} d\eeta = 0.\]
 %\end{equation}
As the correspondent identity holds for $\eeta^n$ and $\cL^n$, i.e.\
\[ \int g \sqa{ f_t\circ e_t - f_s \circ e_s - \int_s^t \sqa{(\partial_t +\cL^n_r) f} \circ e_r  \, dr}d\eeta^n = 0,\]
to deduce that $\eeta$ is a solution to the martingale problem associated to $\cL$, since $f$ and $\partial_t f$ are bounded and continuous, the crucial limit is
\begin{equation} \label{eq:martingale-limit-crucial}\int g \sqa{\int_s^t (\cL_r^n f) \circ e_r \, dr} d\eeta^n - \int g \sqa{ \int_s^t (\cL_r f) \circ e_r\,  dr} d\eeta \to 0, \end{equation}
whose validity we now investigate, according to the  approximations from Section \ref{sec:approximation}.

{\bf Push forward via smooth maps.} For $n \ge 1$, let $\pi^n\in C^2_b(\R^d, \R^d)$ with $(\pi_n)_n$ converging to the identity map locally uniformly and assume that the sequence of first and second derivatives converge (towards the respective limits), pointwise and uniformly bounded,  i.e., $\nabla \pi^n(x) \to Id$ for $x \in \R^d$, $\nabla^2 \pi^n (x) \to 0$, for every $x \in \R^d$, and there exists some constant $C \ge 0$ such that $\abs{\nabla^i \pi^n(x)} \le C$, for $x \in \R^d$ and $i \in \cur{1,2}$.

Let $\nu^n = \pi^n_\sharp \nu$, $\pi^n(\cL)$, and let $\eeta^n$ be corresponding superposition solution. To prove that any narrow limit point $\eeta$ is indeed a superposition solution for $\nu$, with respect to the diffusion operator $\cL$, we add and subtract the term \[
\int g \sqa{\int_s^t (\overline{\cL}_r f ) \circ e_r \,dr }  d\eeta^n - \int g\sqa{ \int_s^t (\overline{\cL}_r f) \circ e_r \, dr} d\eeta\]
in \eqref{eq:martingale-limit-crucial}, where $\overline{\cL} = \cL(\overline{a}, \overline{b} )$ is any diffusion operator on $\R^d$, whose coefficients $\overline{a}$, $\overline{b}$ are continuous and compactly supported. The difference terms above are infinitesimal as $n \to \infty$, by narrow convergence of $\eeta^n$, thus we estimate \eqref{eq:martingale-limit-crucial}, as $n \to \infty$, in terms of
\begin{equation}
\label{eq:crucial-limit-pushforward}
\limsup_{n\to \infty} \int \abs{ \cL^n f - \overline{\cL} f } d\nu^n + \int \abs{  \cL f - \overline{\cL} f} d\pi(\nu).\end{equation}
Let us focus on first term above, at fixed $n \ge 1$ (for simplicity of notation, we drop the dependence upon $n$). Recalling the definition of $\pi(\cL)$,%  it holds
%\[ \pi(\cL) f (y) = \E_{\nu} \sqa{ \cL (f\circ \pi)  \, | \, \pi = y}, \quad \text{ $\nu$-a.e. $y \in \R^d$,}\]
 integration with respect to the push-forward measure gives
\[\int \abs{ \pi(\cL) f - \overline{\cL} f } d\pi_\sharp \nu = \int \abs{ \E_{\nu}\sqa{ \cL (f \circ \pi) \, | \, \pi} - (\overline{\cL} f) \circ \pi } d\nu.\]
Being $(\overline{\cL} f) \circ \pi$ a function of $\pi$, up to $\nu$-negligible sets, we have
\begin{equation*}\begin{split}
\int \abs{ \E_\nu\sqa{ \cL (f \circ \pi) \, | \, \pi} - (\overline{\cL} f) \circ \pi } d\nu 
&= \int \abs{ \E_\nu\sqa{ \cL (f \circ \pi) - (\overline{\cL} f) \circ \pi \, | \, \pi}}d\nu \\ 
&\le \int \abs{ \cL (f \circ \pi) - (\overline{\cL} f) \circ \pi  } d\nu,
\end{split}\end{equation*} 
since conditional expectation reduces $L^1(\nu)$-norms. %As $n \to \infty$, $(\overline{\cL} f) \circ \pi^n \to (\overline{\cL} f) \circ \pi$ pointwise and uniformly bounded, so that, if it is sufficient to 
Writing explicitly the difference
\[ \cL (f \circ \pi) - (\overline{\cL} f) \circ \pi =  \frac 1 2 \sum_{i,j =1}^k  \sqa{ a(\nabla \pi^i, \nabla \pi^j) - \overline{a}^{i,j} \circ \pi}(\partial_{i,j} f) \circ \pi + \sum_{i = 1}^k  \sqa{ \cL(\pi^i) - \overline{b}^i\circ \pi}(\partial_i f) \circ \pi,\]
and recalling that $\nor{f}_{C^{1,2}} \le 1$, we conclude that
\[\int \abs{ \pi(\cL) f - \overline{\cL} f } d\pi_\sharp \nu \le  \int \frac 1 2 \sum_{i,j =1}^k \abs{ a(  \nabla \pi^i, \nabla \pi^j) - \overline{a}^{i,j} \circ \pi } d\nu + \int \sum_{i = 1}^k \abs{ \cL(\pi^i) - \overline{b}^i\circ \pi} d\nu.\]
Letting $n \to \infty$ (recall that $\pi = \pi^n$ above), using the assumption on the convergence of $\pi^n$ towards the identity map (in particular, we use Lebesgue dominated convergence w.r.t.\ the measure $\nu$), we deduce that \eqref{eq:crucial-limit-pushforward} is bounded from above by twice
\[ \int\frac 1 2  \sum_{i,j =1}^k \abs{ a^{i,j} - \overline{a}^{i,j} } d\nu + \int \sum_{i = 1}^k | b^i - \overline{b}^i| \,  d\nu.\]

To conclude, we choose $\overline{a}$, $\overline{b}$ that minimize the right hand side above: this can be made arbitrary small, by the density of continuous and compactly supported functions in $L^1(\nu))$.

{\bf Mollification by convolution.} In this case, the argument is similar, and it is more standard, see e.g.\ \cite[Theorem 8.2.1]{ambrosio-gigli-savare-book}, thus we only sketch it. Given a sequence $\rho^n$ of probability densities on $\R^d$ such that $\rho^n \mathscr{L}^d \to \delta_0$ narrowly as $n \to \infty$, let $\nu^n = \nu * \rho^n$ and $\cL^n$ be the diffusion operator introduced in Section \ref{sec:approximation}. We add and subtract, in \eqref{eq:martingale-limit-crucial}, \[
\int g \sqa{\int_s^t \overline{\cL}_r f  \circ e_r \, dr} d\eeta^n - \int g \sqa{ \int_s^t \overline{\cL}_r f \circ e_r\, dr} d\eeta,\]
where $\overline{\cL} = \cL(\overline{a}, \overline{b} )$ has continuous and compactly supported coefficients. %a diffusion operator on $\R^d$ as, with continuous and compactly supported coefficients. 
Let $\overline{\omega}$ be a common (bounded and continuous) modulus of continuity for $\overline{a}$, $\overline{b}$.

As in the previous case, narrow convergence implies that the absolute value of \eqref{eq:martingale-limit-crucial} is bounded from above, as $n \to \infty$, by
%\begin{equation}
%\label{eq:crucial-limit-convolution}
\[
\limsup_{n\to \infty} \int \abs{ \cL^n f - \overline{\cL} f } d\nu^n + \int \abs{  \cL f - \overline{\cL} f} d\nu.\]
%\end{equation}
First, we prove that $\lim_{n \to \infty} \int |  \overline{\cL}^n f  - \overline{\cL} f | d\nu^n = 0$, 
where $\overline{\cL}^n$ has coefficients
\[ \overline{a}^n := \frac{ d (\overline{a} \nu * \rho_n)}{d (\nu * \rho_n)}, \quad   \quad \overline{b}^n := \frac{ d (\overline{b} \nu * \rho_n)}{d(\nu * \rho_n)}.\]
Indeed, recalling that $\nor{f}_{C^{1,2}} \le 1$, we estimate
\begin{equation*}\begin{split}
 \int \abs{  \overline{\cL}^n f  - \overline{\cL} f } d\nu^n &\le \int \abs{  \overline{a}^n(x) - \overline{a}(x) } d\nu^n + \int |  \overline{b}^n - \overline{b} |\,  d\nu^n\\
 & = \int \abs{ (\overline{a} \nu* \rho_n) (x) - \overline{a}(x) (\nu* \rho_n)(x) }  dx+ \int \abs{ (\overline{b} \nu* \rho_n) (x) - \overline{b}(x) (\nu* \rho_n)(x) } dx \\
 & \le 2 \int \sqa{\int  \overline{\omega}(y-x) \rho_n(y-x) dx} \nu(dy) =  2 \int  \overline{\omega}(z) \rho_n(z) dz \to 0.
\end{split}\end{equation*}
Thanks to this fact, we write
\begin{equation*}\begin{split} \limsup_{ n\to \infty} \int \abs{ \cL^n f - \overline{\cL} f } d\nu^n & = \limsup_{ n\to \infty} \int \abs{ \cL^n f - \overline{\cL}^n f } d\nu^n\\
& =\limsup_{ n\to \infty} \int \bra{ \abs{ a^n - \overline{a}^n }  +  | b^n - \overline{b}^n | }d\nu^n\\
& \le \int \bra{\abs{ a - \overline{a} }  + | b - \overline{b} |} d\nu
\end{split}\end{equation*}
where in the last step we apply \eqref{eq:jensen}. % (whose validity does not rely on smoothness assumption on $\rho$). T
To conclude, it is sufficient to optimize upon $\overline{a}$, $\overline{b}$, by density of continuous and compactly supported functions in $L^1(\nu)$.

\subsection{Proof of Theorem \ref{thm:sp}}

We argue by iterating the three-steps scheme, the base case being that of diffusion operators with smooth and uniformly bounded coefficients. First, we extend the validity to uniformly bounded coefficients (without any regularity assumption), then to locally bounded coefficients, and finally integrable coefficients. Although everything could be obtain in a single iteration, we think the approach highlights the different roles played by different approximation procedures. Indeed, our crucial improvement with respect to \cite[Theorem 2.6]{figalli-sdes} is to move from uniformly bounded to integrable coefficients, which is rather delicate: by comparison, in the deterministic case, one is able to deal directly with locally smooth coefficients (see e.g.\ \cite[Proposition 8.1.8]{ambrosio-gigli-savare-book}), essentially because paths either go to infinity, i.e., the solution explodes in a finite time, or stay in a compact set. Roughly speaking, the source of difficulties in the stochastic case is that we have to deal with ``averages'' of such behaviours, and moreover the solution to a genuinely stochastic martingale problem is expected to instantaneously ``diffuse'' over compact sets (of course, with small probability as these sets become larger).

\vspace{0.5 em}

\noindent {\bf Case of smooth and bounded coefficients.} Let $a$, $b$ be Borel maps as in \eqref{eq:a-b-rn}, with
\begin{equation}\label{eq:smooth-bounded} \int_0^T \sqa{\nor{a_t}_{C^2_b(\R^d)} + \nor{b_t}_{C^2_b(\R^d)}} dt < \infty.\end{equation}
Then, the superposition principle holds for every solution $\nu = (\nu_t)_{t \in (0,T)} \subseteq \scrP ( \R^d)$ of the FPE \eqref{eq:fpe-rn}. This follows from two well-known facts: existence of It\^o's stochastic differential equations and uniqueness for narrowly continuous solutions of FPE's.

The existence result is standard, with the possible exception of the integrable bounds with respect to the variable $t \in [0,T]$ (usually, one requires uniform bounds), but in fact such condition is sufficient for the various applications of Gronwall inequality. For the sole purpose of establishing a case base for the superposition principle, the usual stronger assumptions on the coefficients, e.g.\ $a$, $b \in C^\infty_b((0,T) \times \R^d)$ would even be sufficient, at the price of introducing an extra mollification step with respect to the variable $t \in [0,T]$. %: we prefer to state the general result.

%Our first result relies on It\^o's theory of stochastic differential equations: being rather standard we only sketch its proof.

\begin{theorem}\label{theo:existence-smooth}
Let $a$, $b$ be Borel maps as in \eqref{eq:a-b-rn}, satisfying \eqref{eq:smooth-bounded}. Then, for every $\bar \nu \in \scrP (\R^d)$, there exists a solution $\eeta$ of the MP associated to $\cL(a, b)$, with $\eta_0 = \bar \nu$. %It holds
%\[ \sup_{t \in [0,T]} \int \abs{\gamma(t)}^p d\llambda(\gamma) \le \exp\cur{T C(\aa, b) } \int \abs{x}^p d\mu_0(x),\]
%for some constant depending upon $\aa$, $b$ and $n$ only.
\end{theorem}

\begin{proof}
The assumption $a \in L^1_t(C^2_b(\R^d))$ entails that the symmetric non-negative square-root of $a$, i.e.\ the (essentially unique) map
\[ \sigma: [0,T] \times \R^d \to \operatorname{Sym}_+(\R^d) \quad\text{ such that $\sigma^2_t = a_t$, \quad $\scrL^1$-a.e.\ $t \in (0,T)$,}\]
is bounded and Lipschitz with respect to $x \in \R^d$, with Lipschitz constant integrable w.r.t.\ $t \in (0,T)$, see e.g.\ \cite[Lemma 3.2.3]{stroock-varadhan}. Then, it is sufficient to solve by Picard iteration the It\^o stochastic differential equation
%\begin{equation}\label{eq:sde} 
\[ dX_t = b_t( X_t) dt +   \sigma_t(X_t) dW_t, \quad X_0 = \overline{X},\]%\end{equation}
where $\overline{X}$ is a r.v.\ independent of the $d$-dimensional Wiener process $W$. By It\^o formula, the law of $X$, i.e. $X_{\sharp} \P$, is a solution of the martingale problem associated to $\cL(a,b)$.
%Moreover, by pathwise uniqueness, hence uniqueness in law, for the SDE \eqref{eq:sde}, the function mapping the law of $\overline{X}$ into $\P_{\sharp} X$ is injective and if the space $\Omega$ is sufficiently large, we may provide a Borel (actually, continuous) family of probabilities $(\eeta(x))_{x \in \R^d} \subseteq\scrP( C([0,T]; \R^d))$, each solving the martingale problem associated to $\cL(a,b)$, with $\eta(x)_0 = \delta_x$, see \cite[Theorem 5.11]{stroock-varadhan}. Therefore, even if the law of $\bar{\nu}$ has no finite second moments, we show existence of a solution to the MP letting $\eeta = \int \eeta(x) d \bar{\nu}(x)$.
\end{proof}

Of course, the MP is also well-posed, but we need a stronger uniqueness result, for narrowly continuous solutions of FPE's, which is e.g.\ a consequence of results on backward Kolmogorov equations. We refer e.g.\ to the expository notes by \cite{krylov-lecture-notes} for more details; notice however that, also in this case, the standard literature studies equations of the form
\begin{equation}\label{eq:back-kolmo} \partial_t f = - \cL_t f + g, \quad \text{ in $(0,T) \times \R^d$, \quad $f_T = \bar{f}$,}\end{equation}
assuming $a$, $b$ smooth and $g \in C^\infty_c((0,T)\times \R^d)$. A solution to the equation \eqref{eq:back-kolmo} is defined as a function $f \in C^{1,2}_b((0,T)\times \R^d)$ such that
\[  \partial_t f (s,x) = - \cL_s f(s,x) + g(s,x), \quad  \text{for $(s,x) \in (0,T)\times \R^d$,}\quad \text{with} \quad \lim_{s \uparrow T} f (s,x) = \bar{f}(x).\]
To our purposes, we need existence of a solution, together the following regularity results for the solution $f$ (which entails uniqueness):
\begin{equation}\label{eq:uniform-bound} \sup_{t \in [0,T]} \nor{f_t}_{C^2_x} \le  \bra{\| \bar{f}\|_{C^2_x} + T\nor{ g}_{C^2_{t,x}} } C\bra{ \int_0^T \sqa{\nor{a_t}_{C^2_x} + \nor{b_t}_{C^1_x} } dt + T\nor{ g}_{C^2_{t,x}}}, \end{equation}
where $z \mapsto C(z)$ denotes some function depending on the dimension $d$ only (the proof gives that $C$ has an exponential behaviour). The proof follows by direct differentiation of the equation, see \cite[Theorem 3.2.4]{stroock-varadhan} for a detailed derivation. Moreover, as a consequence of the maximum principle, if $\bar{f} \ge 0$, and $g \ge 0$, then the solution $f$ is non-negative as well.

We are in a position to prove the following result, akin to \cite[Proposition 8.1.7]{ambrosio-gigli-savare-book}. Again, we provide a slightly stronger statement than what is needed for the superposition principle (e.g., we deduce uniqueness for possibly signed solutions of the FPE).

\begin{theorem}\label{theo:uniqueness-smooth}
Let $a$, $b$ be Borel maps as in \eqref{eq:a-b-rn} with 
\[ \int_0^T \nor{ a_t}_{C^2(B)} + \nor{ b_t}_{C^2(B)} dt <\infty, \quad \text{for every bounded open  $B \subseteq \R^d$,}\]
and $\nu = (\nu_t)_{t \in [0,T]}\subseteq \scrM (\R^d)$ be a narrowly continuous solution of the FPE associated \eqref{eq:fpe-rn}. %\footnote{, i.e.\ \eqref{eq:lp-coefficients} holds with $p=1$}.
If $\nu_0 \le 0$, then $\nu_t \le 0$, for every $t \in [0,T]$. Thus, for $\bar{\nu} \in\scrM (\R^d)$  there exists at most one narrowly continuous solution $\nu$ with $\nu_0 = \bar{\nu}$.
\end{theorem}
\begin{proof}
Let $g \in C^\infty_c((0,T) \times \R^d)$, with $g \ge 0$: it is sufficient to show that $\int g \, d\nu \le 0$. Fix $R \ge 1$ large enough so that the support of $g$ is contained in $(0,T) \times B_R(0)$ and let $\chi_R$ be a cut-off function, as below Remark \ref{rem:extension-weak-formulation}. Notice that letting $a_R = a\chi_R$ and $b_R = b\chi_R$ in place of $a$, $b$, condition \eqref{eq:smooth-bounded} holds and $\cL_R f = \cL f$ on $(0,T) \times B_R(0)$, for every $f \in C^2_b((0,T)\times \R^d)$.

For $\veps >0$, let $a_R^\veps$, $b_R^\veps$ be a double mollification with respect to the space and time variables, and define $\cL_R^\veps = \cL(a_R^\veps, b_R^\veps)$, which is a diffusion operator with smooth and bounded coefficients, satisfying \eqref{eq:smooth-bounded} uniformly in $\veps >0$. Let $f^\veps$ be a solution to the backward Kolmogorov equation
\[ \partial_t f^\veps  = -\cL^\veps_R f^\veps + g, \quad f_T^\veps = 0,\]
and choose $f^\veps \chi_R $ in the weak formulation \eqref{eq:weak-fpe-rn}, which is admissible because $f^\veps \in C^{1,2}_b((0,T)\times\R^d)$. Since $f^\veps \le 0$ and $\nu_0 \le 0$, we have
\begin{equation*} \begin{split} 0 & \ge - \int f^\veps \chi_R\,  d\nu_0 = \int \sqa{\chi_R \, \partial_t f^\veps + \cL( f^\veps \chi_R )} d\nu\\
&=\int \sqa{- \chi_R \, \cL^\veps_R f + \cL( f^\veps \chi_R )} d\nu\\
&= \int \cur{\chi_R \sqa{ g+ \cL^\veps_R f^\veps - \cL f^\veps } + f^\veps \cL \chi_R +  a (\nabla f^\veps, \nabla \chi_R )} d\nu \\
&\ge \int g\,  d\nu -  \sup_{t \in [0,T]} \nor{f^\veps_t}_{C^2_b(\R^d)} \int \sqa{ \chi_R \abs{a^\veps_R-a} + \abs{b^\veps_R-b} +  \abs{ \cL \chi_R } +  \abs{a} \abs{\nabla \chi_R}}d\abs{\nu}.
\end{split}\end{equation*}
As $\veps \downarrow 0$, since $a_R = a$ and $b_R = b$ on $(0,T)\times B(0,R)$, the second integral converges to $\int\sqa{ \abs{\cL \chi_R }+  \abs{a} \abs{\nabla \chi_R} }d\abs{\nu}$, and $\sup_{t \in [0,T]} \nor{f^\veps_t}_{C^2_b}$ is uniformly bounded in $\veps >0$, by \eqref{eq:uniform-bound}. Finally, we let $R \to \infty$ and conclude, since $\abs{\nabla \chi_R} + \abs{\nabla \chi_R}\to 0$, pointwise and uniformly bounded.
\end{proof}

The superposition principle follows immediately from these facts: since any weak solution $\nu = (\nu_t)_{t \in (0,T)}$ admits a narrowly continuous representative $\tilde{\nu}$, we let $\eeta$ be a solution of the MP associated to $\cL(a,b)$, with $\bar{\nu} = \tilde{\nu}_0$ (Theorem \ref{theo:existence-smooth}) and notice that the curve $\eta = (\eta_t)_{t\in [0,T]}$ is a narrowly continuous solution of the FPE associated to $\cL$, with $\eta_0 = \tilde{\nu}_0$. By Theorem \ref{theo:uniqueness-smooth}, we conclude that $\eta_t = \tilde{\nu}_t$, for $t \in [0,T]$.

%We address the proof of our general superposition superposition principle for diffusions in $\R^d$, i.e.\ Theorem \ref{thm:sp} below, by using the case just settled as a base case to iterate the approximation-tightness-limit scheme.

\vspace{0.5em}

\noindent {\bf Case of bounded coefficients.} We extend the validity of the superposition principle for diffusions with uniformly bounded coefficients: this already provides an extension of \cite[Theorem 2.6]{figalli-sdes}, as uniform bounds are imposed only with respect to $x \in \R^d$. Precisely, we assume that the coefficients $a$, $b$ satisfy
\begin{equation}\label{eq:bounded} \int_0^T \sup_{x \in \R^d} \sqa{\abs{a_t(x)} + \abs{b_t(x) } }dt < \infty.\end{equation}
%and prove that the superposition principle holds for every weak solution $\nu = (\nu_t)_{t \in (0,T)} \subseteq \scrP ( \R^d)$ to the FPE \eqref{eq:fpe-rn}.

%We rely on the approximation-tightness-limit strategy discussed in the previous sections.

%\vspace{0.5em}

\noindent {\it Step 1} (approximation){\it.}\  We argue by convolution with a kernel $\rho =a \exp(-\sqrt{ 1 + |x|^2 } )$. For $\veps\in (0,1)$, let $\rho^{\veps}(x) = \veps^n\rho(x/\veps)$ and notice that $\abs{\nabla^i \rho^\veps} \le C \veps^{-2}\rho^\veps$, for $i\in \cur{1,2}$, where $C$ is some absolute constant. Then, $\nu^\veps = \nu * \rho^\veps$ solves a FPE with respect to a diffusion operator with coefficients $a^\veps$, $b^\veps$ satisfying (the correspondent of) \eqref{eq:smooth-bounded}, as a consequence of the last statement in Lemma \ref{lemma:approximation}. Existence of superposition solutions $\eeta^\veps \in \scrP(C([0,T]; \R^d))$ for the  associated martingale problems follows from the smooth case settled above.

\vspace{0.5em}

\noindent {\it Step 2} (tightness){\it.}\ We notice first that, being $(\nu^\veps_0)_{\veps >0}$ a narrowly convergent sequence of probability measures (thus, it is also tight), there exists some increasing function $\theta: \R \to \R$ with $\lim_{z \to \infty}\theta(z) = \infty$ such that $\sup_{\veps >0} \int \theta(\abs{x} ) d\nu^\veps_0  \le 1$. For $R \ge 1$, let then $\chi_R: \R^d \to [0,1]$ be the usual cut-off function and, for $i \in \cur{1, \ldots, d}$, let $x^i_R(x) := x_i\chi_R \in \Algebra$.  For any (but fixed) $p \in (1,\infty)$, let $\Theta_1(x) = \Theta_2(x) = \abs{x}^p$, and apply Corollary \ref{coro:regularity-paths-martingale-general} to the solution $\eta^\veps$ ($\veps >0)$, with $f := x^i_R \circ e_t$ and  these precise choices of $\theta$, $\Theta_1$, $\Theta_2$. We obtain some coercive functional $\Psi: C([0,T]; \R) \to [0,\infty]$ (depending upon $\theta$ and $p$ only) such that
\begin{equation}\label{eq:tightness-1}  \int  \Psi( x^i_R \circ \gamma ) d\eeta^\veps(\gamma) \le \int \theta(\abs{ x^i_R }) d\eta_0 ^\veps+ \int_0^1 \bra{ \abs{ \cL_t^\veps  x^i_R }^p + \abs{a_t^\veps(\nabla x^i_R, \nabla x^i_R)}^p }d\eta_t^\veps dt.\end{equation}
%\[ \cA(\vphi_R^i) \le C \cur{ \sqa{ \int_0^T 
% and $(\varphi^i_R)_t :=x^i_R \circ e_t$, for $t \in [0,T]$. %Since $\abs{ \nabla f^i_R}$ and $\abs{ \nabla^2 f^i_R}$ are uniformly bounded (both in $x \in \R^d$ and in $R \ge 1$.
%Corollary \ref{coro:tightness-2} entails that, for any $\delta \in (0,1/2)$ and $p \in (2,\infty)$, it holds
%\[ \int \nor{ \varphi^i_R }_{\delta,p}^p d\eeta^\veps \le  C \int_0^T \int \sqa{\abs{\ell_r}^{p} + \alpha_r^{p/2} } d\nu^\veps_r dr,\]
%for some constant $C$ depending on $p$, $\delta$ and $T$ only, where
%\[ \ell = \sqa{\partial_t + \cL^\veps} x^i_R, \quad \alpha = 2 a^\veps( \nabla x^i_R, \nabla x^i_R ).\]
Since $\nor{ \nabla x^i_R}_\infty$ is uniformly bounded and $\nor{ \nabla^2 x^i_R}_\infty$ is infinitesimal as $R \to \infty$, we may let $R \to \infty$, and by lower-semicontinuity of $\Psi$, Fatou's lemma and Lebesgue dominated convergence theorem, we obtain a similar bound with the functions $x^i$ in place of $x^i_R$:
% \begin{equation}\label{eq:tightness-2}
 \[  \int  \Psi( \gamma^i ) d\eeta^\veps(\gamma) \le \int \theta(\abs{ x^i }) d \nu_0^\veps(x) + \int_0^1  \bra{ \abs{ (b^\veps)^i_t }^p + |(a^\veps)^{i,i}_t|^p } d\nu_t^\veps dt,\]%\end{equation}
where we also make explicit the fact that $\eta^\veps_t = \nu_t^\veps$, for $t \in [0,T]$.

%\[ \limsup_{R \to \infty} \int \nor{ \varphi^i_R }_{\delta,p}^p d\eeta^\veps \le C \int \sqa{\abs{a^\veps}^{p} + \abs{b^\veps}^p + \abs{a^\veps}^{p/2} } d\nu^\veps,\]
%which by \eqref{eq:jensen} and Fatou's lemma gives
%\[ \int \nor{ \varphi^i }_{\delta,p}^p d\eeta^\veps \le C \int \sqa{\abs{a}^{p} + \abs{b}^p + \abs{a}^{p/2} } d\nu,\]
%where $\varphi^i_t := x^i \circ e_t$. Let us remark that a slightly more detailed analysis gives
%\begin{equation}
%\int \nor{ \varphi^i }_{\delta,p}^p d\eeta^\veps \le C \int \sqa{\abs{b^i}^p + \abs{a^{i,i}}^{p/2} } d\nu.
%\end{equation}
%\[ \E\sqa{\cA(x^i)} \le C
Inequality \eqref{eq:jensen} and the assumptions on $\theta$ entails uniform bounds for $\veps >0$ of the form
 %\begin{equation}\label{eq:tightness-3}
 \[  \int  \Psi( \gamma^i ) d\eeta^\veps(\gamma) \le 1 + \int_0^1 \bra{ | b^i |^p + |a_t^{i,i}|^p }d\nu_t dt.\]%\end{equation}
Tightness follows, for $\gamma \mapsto \sum_{i=1}^d \Psi( \gamma^i )$ is coercive in $C([0,T]; \R^d)$.

\vspace{0.5em}

\noindent {\it Step 3} (limit){\it.}\ This step is fully covered in Section \ref{sec:weak-limit-martingales}.

%The argument in the tightness step in the proof above entails the following result, if applied to any solution to the martingale problem associated with $\cL(a,b)$.

%\begin{corollary}
%Let $(a,b)$ be Borel maps as in \eqref{eq:a-b-rn} such that \eqref{eq:bounded} holds. Then, any solution to the martingale problem associated with $\cL(a,b)$ is concentrated on $\delta$-H\"older continuous curves, for every $\delta \in (0,1/2)$.
%\end{corollary}
\vspace{0.5em}

\noindent {\bf Case of locally bounded coefficients.} Next, we assume that
\begin{equation}\label{eq:locally-bounded} \int_0^T \sup_{x \in B} \sqa{\abs{a_t(x)} + \abs{b_t(x) } }dt < \infty, \quad \text{for every bounded borel $B \subseteq \R^d$.}
\end{equation}
and we prove the validity of the superposition principle for every weak solution $\nu = (\nu_t)_{t \in (0,T)} \subseteq \scrP ( \R^d)$ of the FPE \eqref{eq:fpe-rn} (recall that we also assume \eqref{eq:fpe-integrability}).

%We follow the approximation-tightness-limit scheme discussed in Section \ref{sec:ap-ti-li}.
\noindent {\it Step 1} (approximation){\it.}\  We approximate via push-forward by smooth maps. For $M \ge 1$, let $\chi_M$ be the usual cut-off function and let $\pi_M: \R^d \mapsto \R^d$ be the map
\[ \pi_M(x) = x \chi_M(x), \quad \text{so that $\pi_M^i(x) = x ^i\chi_M(x) \in C^2_c(\R^d)$.}\] 
By \eqref{eq:locally-bounded}, it holds $\abs{\cL ( \pi_M^i ) } \le   \nor{ \pi_M^i}_{C^2} \sup_{ \abs{x} \le 2M } \sqa{\abs{a(x)} + \abs{b(x)} }$ for $x \in \R^d$, $i \in \cur{1,\ldots d}$, and similarly $\abs{ a(\nabla \pi_M^i, \nabla \pi_M^j) } \le \nor{ \pi_M^i}_{C^1} \sup_{ \abs{x} \le 2M } \abs{a(x)}$, for $x \in \R^d$, $i, j \in \cur{1,\ldots d}$.

Since conditional expectations reduce norms, we deduce that $\nu^M := \pi^M(\nu)$ solves a FPE associated to a diffusion on $\R^d$, whose coefficients $a^M$, $b^M$ satisfy \eqref{eq:bounded}: thus the previous argument gives superposition solutions $\eeta^M$.

\vspace{0.5em}

\noindent {\it Step 2} (tightness){\it.}\  The argument is very similar to the previous case, with the only caveat that $\Theta_1$ and $\Theta_2$ must be chosen more carefully. Indeed, we rely on the de la Vall\'ee Poussin criterion, which improves the integral bound \eqref{eq:fpe-integrability} to one of the form
\[ \int_0^T \int \sqa{\Theta_1(\abs{b}) + \Theta_2(\abs{a})} d\nu_t dt < \infty,\]
for some suitable $\Theta_1$, $\Theta_2$ that fulfil the assumptions of Theorem \ref{theorem:basic-burkholder} (the moderate growth assumption on $\Theta_2$ can be always obtained, possibly passing to a slightly ``worse'' function). With such a choice of $\Theta_1$, $\Theta_2$ (and building $\theta$ as in the previous case, for $(\nu^M)_{M >0}$ is tight), we obtain for some coercive functional $\Psi $ such that inequalities akin to \eqref{eq:tightness-1} and the following ones, with $\Theta_1(z)$, $\Theta_2(z)$ in place of $\abs{z}^p$. %\eqref{eq:tightness-1} (letting $R \to \infty$) and \eqref{eq:tightness-2} (because of Jensen inequality for conditional expectations).
%\[ \ell^M := (b^M)^i, \quad \text{and} \quad \alpha^R := 2(a^R)^{i,i}.\]
%Jensen's inequality and \eqref{eq:p_1-p_2,bound} entail that these bounds are uniform as $R \to \infty$, so that tightness follows at once.
%\[ \int \nor{ \varphi^i}_{\delta,p}^p d\eeta^R \le C \int \sqa{\abs{b^R}^p + \abs{a^R}^p} d\nu^R, \text{ for $i \in \cur{1,\ldots, d}$,}\]
%for any $R$, where $C$ depends uniquely upon $p$, $T$, and $\delta$. Since $\nor{ \nabla \pi_R^i}$ and $\nor{\nabla^2 \pi_R^i}$ are uniformly bounded in $R \ge 1$ and conditional expectations reduce $L^p$ norms, the right hand side above can be estimated from above, leading to
%\[ \int \nor{ \varphi^i}_{\delta,p}^p d\eeta^R \le C \int \sqa{\abs{b}^p + \abs{a}^p} d\nu, \text{ for $i \in \cur{1,\ldots, d}$,}\]
%where $C$ still depends uniquely upon $p$, $T$, and $\delta$ (and the choice of the family $\chi_R$, but the key point is that it is uniformly bounded as $R \to \infty$). As $\delta$ can be chosen so that $\delta p >1$, we obtain tightness by by arguing as in the correspondent step in the proof of Theorem \ref{theo:superposition-bounded}.

\vspace{0.5em}

\noindent {\it Step 3} (limit){\it.}\ This step is described in Section \ref{sec:weak-limit-martingales}.

%Again, the tightness step in the proof above actually entails a regularity result, when the argument is applied to any solution to the martingale problem associated with $\cL(a,b)$.

\vspace{0.5em}

\noindent {\bf General case.} The final step consists of removing the assumption \eqref{eq:locally-bounded}.
\vspace{0.5em}

\noindent {\it Step 1} (approximation){\it.}\ We perform once again an approximation via convolution, e.g.\ as in the case of uniformly bounded coefficients. In this case, however, we only use the fact that $(\nu^\veps)_\veps$ are solutions to FPE's associated to diffusion operators whose coefficients are locally bounded (and the bound \eqref{eq:fpe-integrability} is preserved).% We may consider superposition solutions $(\eeta^\veps)_\veps$ to the correspondent martingale problem.
\vspace{0.5em}

\noindent {\it Step 2} (tightness){\it.}\  We argue exactly as in the previous case, i.e.\ using de la Vall\'ee Poussin criterion to provide suitable $\Theta_1$, $\Theta_2$.
\vspace{0.5em}

\noindent {\it Step 3} (limit){\it.}\  Again, this step is described in Section \ref{sec:weak-limit-martingales}.

\vspace{0.5em}

As already remarked at the beginning of this section, one could combine all the arguments above and prove Theorem \ref{thm:sp}, starting from the ``base case''  with a single combination of mollifications and push-forwards approximations. On a technical level, the main difficulty is to obtain the result for locally bounded coefficients, and this is done after we establish the result for uniformly bounded coefficients, regardless of their regularity, essentially because the push-forward approximation may not preserve it.

\def\cprime{$'$} \def\cprime{$'$}
\providecommand{\bysame}{\leavevmode\hbox to3em{\hrulefill}\thinspace}


\begin{thebibliography}{BDPRS07}



\bibitem[AC08]{ambrosio-crippa-lecture-notes}
L.\ Ambrosio and G.\ Crippa, \emph{Existence, uniqueness, stability and
  differentiability properties of the flow associated to weakly differentiable
  vector fields}, Transport equations and multi-{D} hyperbolic conservation
  laws, Lect. Notes Unione Mat. Ital., vol.~5, Springer, Berlin, 2008,
  pp.~3--57.

\bibitem[AC14]{ambrosio-crippa-edinburgh}
\bysame, \emph{Continuity equations and ode flows with non-smooth velocity.},
  Proc. Roy. Soc. Edinburgh Sect. \textbf{A 144} (2014), 1191--1244.

\bibitem[AF09]{ambrosio-figalli-wiener}
L.\ Ambrosio and A.\ Figalli, \emph{On flows associated to {S}obolev
  vector fields in {W}iener spaces: an approach \`a la {D}i{P}erna-{L}ions}, J.
  Funct. Anal. \textbf{256} (2009), no.~1, 179--214.
  
\bibitem[AGS08]{ambrosio-gigli-savare-book}
L.\ Ambrosio, N.\ Gigli, and G.\ Savar{\'e}, \emph{Gradient flows in
  metric spaces and in the space of probability measures}, second ed., Lectures
  in Mathematics ETH Z\"urich, Birkh\"auser Verlag, Basel, 2008. 

\bibitem[Amb04]{ambrosio-bv}
L.\ Ambrosio, \emph{Transport equation and {C}auchy problem for {$BV$} vector
  fields}, Invent. Math. \textbf{158} (2004), no.~2, 227--260. 

%\bibitem[AST]{ambrosio-stra-trevisan}
%L.~Ambrosio, F.~Stra, and D.~Trevisan, \emph{Stability of {D}i{P}erna-{L}ions
%  flows on metric measure spaces}, In progress.

\bibitem[AT14]{ambrosio-trevisan}
L.~{Ambrosio} and D.~{Trevisan}, \emph{{Well posedness of {L}agrangian flows
  and continuity equations in metric measure spaces}}, Analysis and PDE
  \textbf{7} (2014), no.~5, 1179--1234.
  
    
\bibitem[BDPRS07]{bogachev-daprato-rockner-stannat}
V. I.\ Bogachev, G.\ Da Prato, M.\ R\"ockner and W.\ Stannat, \emph{Uniqueness of Solutions to Weak Parabolic Equations for Measures}, Bull.\ London Math.\ Soc.\ \textbf{39} (2007), no.\ 4, 631--40.
  
\bibitem[BDPR08]{bogachev-daprato-rockner-2008}
V.I.\ Bogachev, G.\ Da Prato and M.\ R\"ockner, \emph{{O}n parabolic equations for measures}, Communications in PDEs \textbf{33} (2008), no.~1-3, 397--418.

\bibitem[BRS11]{bogachev-rockner-shapo-2011}
V.I.\ Bogachev, M.\ R\"ockner and S.V.\ Shaposhnikov, \emph{{O}n uniqueness problems related to the {F}okker-{P}lanck-{K}olmogorov equation for measures}, J.\ Math.\ Sci.\ \textbf{179} (2011), no.\ 1,  7--47.

\bibitem[BRS13]{bogachev-rockner-shapo-2013}
\bysame \emph{{O}n uniqueness of solutions to the {C}auchy problem for degenerate {F}okker-{P}lanck-{K}olmogorov equations}, J.\ Evol.\ Eq.\ \textbf{13} (2013), no.\ 3, 577--593.

\bibitem[BRS15]{bogachev-rockner-shapo-2015}
\bysame, \emph{{U}niqueness problems for degenerate {F}okker-{P}lanck-{K}olmogorov equations},  J.\ Math.\ Sci. \textbf{207} (2015), no.\ 2, 147--65.

\bibitem[BRKS15]{bogachev-rockner-krylov-shapo-2015}
V.I.\ Bogachev, N.V.\ Krylov, M.\ R\"ockner and S.V.\ Shaposhnikov, \emph{{F}okker-{P}lanck-{K}olmogorov equations}, in preparation.
  
\bibitem[BC06]{bouchut-crippa-06}
F.~Bouchut and G.~Crippa, \emph{Uniqueness, renormalization, and smooth approximations for
              linear transport equations}, SIAM J. Math. Anal. \textbf{38}, (2006), no.~4, 1316--1328.

\bibitem[CDL08]{crippa-delellis-08}
G.~Crippa and C.~De Lellis, \emph{Estimates and regularity results for the {D}i{P}erna-{L}ions          flow}, J. Reine Angew. Math., \textbf{616}, {2008}, {15--46}.

\bibitem[DL89]{diperna-lions}
R.~J. DiPerna and P.-L. Lions, \emph{Ordinary differential equations, transport
  theory and {S}obolev spaces}, Invent. Math. \textbf{98} (1989), no.~3,
  511--547. 
  
\bibitem[DFPR13]{daprato-flandoli-priola-rockner}
G.~Da~Prato, F.~Flandoli, E.~Priola, and M.~R{\"o}ckner, \emph{Strong
  uniqueness for stochastic evolution equations in {H}ilbert spaces perturbed
  by a bounded measurable drift}, Ann. Probab. \textbf{41} (2013), no.~5,
  3306--3344.

\bibitem[EK86]{ethier-kurtz}
S.N.\ Ethier and T.G.\ Kurtz, \emph{Markov processes. characterization
  and convergence}, Wiley Series in Probability and Mathematical Statistics:
  Probability and Mathematical Statistics, John Wiley \& Sons, Inc., New York,
  1986.

\bibitem[Fig08]{figalli-sdes}
A.\ Figalli, \emph{Existence and uniqueness of martingale solutions for
  {SDE}s with rough or degenerate coefficients}, J. Funct. Anal. \textbf{254}
  (2008), no.~1, 109--153. 

\bibitem[FLT10]{fang-luo-thalmaier-sde-sobolev}
S.\ Fang, D.\ Luo, and A.\ Thalmaier, \emph{Stochastic differential
  equations with coefficients in {S}obolev spaces}, J. Funct. Anal.
  \textbf{259} (2010), no.~5, 1129--1168. 

\bibitem[GT01]{gilbarg-trudinger}
D.\ Gilbarg and N.S.\ Trudinger, \emph{Elliptic partial differential
  equations of second order}, Classics in Mathematics, Springer-Verlag, Berlin,
  2001, Reprint of the 1998 edition. 

\bibitem[KR05]{krylov-rockner}
N.~V.\ Krylov and M.~R{\"o}ckner, \emph{Strong solutions of stochastic equations
  with singular time dependent drift}, Probab. Theory Related Fields
  \textbf{131} (2005), no.~2, 154--196. 

\bibitem[Kry99]{krylov-lecture-notes}
N.V.\ Krylov, \emph{On {K}olmogorov's equations for finite dimensional
  diffusions}, Stochastic {PDE}'s and {K}olmogorov {E}quations in {I}nfinite
  {D}imensions (Giueppe Da~Prato, ed.), Lecture Notes in Mathematics, vol.
  1715, Springer Berlin Heidelberg, 1999, pp.~1--63 (English).

\bibitem[KS98]{kurtz-stockbridge}
T.G.\ Kurtz and R.H.\ Stockbridge, \emph{Existence of {M}arkov
  controls and characterization of optimal {M}arkov controls}, SIAM J. Control
  Optim. \textbf{36} (1998), no.~2, 609--653 (electronic). 

\bibitem[Kur07]{kurtz-yw}
T.G.\ Kurtz, \emph{The {Y}amada-{W}atanabe-{E}ngelbert theorem for general
  stochastic equations and inequalities}, Electron. J. Probab. \textbf{12}
  (2007), 951--965. 

\bibitem[LBL08]{lebris-lions}
C.~Le~Bris and P.-L. Lions, \emph{Existence and uniqueness of solutions to
  {F}okker-{P}lanck type equations with irregular coefficients}, Comm. Partial
  Differential Equations \textbf{33} (2008), no.~7-9, 1272--1317. 

\bibitem[LLP80]{lenglart-lepingle-pratelli}
E.~Lenglart, D.~L{\'e}pingle, and M.~Pratelli, \emph{Pr\'esentation unifi\'ee
  de certaines in\'egalit\'es de la th\'eorie des martingales}, Seminar on
  {P}robability, {XIV} ({P}aris, 1978/1979) ({F}rench), Lecture Notes in Math.,
  vol. 784, Springer, Berlin, 1980, pp.~26--52.


\bibitem[Luo13]{luo-fpe-sobolev-bv}
De~Jun Luo, \emph{Fokker-{P}lanck type equations with {S}obolev diffusion
  coefficients and {BV} drift coefficients}, Acta Math. Sin. (Engl. Ser.)
  \textbf{29} (2013), no.~2, 303--314.
  
\bibitem[PS12]{pao-ste-12}
E.\ Paolini and E.\ Stepanov,  \emph{Decomposition of acyclic normal currents in a metric space},  J. Funct. Anal., \textbf{263} (2012), no.~11, 3358--3390.

\bibitem[RZ10]{rockner-zhang-uniqueness-fpe}
M.\ R{\"o}ckner and X.\ Zhang, \emph{Weak uniqueness of
  {F}okker-{P}lanck equations with degenerate and bounded coefficients}, C. R.
  Math. Acad. Sci. Paris \textbf{348} (2010), no.~7-8, 435--438. 

\bibitem[Sho97]{showalter}
R.~E.\ Showalter, \emph{Monotone operators in {B}anach space and nonlinear
  partial differential equations}, Mathematical Surveys and Monographs,
  vol.~49, American Mathematical Society, Providence, RI, 1997. 

\bibitem[Smi93]{smirnov-sp}
S.~K.\ Smirnov, \emph{Decomposition of solenoidal vector charges into elementary
  solenoids, and the structure of normal one-dimensional flows}, Algebra i
  Analiz \textbf{5} (1993), no.~4, 206--238.

\bibitem[SV06]{stroock-varadhan}
D.W.\ Stroock and S.~R.~Srinivasa Varadhan, \emph{Multidimensional
  diffusion processes}, Classics in Mathematics, Springer-Verlag, Berlin, 2006,
  Reprint of the 1997 edition. 

\bibitem[Tre14]{trevisan-phd}
D.\ Trevisan, \emph{Well-posedness of diffusion processes in metric measure
  spaces}, PhD thesis, 2014.

\bibitem[Tre15]{trevisan-diff-mms}
D.\ Trevisan, \emph{Diffusion processes in metric measure
  spaces}, in progress.
  
\bibitem[Ver80]{veretennikov}
A.~Ju. Veretennikov, \emph{Strong solutions and explicit formulas for solutions
  of stochastic integral equations}, Mat. Sb. (N.S.) \textbf{111(153)} (1980),
  no.~3, 434--452, 480. 

\bibitem[YW71]{yamada-watanabe}
T.\ Yamada and S.\ Watanabe, \emph{On the uniqueness of solutions of
  stochastic differential equations.}, J. Math. Kyoto Univ. \textbf{11} (1971),
  155--167.
  
\bibitem[Zha10]{zhang-10}
X.~Zhang, \emph{Stochastic flows of {SDE}s with irregular coefficients and stochastic
 transport equations}, 
 Bull. Sci. Math.  \textbf{134}  (2010),  no. 4, 340--378.


\bibitem[Zha13]{zhang-degenerate}
X.\ Zhang, \emph{Degenerate irregular {SDE}s with jumps and application to
  integro-differential equations of {F}okker-{P}lanck type}, Electron. J.
  Probab. \textbf{18} (2013), no. 55, 25.

\end{thebibliography}
\end{document}